\documentclass[a4paper,10pt]{article}
\usepackage[T1]{fontenc}
\usepackage[utf8]{inputenc}
\usepackage[english,greek,frenchb,french]{babel}
\usepackage{url,csquotes}
\usepackage[hidelinks,hyperfootnotes=false]{hyperref}
\usepackage{float}
\usepackage{enumitem}
\usepackage{amsfonts,amssymb,enumerate}
\usepackage{amsmath}
\allowdisplaybreaks[1]
\usepackage{amsthm}
\usepackage{graphicx}
\usepackage{bbm}
\usepackage{listings}
\usepackage{titling}
\usepackage{dsfont}
\usepackage{color}

\newcommand{\E}{\mathbb{E}}

\newcommand{\cE}{\mathcal{E}}
    \newcommand{\Prb}{\mathbb{P}}
	
		\newcommand{\cG}{\mathcal{G}}
			\newcommand{\cM}{\mathcal{M}}
\newcommand{\cV}{\mathcal{V}}
\newcommand{\cL}{\mathcal{L}}
\newcommand{\cI}{\mathcal{I}}
\newcommand{\cW}{\mathcal{W}}
\newcommand{\vv}{\overrightarrow{v}}

		\newcommand{\cH}{\mathcal{H}}
		\newcommand{\cP}{\mathcal{P}}
	\newcommand{\cB}{\mathcal{B}}
		\newcommand{\cF}{\mathcal{F}}
		\newcommand{\fF}{\mathfrak{F}}
	\newcommand{\sR}{\mathbb{R}}
	
		\newcommand{\sS}{\mathbb{S}}
	
\DeclareMathOperator{\Animals}{Animals}
		
			\DeclareMathOperator{\disc}{disc}

				\DeclareMathOperator{\card}{card}
\DeclareMathOperator{\e}{e}					
		\DeclareMathOperator{\dive}{div}
		
		\DeclareMathOperator{\cyl}{cyl}
		\DeclareMathOperator{\hull}{hull}
    \newcommand{\sZ}{\mathbb{Z}}
    \newcommand{\sC}{\mathcal{C}}
    \newcommand{\ep}{\varepsilon}

    \newcommand{\ind}{\mathds{1}}
 \theoremstyle{plain}  
\newtheorem{thm}{Theorem}
\newtheorem{prop}{Proposition}[section]
\theoremstyle{plain}
\newtheorem{defn}{Definition}[section]
\newtheorem{lem}{Lemma}
\newtheorem{cor}{Corollary}[section]
\newlength{\separationtitre}
\author{Barbara Dembin\thanks{LPSM UMR 8001, Université Paris Diderot, Sorbonne Paris Cité, CNRS, F-75013 Paris, France}}
\theoremstyle{remark}
\newtheorem{rk}{Remark}[section]
\date{}
\begin{document}
\newpage

 \selectlanguage{english}

\title{Existence of the anchored isoperimetric profile in supercritical bond percolation in dimension two and higher}

\maketitle

\textbf{Abstract}: Let $d\geq 2$. We consider an i.i.d. supercritical bond percolation on $\sZ^d$, every edge is open with a probability $p>p_c(d)$, where $p_c(d)$ denotes the critical point. We condition on the event that $0$ belongs to the infinite cluster $\sC_\infty$ and we consider connected subgraphs of $\sC_\infty$ having at most $n^d$ vertices and containing $0$. Among these subgraphs, we are interested in the ones that minimize the open edge boundary size to volume ratio. These minimizers properly rescaled converge towards a translate of a deterministic shape and their open edge boundary size to volume ratio properly rescaled converges towards a deterministic constant.
\newline



\section{Introduction}
%In this section, we give informal definitions in order to present the background and state our results. More rigorous definitions will be given in section \ref{defs}.

Isoperimetric problems are among the oldest problems in mathematics. They consist in finding sets that maximize the volume given a constraint on the perimeter or equivalently that minimize the perimeter to volume ratio given a constraint on the volume. These problems can be formulated in the anisotropic case. Given a norm $\nu$ on $\sR^ d$ and $S$ a continuous subset of $\sR^d$, we define the tension exerted at a point $x$ in the boundary $\partial S$ of $S$ to be $\nu(n_S(x))n_S(x)$, where $n_S(x)$ is the exterior unit normal vector of $S$ at $x$. The quantity $\nu(n_S(x))$ corresponds to the intensity of the tension that is exerted at $x$. We define the surface energy of $S$ as the integral of the intensity of the surface tension $\nu(n_S(x))$ over the boundary $\partial S$. An anisotropic  isoperimetric problem consists in finding sets that minimize the surface energy to volume ratio given a constraint on the volume. To solve this problem, in \cite{wulff_cluster}, Wulff introduced through the Wulff construction a shape achieving the infimum. This shape is called the Wulff crystal, it corresponds to the unit ball for a norm built upon $\nu$. Later, Taylor proved in \cite{taylor1975} that this shape properly rescaled is the unique minimizer, up to translations and modifications on a null set, of the associated isoperimetric problem.

The study of isoperimetric problems in the discrete setting is more recent. In the continuous setting, we study the perimeter to volume ratio, in the context of graphs, the analogous problem is the study of the size of edge boundary to volume ratio. This can be encoded by the Cheeger constant. For a finite graph $\cG=(V(\cG),E(\cG))$, we define the edge boundary $\partial_\cG A$  of a subset $A$ of $V(\cG)$ as  $$\partial_\cG A =\Big\{\,e=\langle x,y\rangle \in E(\cG):x\in A,y\notin A \,\Big\}\,. $$ We denote by $\partial A$ the edge boundary of $A$ in $(\sZ^d,\E^d)$ and by $|B|$ the cardinal of the finite set $B$. The isoperimetric constant, also called Cheeger constant, is defined as
$$\varphi_\cG=\min\left\{\,\frac{|\partial_\cG A|}{|A|}\,: \, A\subset V(\cG), 0<|A|\leq \frac{|V(\cG)|}{2}\,\right\}\,.$$
This constant was introduced by Cheeger in his thesis \cite{thesis:cheeger}  in order to obtain a lower bound for the smallest eigenvalue of the Laplacian. The isoperimetric constant of a graph gives information on its geometry.

Let $d\geq 2$. We consider an i.i.d. supercritical bond percolation on $\sZ^d$, every edge is open with a probability $p>p_c(d)$, where $p_c(d)$ denotes the critical parameter for this percolation. We know that there exists almost surely a unique infinite open cluster $\sC_\infty$ \cite{Grimmett99}. In this paper, we want to study the geometry of $\sC_\infty$ through its Cheeger constant. However, if we minimize the isoperimetric ratio over all possible subgraphs of $\sC_\infty$ without any constraint on the size, one can show that $\varphi_{\sC_\infty}=0$ almost surely. For that reason, we shall minimize the isoperimetric ratio over all possible subgraphs of $\sC_\infty$ given a constraint on the size. There are several ways to do it. We can for instance study the Cheeger constant of the graph $\sC_n=\sC_\infty \cap [-n,n]^d$ or of the largest connected component \smash{$\widetilde{\sC}_n$} of $\sC_n$ for $n\geq 1$. As we have $\varphi_{\sC_\infty}=0$ almost surely, the isoperimetric constants $\varphi_{\sC_n}$ and $\varphi_{\widetilde{\sC}_n}$ go to $0$ when $n$ goes to infinity. Benjamini and Mossel \cite{benjamini2003mixing}, Mathieu and Remy \cite{mathieu2004isoperimetry}, Rau \cite{Rau2006}, Berger, Biskup, Hoffman and Kozma \cite{berger2008anomalous}, Pete \cite{ECP1390} proved that $\varphi_{\widetilde{\sC}_n}$ is of order $n^{-1}$. Roughly speaking, by analogy with the full lattice, we expect that subgraphs of \smash{$\widetilde{\sC}_n$} that minimize the isoperimetic ratio have an edge boundary size of order $n^{d-1}$ and a size of order $n^d$, this is coherent with the fact that $\varphi_{\widetilde{\sC}_n}$ is of order $n^{-1}$.  This leads Benjamini to conjecture that for $p>p_c(d)$, the limit of $n\varphi_{\widetilde{\sC}_n}$ when $n$ goes to infinity exists and is a positive deterministic constant. 

This conjecture was solved in dimension $2$ by Biskup, Louidor, Procaccia and Rosenthal in \cite{biskup2012isoperimetry} and by Gold in dimension $3$ in \cite{Gold2016}. They worked on a modified Cheeger constant. Instead of considering the open edge boundary of subgraphs within $\sC_n$, they considered the open edge boundary within the whole infinite cluster $\sC_\infty$, this is more natural because $\sC_n$ has been artificially created by restricting $\sC_\infty$ to the box $[-n,n]^d$. They also added a stronger constraint on the size of subgraphs of $\sC_n$ to ensure that minimizers do not touch the boundary of the box $[-n,n]^d$.  Moreover, the subgraphs achieving the minimum, properly rescaled, converge towards a deterministic shape that is the Wulff crystal. Namely, it is the shape solving the continuous anisotropic isoperimetric problem associated with a norm $\beta_p$ corresponding to the surface tension in the percolation setting. The quantity $n\varphi_{\sC_n}$ converges towards the solution of a continuous isoperimetric problem. 

Dealing with the isoperimetric ratio within $\sC_n$ needs to be done with caution. Indeed, we do not want minimizers to be close to the boundary of $\sC_n$ because this boundary does not exist in $\sC_\infty$. There is another way to define the Cheeger constant of $\sC_\infty$, that is more natural in the sense that we do not restrict minimizers to remain in the box $[-n,n]^d$. This is called the anchored isoperimetric profile  $\varphi_n$ and it is defined by:
$$ \varphi_n=\min\left\{\,\frac{|\partial_{\sC_\infty} H|}{|H|}: 0\in H\subset \sC_\infty,\, \text{ H connected, } 0<|H|\leq n^d\,\right\}\,, $$
 where we condition on the event $\{0\in\sC_\infty\}$.
We say that $H$ is a valid subgraph if $0\in H\subset \sC_\infty$, $H$ is connected and $|H|\leq n^d$. We also define
$$\partial ^o H=\Big\{\,e\in\partial H,\text{ $e$ is open}\,\Big\}\,.$$
Note that if $H\subset \sC_\infty$, then 
$\partial_{\sC_\infty} H=\partial^o H$. For each $n$, let $\cG_n$ be the set of the valid subgraphs that achieve the infimum in $\varphi_n$. In this context, a minimizer $G_n\in\cG_n$ can go potentially very far from $0$. The minimizer $G_n$ properly rescaled do not belong anymore to a compact set. This lack of compacity is the main issue to overcome to prove that the limit exists. It was done in dimension $2$ in \cite{biskup2012isoperimetry}, with a specific norm that cannot be extended to higher dimensions. We need to introduce some definitions to be able to define properly a limit shape in dimension $d\geq 2$. In order to build a continuous limit shape, we shall define a continuous analogue of the open edge boundary. In fact, we will see that the open edge boundary may be interpreted in term of a surface tension $\cI$, in the following sense. Given a norm $\tau$ on $\sR^d$ and a subset $E$ of $\sR^d$ having a regular boundary, we define $\cI_\tau(E)$ as 
$$\cI_\tau (E)=\int_{\partial E}\tau(n_E(x))\cH^{d-1}(dx)\,,$$
where $\cH ^{d-1}$ denotes the Hausdorff measure in dimension $d-1$.
The quantity $\cI_\tau(E)$ represents the surface tension of $E$ for the norm $\tau$. At the point $x$,  the tension has intensity $\tau(n_E(x))$ in the direction of the normal unit exterior vector $n_E(x)$. We denote by $\cL ^d$ the $d$-dimensional Lebesgue measure. We can associate with the norm $\tau$ the following isoperimetric problem:
$$\text{minimize $\frac{\cI_\tau(E)}{\cL^d(E)}$ subject to $\cL^d(E)\leq 1$}\,.$$
We use the Wulff construction to build a minimizer for this anisotropic isoperimetric problem. We define the set $\widehat{W}_\tau$ as
$$\widehat{W}_\tau=\bigcap_{v\in\sS^{d-1}}\left\{x\in\sR^d:\, x\cdot v\leq \tau(v)\right\}\,,$$
where $\cdot$ denotes the standard scalar product and $\sS^{d-1}$ is the unit sphere of $\sR^d$. The set $\widehat{W}_\tau$ is a minimizer for the isoperimetric problem associated with $\tau$. We will build in section \ref{sect3} an appropriate norm $\beta_p$ for our problem that will be directly related to the open edge boundary ratio. We define the Wulff crystal $W_p$ as the dilate of $\widehat{W}_{\beta_p}$ such that $\cL^d(W_p)=1/\theta_p$, where $\theta_p=\Prb(0\in\sC_\infty)$. 

In this paper, we adapt the proof of Gold to any dimension $d\geq 2$ to give a self-contained proof of the existence of the limit for the anchored isoperimetric profile. Note that this proof also holds in dimension $2$, it gives an alternative proof of \cite{biskup2012isoperimetry} with a simpler norm.  The aim of this paper is the proof of the two following Theorems. The first theorem asserts the existence of the limit of $n\varphi_n$. 
\begin{thm}\label{thmheart}
Let $d\geq 2$, $p>p_c(d)$ and let $\beta_p$ be the norm that will be properly defined in section \ref{sect3}. Let $W_p$ be the Wulff crystal for this norm, \textit{i.e.}, the dilate of $\widehat{W}_{\beta _p}$ such that $\cL^d(W_p)=1/{\theta_p}$. Then, conditionally on $\{0\in\sC_\infty\}$,
$$\lim_{n\rightarrow \infty} n\varphi_n=\frac{\cI_p(W_p)}{\theta_p\cL^d(W_p)}=\cI_p(W_p)\text{ a.s..}$$
\end{thm}
The second theorem shows that the graphs $G_n$ realizing the minimum converge in probability towards a translate of $W_p$.
\begin{thm}\label{formeGn}
Let $d\geq 2$ and $p>p_c(d)$. Let $\ep>0$. There exists positive constants $C_1$ and $C_2$ depending on $d$, $p$ and $\ep$ such that, for all $n\geq 1$,
$$\Prb\left[\max_{G_n\in\cG_n}\inf_{x\in\sR^d}\frac{1}{n^d}\big|G_n\Delta (n(x+W_p)\cap\sC_\infty)\big|\geq \ep\,\Big|\,0\in\sC_\infty\,\right ]\leq C_1\e^{-C_2n^{1-3/2d}}\,,$$
where $\Delta$ denotes the symmetric difference.
\end{thm}

% The proof of Gold relies crucially on the fact that minimizers $G_n\in \cG_n$ are contained in the cube $[-n,n]^d$ and so $G_n$ properly rescaled belongs to a compact region. For the anchored isoperimetric profile, a minimizer can go potentially very far from $0$, and compacity arguments do not hold anymore. The main difficulty of this work is to cope with this lack of compacity.

To prove Theorem \ref{thmheart}, we first prove a large deviations result from above for $n\varphi_n$ stated in the following Theorem.
\begin{thm}\label{ULD}
Let $d\geq 2$. Let $p>p_c(d)$. For all $\ep>0$, there exist positive constants $C_1$ and $C_2$ depending on $p$, $d$, $\ep$ such that, for all $n\geq 1$,
$$\Prb\left[n\varphi_n \geq (1+\ep) \dfrac{\cI_p(W_p)}{\theta_p(d)\cL^d(W_p)}\,\Big |\,0\in\sC_\infty\right] \leq C_1\exp(-C_2n)\,.$$%Vérifier le n^d
\end{thm}
The proof of Theorem \ref{ULD} is inspired from the proof of Theorem 5.4 in \cite{Gold2016}. We shall build a valid subgraph that has an isoperimetric ratio close to $\varphi_n$. In order to do so, we approximate the Wulff shape $W_p$ from the inside by a convex polytope $P$. We shall build a cutset $\Gamma_n$ that cuts $nP$ from infinity whose number of open edges is close to $n^{d-1}\cI_p(P)$ with high probability.  For each face $F$ of $P$ and $v$ its associated exterior unit normal vector, we consider the cylinder $\cyl(F+\ep v,\ep)$ of basis $F+\ep v$ and of height $\ep>0$. We build $E$ by merging the cutsets from the top to the bottom of minimal capacity of the cylinders $\cyl(F+\ep v,\ep)$. The union of these cutsets is not yet a cutset itself because of the potential holes between these cutsets. We fix this issue by adding extra edges to fill the holes. We control next the number of extra edges we have added. We also need to control the capacity of the cutsets in a cylinder of polyhedral basis. We next build a valid subgraph $H_n\subset \sZ^d$ from $\Gamma_n$ by taking all the vertices of $\sC_\infty \cap nP$ that are connected to $0$ without using edges in $\Gamma_n$. We expect that $|H_n|$ is of order $\theta_p(d)n^d\cL^d(P)$. We can bound from above $|\partial_{\sC_\infty}H_n|$ thanks to the number of open edges in $\Gamma_n$ and so we control its isoperimetric ratio. Finally, we control the upper large deviations for this number of open edges thanks to the upper large deviations for the flow in a cylinder of polyhedral basis. 
The next step is to obtain the large deviations result from below.
\begin{thm}\label{LLD}
Let $d\geq 2$. Let $p>p_c(d)$. For all $\ep>0$, there exist positive constants $C_1$ and $C_2$ depending on $p$, $d$, $\ep$ such that, for all $n\geq 1$,
$$\Prb\left[n\varphi_n \leq (1-\ep) \dfrac{\cI_p(W_p)}{\theta_p(d)\cL^d(W_p)}\,\Big |\,0\in\sC_\infty\right] \leq C_1\exp(-C_2n^{1-3/2d})\,.$$
\end{thm}
\begin{rk} The deviation order in Theorem \ref{LLD} is not optimal due to technical details of the proof. In this work we do not make any attempt to get the proper order of deviation. Our aim is mainly to obtain Theorems \ref{thmheart} and \ref{formeGn}. The study of the large deviations order would be an interesting problem in itself.
\end{rk}
\noindent Theorem \ref{thmheart} follows from Theorem \ref{ULD} and Theorem \ref{LLD} by a straightforward application of the Borel-Cantelli Lemma. Proving the large deviations result from below is the most difficult part of this work.  To be able to compare discrete objects with continuous ones, we shall encode each optimizer $G_n\in\cG_n$ as a measure $\mu_n$ defined as
$$\mu_n=\frac{1}{n^d}\sum_{x\in V(G_n)}\delta _{x/n}\,.$$ We first need to build from a minimizer $G_n$ an appropriate continuous object $P_n$. To do so, we use the same method as in \cite{Gold2016}. The main issue is that the boundary of $G_n$ may be very tangled, we will have to build a smoother boundary of size of order $n^{d-1}$. This will enable us to build a continuous object $P_n$ of finite perimeter such that, with high probability, its associated measure is close to $\mu_n$ in some sense to be specified later.

Let $F$ be a subset of $\sR^d$. We define its associated measure $\nu_F$:
$$\forall E\in\cB(\sR^d),\qquad \nu_F(E)=\theta_p\cL^d(F\cap E)\,.$$
We now define the set $\cW$ of the measures associated with the translates of the Wulff shape as
$$\cW=\big\{\,\nu_{x+W_p}:\,x\in\sR^d\,\big\}\,.$$
Note that $\mu_n$ belongs to $\cM(\sR^d)$, the set of finite measures on $\sR^d$. We cannot use a metric as in \cite{Gold2016} where $\mu_n$ was a measure on $[-1,1]^d$. In fact, we will not use a metric here.
We first show that all the minimizers $G_n\in\cG_n$ are with high probability in a local neighborhood of $\cW$ for a weak topology. This is the key step before proving Theorem \ref{LLD}.
\begin{thm}\label{prel}Let $d\geq 2$ and $p>p_c(d)$.  Let $u:]0,+\infty[\rightarrow ]0,+\infty[$ be a non-decreasing function such that $\lim_{t\rightarrow 0}u(t)=0$. For all $\zeta>0$, there exist positive constants $C_1$ and $C_2$ depending on $d$, $p$, $u$ and $\zeta$ such that for all $n\geq 1$, for any finite set $\fF_n$ of uniformly continuous functions that satisfies: 
$$ \forall f \in\fF_n \quad\|f\|_\infty\leq 1\qquad\text{ and }\qquad\forall x,y\in\sR^d\quad |f(x)-f(y)|\leq u(\|x-y\|_2)\,,$$ we have
$$\Prb\left[\exists G_n\in\cG_n, \,\forall \nu \in \cW, \,\sup_{f\in\fF_n}|\mu_n(f)-\nu(f)|>\zeta \,\Big |0\in\sC_\infty\right]\leq C_1 \e ^{-C_2n^{1-3/2d}}\,.$$
\end{thm}
The main difficulty of this paper lies in the proof of this theorem. In our context, an issue that was not present in \cite{Gold2016} arises. Whereas the support of the measure $\mu_n$ was included in a fixed compact set in \cite{Gold2016}, this is not the case here because we do not constrain $G_n\in\cG_n$ to remain in the box $[-n,n]^d$. To fix this issue, we will use the method developed in \cite{Cerf:StFlour}. We will first localize the set $G_n$ in a finite number of balls of radius of order $n$ up to a set of small fractional volume. We will study $G_n$ only inside these balls, \textit{i.e.}, the intersection of $G_n$ with these balls. The intersection of $G_n$ with the boundary of these balls will create an additional surface tension. However, this surface tension is not related to the open boundary edges of $G_n$ but to the fact that we have cut $G_n$ along these boundaries. Therefore, we should not take this surface tension into account for the isoperimetric constant. In fact, we will cut $G_n$ in such a way to ensure that we do not create too much surface tension, \textit{i.e.}, we will cut in regions where $G_n$ is not concentrated. To conclude, we will link the probability that the measure $\mu_n$ corresponding to $G_n\in\cG_n$ is far from  a weak neighborhood of $\cW$ with the probability that the surface tension of $G_n$ is locally abnormally small.

Finally, to prove Theorem \ref{formeGn}, we exhibit a set $\cF_n$ of uniformly continuous functions such that we can bound from above the symmetric difference $|G_n\Delta (n(x+W_p)\cap\sC_\infty)|$ by $\sup_{f\in\fF_n}|\mu_n(f)-\nu(f)|$ for some $\nu\in\cW$ and then apply the result of Theorem \ref{prel}.

The rest of the paper is organized as follows. In section \ref{defs}, we give some definitions and useful results. We do the construction of the norm $\beta_p$ in section \ref{sect3}. In section \ref{sectionULD}, we prove the upper large deviations in Theorem \ref{ULD}. We build a continuous object $P_n$ from a minimizer $G_n\in\cG_n$ and prove that its associated measure is close in some sense to the measure $\mu_n$ of $G_n$ in section \ref{sectconstruction}. Finally, in section \ref{sectionLLD}, we prove Theorem \ref{prel} that is a preliminary work before proving the lower large deviations Theorem \ref{LLD} and the convergence of $G_n$ properly rescaled towards a limit shape in Theorem \ref{formeGn}.
\section{Some definitions and useful results}\label{defs}
\subsection{Geometric notations}
For $x=(x_1,\dots,x_d)$, we define $$\|x\|_2=\sqrt{\sum_{i=1}^dx_i^2}\,\qquad \text{and}\qquad\|x\|_\infty=\max_{1\leq i\leq d}|x_i|\,.$$
We say that $x,y\in\sZ^d$ are $*$-connected if $\|x-y\|_\infty=1$. We say that $\gamma=(x_0,\dots,x_n)$ is an $*$-path of $\sZ^d$ if for any $0\leq i\leq n-1$, the points $x_i$ and $x_{i+1}$ belong to $\sZ^d$ and are $*$-connected. We say that $\Gamma$ is $*$-connected or a lattice animal if any $x,y\in\Gamma$ are connected by an $*$-path in $\Gamma$. We denote by $\Animals_x$ the set of lattice animals containing the point $x\in\sZ^d$.
%$$\Animals_x=\left\{\Gamma, \text{  $\Gamma$ is an $*$-connected set of cubes containing $x$}\right\}$$
\begin{lem}\label{tailleanimal}[Kesten \cite{Kesten:StFlour}, p82 or Grimmett \cite{Grimmett99}, p85]
Let $x\in\sZ^d$. For all positive integer $l$,
$$|\{\Gamma\in\Animals_x,\,|\Gamma|=l\}|\leq (7^d)^l\,.$$
\end{lem}

Let $S\subset \sR^d$ and $r>0$, we define $ d_2(x,S)=\inf_{y\in S}\|x-y\|_2$ and $\cV(S,r)$ the open $r$-neighborhood of $S$ by
$$\cV(S,r)=\Big\{\,x\in\sR^d \,:\, d_2(x,S)< r\,\Big\}\,.$$
Let $x\in \sR^d$, $r>0$ and a unit vector $v$. We denote by $B(x,r)$ the closed ball of radius $r$ centered at $x$, by $\disc(x,r,v)$ the closed disc centered at $x$ of radius $r$ normal to $v$, and by $B^+(x,r,v)$ (respectively $B^- (x,r,v)$) the upper (resp. lower) half part of $B(x,r)$ along the direction of $v$, \textit{i.e.},
$$B^+ (x,r,v)=\Big\{\,y\in B(x,r)\,:\, (y-x)\cdot v\geq 0\,\Big\},$$
and
$$B^- (x,r,v)=\Big\{\,y\in B(x,r)\,:\, (y-x)\cdot v\leq 0\,\Big\}\,.$$  We denote by $\alpha_d$ the $\cL^d$ measure of a unit ball in $\sR^d$. We denote by $\cH^{d-1}$ the Hausdorff measure in dimension $d-1$. In particular, the $\cH^{d-1}$ measure of a $d-1$ dimensional unit disc in $\sR^d$ is equal to $\alpha_{d-1}$.
 Let $A$ be a non-degenerate hyperrectangle, \textit{i.e.}, a rectangle of dimension $d-1$ in $\sR^d$. Let $\vv$ be one of the two unit vectors normal to $A$. Let $h>0$, we denote by $\cyl(A,h)$ the cylinder of basis $A$ and height $h$ defined by 
$$\cyl(A,h)=\Big\{\,x+t\vv\, : \,  x\in A,\, t\in[-h,h]\,\Big\}\,.$$
The dependence on $\vv$ is implicit in the notation $\cyl(A,h)$.
Note that these definitions of cylinder may be extended in the case where $A$ is of linear dimension $d-1$, \textit{i.e.}, $A$ is included in an hyperplane of $\sR^d$, which is the affine span of $A$.

\subsection{Sets of finite perimeter and surface energy}
The perimeter of a Borel set $S$ of $\sR^d$ in an open set $O$ is defined as 
$$\cP(S,O)=\sup \left\{\int _S \dive f(x)d\cL^d(x):\, f\in C^\infty_c(O,B(0,1))\right\}\,, $$
where $C^\infty_c(O,B(0,1))$ is the set of the functions of class $\sC^\infty$ from $\sR^d$ to $B(0,1)$ having a compact support included in $O$, and $\dive$ is the usual divergence operator. The perimeter $\cP(S)$ of $S$ is defined as $\cP(S,\sR^d)$. The topological boundary of $S$ is denoted by $\partial S$.  The reduced boundary $\partial ^* S$ of $S$ is a subset of $\partial S$ such that, at each point $x$ of $ \partial ^* S$, it is possible to define a normal vector $n_S(x)$ to $S$ in a measure-theoretic sense, and moreover $\cP(S)=\cH^{d-1}(\partial^*S)$. 
Let $\nu$ be a norm on $\sR^d$. We define its associated Wulff crystal $\cW_\nu$ as
$$\cW_\nu=\Big\{\,x\in\sR^d\,:\, \forall y, \,\, y\cdot x\leq \nu(y)\,\Big\}.$$ 
With the help of the Wulff crystal, we define the surface energy of a general set.
\begin{defn} The surface energy $\cI(S,O)$ of a Borel set $S$ of $\sR^d$ in an open set $O$ is defined as 
$$\cI(S,O)=\sup \left\{\int _S \dive f(x)d\cL^d(x):\, f\in C^1_c(O,\cW_\nu)\right\} .$$
\end{defn}
\noindent We will note simply $\cI(S)=\cI(S,\sR^d)$.
\begin{prop}[Proposition 14.3 in \cite{Cerf:StFlour}]\label{propI} The surface energy $\cI(S,O)$ of a Borel set $S$ of $\sR^d$ of finite perimeter in an open set $O$ is equal to
$$\cI(S,O)=\int _{\partial^* S \cap O} \nu(n_S(x))d\cH^{d-1}(x)\,.$$
\end{prop}

We recall the two following fundamental results.
\begin{prop}[Isoperimetric inequality]\label{isop}
There exist two positive constants $b_{iso}$, $c_{iso}$ which depend only on the dimension $d$, such that for any Cacciopoli set $E$, any ball $B(x,r)\subset \sR^d$,
$$\min\left(\cL^d(E\cap B(x,r)),\cL^d((\sR^d\setminus E)\cap B(x,r))\right)\leq b_{iso}\cP(E,\mathring{B}(x,r))^{d/d-1},$$
$$\min\left(\cL^d(E),\cL^d(\sR^d\setminus E)\right)\leq c_{iso} \cP(E)^{d/d-1}\,.$$
\end{prop}

\subsection{Approximation by convex polytopes}
We recall here an important result, which allows to approximate adequately a set of finite perimeter by a convex polytope.
\begin{defn}[Convex polytope]
Let $P\subset \sR^d$. We say that $P$ is a convex polytope  if there exist $v_1,\dots,v_m$ unit vectors and $\varphi_1,\dots,\varphi_m$ real numbers such that
$$P=\bigcap_{1\leq i\leq m}\Big\{\,x\in \sR^d  : \, x\cdot v_i\leq \varphi_i\,\Big\}\,.$$
We denote by $F_i$ the face of $P$ associated with $v_i$, \textit{i.e.}, $$F_i= P\cap\Big\{\,x\in \sR^d  : \, x\cdot v_i= \varphi_i\,\Big\}\,.$$
\end{defn}
\noindent Any convex subset can be approximated from the outside and from the inside by a convex polytope with almost the same surface energy.
\begin{lem}\label{ApproP} Let $A$ be a bounded convex set. For each $\ep>0$, there exist convex polytopes $P$ and $Q$ such that $P\subset A\subset Q$ and $\cI(Q)-\ep\leq \cI(A)\leq \cI(P)+\ep$.  
\end{lem}
\begin{proof}
Let $A$ be a bounded convex set. Let $\ep>0$. Let $(x_k)_{k\geq 1}$ be a dense family in $\partial A$. For $n\geq 1$, we define $P_n$ as the convex hull of $x_1,\dots,x_n$, \textit{i.e.}, the smallest convex that contains the points $x_1,\dots,x_n$. 
As $A$ is convex, we have $P_n\subset A$ and $P_n$ converges towards $A$ when $n$ goes to infinity for the $\cL ^1$ topology.
The functional $\cI$ is lower semi-continuous, thus
$$\cI(A)\leq\liminf_{n\rightarrow \infty}\cI(P_n)\,,$$
so there exists $n$ large enough such that 
$$\cI(A)\leq \cI(P_n)+\ep\,$$
and we take $P=P_n$.
The existence of $Q$ was shown in Lemma 5.1. in \cite{cerf2000} for the Wulff shape. The proof may be easily adapted to a general convex bounded set $A$. 
%For any $a\in \partial^* A$ and any sequence $(a_n)_{n\geq 1}$ converging to $a$ with $a_n\in \partial^* Q_n$ for all $n\geq 1$. We obtain that the normal unit vector $n_{Q_n}(a_n)$ converges towards $n_A(a)$ when $n$ goes to infinity. As $\cH^{d-1}(\partial A \setminus \partial^* A)=0$, we obtain 
%$$\lim_{n\rightarrow \infty} \cI(Q_n)=\cI(A)$$
%%%ref xxx lachand robert
%and so for $n$ large enough, $\cI(Q_n)-\ep\leq \cI(A)$. As $P_n$ and $Q_n$ are convex polytopes, we conclude by setting $P=P_n$ and $Q=Q_n$ for $n$ large enough.
\end{proof}

%We will also need to ensure that the volume $P_n$ is lower bounded with high probability.
%\begin{lem}
%Let $d\geq 2$ and $p>p_c(d)$. There exist positive constants $c_1,\,c_2$ and $\eta_2$ depending only on $d$ and $p$ such that, for all $n\geq 1$,
%$$\Prb[\exists G_n\in\cG_n|\, \cL^d(P_n)\leq \eta_2]\leq c_1\exp(-c_2n^{\kappa(d)})\,.$$
%\end{lem}
%\begin{proof}On the event $$\left\{\min_{G_n\in\cG_n}|G_n|>\eta_1n^d,\, \left|\theta_p\cL^d (P_n)- \dfrac{|G_n|}{n^d}\right|<\eta_1/(2\theta_p)\right\}$$
%we have
%\begin{align*}
%\cL^d(P_n)\geq \dfrac{|G_n|}{\theta_p n^d}-\dfrac{\eta_1}{2\theta_p}\geq \dfrac{\eta_1}{2\theta_p}\,.
%\end{align*}
%So by setting $\eta_2= \eta_1/(2\theta_p)$,we obtain
%\begin{align*}
%\Prb[\exists G_n\in\cG_n|\, \cL^d(P_n)\leq \eta_2]&\leq \Prb[\exists G_n\in\cG_n|\, |G_n|\leq \eta_1n^d]+\Prb[\fd(\mu_n,\nu_n)\geq \eta_2]\\
%&\leq C_1\exp(-C_2n^{\kappa(d)})
%\end{align*}
%where we use the control in Lemma \ref{nottoobig} and Proposition \ref{propcontiguity}.
%\end{proof}

\section{Construction of the norm}\label{sect3}

Minimizing the open edge boundary is the analogue of minimizing a surface tension in the continuous setting. We shall build a norm $\beta_p$ that represents the tension that is exerted on the surface, \textit{i.e.}, any point $x$ in a surface $S$ having $n_S(x)$ as a normal unit exterior vector has a tension $\beta_p(n_S(x))n_S(x)$ that exerts at the point $x$. To build this norm, let us consider $G_n\in\cG_n$. We zoom on the boundary of $G_n$, we look at what happens in a small but macroscopic cube centered at a point $x$ in the boundary $\partial G_n$ (see figure \ref{fig1}). The cube is located in such a way that its bottom intersects $G_n$ and its top intersects $\sZ^d\setminus G_n$, and it is rotated so that its normal vector coincides with the normal exterior vector at the point $x$. As this cube is small, the portion of $G_n$ in that cube does not affect much $|G_n|$, the total volume of $G_n$. Thus, if one would like to minimize the open edges to volume ratio, one needs to minimize the number of open edges of $\partial G_n$ in that cube. This problem is equivalent to finding a set of edges that separates the top from the bottom of the cube with a minimal number of open edges. %By the max-flow min-cut theorem, this problem is also known as finding the maximal flow. Maximal flow in cylinders have been intensively studied.  
\begin{figure}[H]
\def\svgwidth{0.5\textwidth}
\begin{center}
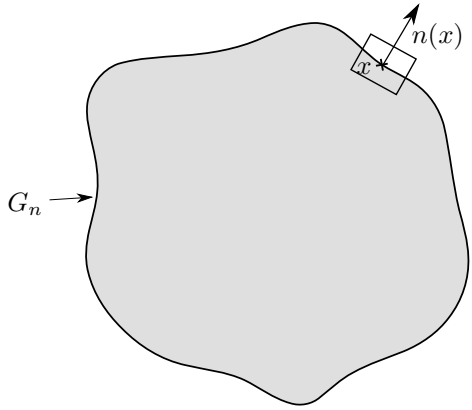
\caption[fig2]{\label{fig1}A small box on the boundary $\partial G_n$ of a minimizer $G_n\in\cG_n$}
\end{center}
\end{figure}
Let us give now a more precise definition of the norm $\beta_p$.
We consider a bond percolation on $\sZ^d$ of parameter $p>p_c(d)$ with $d\geq 2$. We introduce many notations used for instance in  \cite{Rossignol2010} concerning flows through cylinders.
Let $A$ be a non-degenerate hyperrectangle, \textit{i.e.}, a rectangle of dimension $d-1$ in $\sR^d$. Let $\vv$ be one of the two unit vectors normal to $A$. Let $h>0$, we denote by $\cyl(A,h)$ the cylinder of basis $A$ and height $2h$ defined by 
$$\cyl(A,h)=\Big\{\,x+t\vv\,: \,  x\in A,\, t\in[-h,h]\,\Big\}\,.$$
The set $\cyl(A,h)\setminus A$ has two connected components, denoted by $C_1(A,h)$ and $C_2(A,h)$. For $i=1,2$, we denote by $C'_i(A,h)$ the discrete boundary of $C_i(A,h)$ defined by 
$$C'_i(A,h)=\Big\{\,x\in\sZ^d\cap C_i(A,h)\,:\,
\exists y \notin \cyl(A,h),\, \langle x,y \rangle\in\E^d  \,\Big\}\,.$$

We say that a set of edges $E$ cuts $C'_1(A,h)$ from $C'_2(A,h)$ in $\cyl(A,h)$ if any path $\gamma$ from $C'_1(A,h)$ to $C'_2(A,h)$ in $\cyl(A,h)$ contains at least one edge of $E$. We call such a set a cutset. For any set of edges $E$, we denote by $|E|_o$ the number of open edges in $E$. We shall call it the capacity of $E$. We define
$$\tau_p(A,h)=\min\Big\{\, |E|_o: \,\text{ $E$ cuts $C'_1(A,h)$ from $C'_2(A,h)$ in $\cyl(A,h)$}\,\Big\}\,.$$
Note that this is a random quantity as $|E|_o$ is random, and that the cutsets in this definition are pinned near the boundary of $A$. Finding cutsets of minimal capacity is equivalent to the study of maximal flows, see \cite{Bollobas}. To each edge $e$, we can associate the random variable $t(e)=\ind_{e\text{ is open}}$. In the study of maximal flows, we interpret each $t(e)$ as the capacity of the edge $e$, \textit{i.e.}, the maximal amount of water that can flow through $e$ per unit of time. We are interested in the maximal amount of water that can flow through the cylinder given the constraint on the capacity. We refer to \cite{flowconstant} for a rigorous definition of maximal flows. In the following, we will use the term flow to speak about the quantity $\tau_p$. 
The following proposition is a corollary of Proposition 3.5 in  \cite{Rossignol2010}, it enables us to give a rigorous definition of the norm $\beta_p$.
\begin{prop}[Definition of the norm $\beta_p$]\label{normbeta}
Let $d\geq 2$, $p>p_c(d)$, $A$ be a non-degenerate hyperrectangle and $\vv$ one of the two unit vectors normal to $A$. Let $h$ an height function such that $\lim_{n\rightarrow \infty }h(n)=\infty$. The limit 
$$\beta_p(\vv)=\lim_{n\rightarrow \infty } \dfrac{\E[\tau_p(nA,h(n))]}{\cH^{d-1}(nA)}$$
exists and is finite. Moreover, this limit is independent of $A$ and $h$ and $\beta_p$ is a norm.%XXXpas évident
\end{prop}
\noindent The norm $\beta_p$ is called the flow constant. Roughly speaking, $\beta_p(\vv)$ corresponds to the expected maximal amount of water that can flow in the direction $\vv$ on average. Actually, we can obtain a stronger convergence. A straightforward application of Theorem 3.8 in \cite{Rossignol2010} gives the existence of the following almost sure limit:
 $$\lim_{n\rightarrow \infty } \dfrac{\tau_p(nA,h(n))}{\cH^{d-1}(nA)}=\beta_p(\vv)\,.$$
We define $$\beta_{min}=\inf_{\vv\in \sS^{d-1}}\beta_p(\vv)\,,\qquad\beta_{max}=\sup_{\vv\in \sS^{d-1}}\beta_p(\vv)\,.$$
As $\beta_p$ is a norm on $\sR^d$, we have $\beta_{min}>0$ and $\beta_{max}<\infty$.
 We will need the following upper large deviations result which is a straightforward application of Theorem 4 in \cite{TheretUpperTau14}.
 
\begin{thm}\label{upperlargedeviationcyl}
Let $d\geq 2$ and $p>p_c(d)$. For every unit vector $\vv$, for every non-degenerate hyperrectangle $A$ normal to $\vv$, for every  $h>0$ and for every  $\lambda>\beta_p(\vv)$, there exist $C_1$ and $C_2$ depending only on $\lambda$ and $G$, such that, for all $n\geq 0$,
$$\Prb\left[ \tau_p(nA,hn)\geq \lambda \cH^{d-1}(A)n^{d-1}\right]\leq C_1\exp(-C_2h n^{d})\,.$$
\end{thm} 
To ease the reading and lighten the notations, the value of the constants may change from appearance to appearance.

%\begin{rk}This theorem was initially shown for $\phi_p$ that corresponds to another flow than $\tau_p$ such that $\phi_p\geq \tau_p$.
%\end{rk}
\section{Upper large deviations}\label{sectionULD}

  \subsection{The case of a cylinder}

The aim of this section is to prove Theorem \ref{ULD}.
A convex polytope of dimension $d-1$ is a convex polytope $F$ which is contained in an hyperplane of $\sR^d$ and such that $\cH^{d-1}(F)>0$. We have the following Lemma.

\begin{lem}\label{lem2}Let $p>p_c(d)$. Let $F$ be a convex polytope of dimension $d-1$. Let $v$ be a unit vector normal to $F$. There exist positive real numbers $C_1$ and $C_2$ depending on $F$, $p$ and $d$ such that for all $n\geq 1$, for all $\lambda>\beta_p(v)\cH^{d-1}(F)$, for all $h>0$
$$\Prb[\tau_p(nF,nh)\geq \lambda n^{d-1}]\leq C_1\exp(-C_2 h n^{d})\,.$$
\end{lem}

\begin{proof} Let $p>p_c(d)$. Let $F$ be a convex polytope of dimension $d-1$ and $v$ a unit vector normal to $F$. We shall cover $F$ by a finite family of hypersquares and control the probability that the flow is abnormally big in $\cyl(nF,nh)$ by the probability that the flow is abnormally big in one of the cylinders of square basis. Let $\lambda>\beta_p(v)\cH^{d-1}(F)$.
Let $\kappa>0$ be a real number that we will choose later. We denote by $S(\kappa)$ an hypersquare of dimension $d-1$ of side length $\kappa$ and normal to $v$. We want to cover the following region of $F$ by hypersquares isometric to $S(\kappa)$:
$$D(\kappa,F)=\Big\{\,x\in F\,: \, d(x,\partial F)>2\sqrt{d}\kappa\,\Big\}\,.$$
There exists a finite family $(S_i)_{i\in I}$ of closed hypersquares isometric to $S(\kappa)$ included in $F$ having pairwise disjoint interiors, such that $D(\kappa,F )\subset\cup_{i\in I} S_i$ (see figure \ref{fig2}). Moreover, there exists a constant $c_d$ depending only on the dimension $d$ such that
\begin{align}\label{jsap4}
\cH^{d-1}\big(F\setminus D(\kappa,F)\big)\leq c_d\cH^{d-2}(\partial F) \,\kappa\,.
\end{align}
We have then
\begin{align}\label{controleI}
|I|\leq \frac{\cH ^{d-1}(F)}{\cH ^{d-1}(S(\kappa))}\,.
\end{align}
\begin{figure}[H]
\def\svgwidth{0.5\textwidth}
\begin{center}
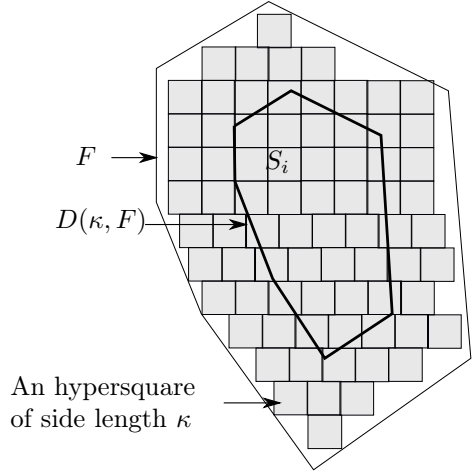
\caption[fig2]{\label{fig2}Covering $P$ with hypersquares}
\end{center}
\end{figure} 
\noindent Let $h>0$.
We would like to build a cutset between $C'_1(nF,nh)$ and $C'_2(nF,nh)$ out of minimal cutsets for the flows  $\tau_p(nS_i,nh)$, $i\in I$. Note that a cutset that achieves the infimum defining $\tau_p(nS_i,nh)$ is pinned near the boundary $\partial nS_i$. However, if we pick up two hypersquares $S_i$ and $S_j$ that share a common side, due to the discretization, their corresponding minimal cutsets for the flow $\tau_p$ do not necessarily have the same trace on the common face of the associated cylinders $\cyl(nS_i,nh)$ and $\cyl(nS_j,nh)$. We shall fix this problem by adding extra edges around the boundaries of the hypersquares $\partial S_i$ in order to glue properly the cutsets. We will need also to add extra edges around $n(F\setminus D(\kappa,F))$ in order to build a cutset between $C'_1(nF,nh)$ and $C'_2(nF,nh)$.
For $i\in I$, let $E_i$ be a minimal cutset for $\tau_p(nS_i,nh)$, \textit{i.e.}, $E_i\subset \E^d$ cuts $C'_1(nS_i,nh)$ from $C'_2(nS_i,nh)$ in $\cyl(nS_i,nh)$ and $|E_i|_o=\tau_p(nS_i,nh)$. 
We fix $\zeta=4d$. Let  $E_0$ be the set of edges of $\E^d$ included in $\cE_{0}$, where we define 
$$\cE_0=\Big\{\,x\in\sR^d\,:\,d\Big(x,nF\setminus \bigcup_{i\in I}nS_i\Big)\leq \zeta\,\Big\}\cup\bigcup_{i\in I}\Big\{\,x\in\sR^d\,:\,d(x,\partial n S_i)\leq \zeta\,\Big\}\,.$$ 
The set of edges $E_0\cup\bigcup_{i\in I} E_i$ separates $C'_1(nF,nh)$ from $C'_2(nF,nh)$ in the cylinder $\cyl(nF,nh)$, therefore,
\begin{align}\label{eq31}
\tau_p(nF,nh)\leq |E_0|_o+\sum_{i\in I} |E_i|_o\leq \card(E_0)+\sum_{i\in I}\tau_p(nS_i,nh)\,.
\end{align}
There exists a constant $c'_d$ depending only on $d$ such that, using inequalities \eqref{jsap4} and \eqref{controleI},
\begin{align*}
\card(E_0)&\leq c'_d\left(\kappa n^{d-1}\cH^{d-2}(\partial F)+|I|\cH^{d-2}(\partial S(\kappa))n^{d-2}\right)\\
&\leq c'_d\left(\kappa n^{d-1}\cH^{d-2}(\partial F)+\dfrac{\cH^{d-1}(F)}{\cH^{d-1}(S(\kappa))}\cH^{d-2}(\partial S(\kappa))n^{d-2}\right)\\
&\leq c'_d\left(\kappa n^{d-1}\cH^{d-2}(\partial F)+\dfrac{\cH^{d-1}(F)}{\kappa}n^{d-2}\right)\,.
\end{align*}
Thus, for $n$ large enough, 
\begin{align}\label{eq32}
\card(E_0) \leq 2c'_d\,\kappa\,\cH^{d-2}(\partial F) n^{d-1}\,.
\end{align}
There exists $s>0$ such that $\lambda>(1+s)\beta_p(v)\cH ^{d-1}(F)$. We choose $\kappa$ small enough such that 
\begin{align}\label{eq3333}
2c'_d\kappa\cH^{d-2}(\partial F) <\frac{s}{2}\beta_{min}\cH ^{d-1}(F)\, .
\end{align}
Inequalities \eqref{eq32} and \eqref{eq3333} yield that 
\begin{align}\label{eq333}
\card(E_0)\leq \frac{s}{2}\beta_p(v)n^{d-1}\cH^{d-1}(F)\,.
\end{align} 
Thanks to inequality \eqref{eq333}, we obtain
\begin{align}\label{eqb1}
\Prb&[\tau_p(nF,nh)\geq \lambda n^{d-1}]\nonumber\\
&\hspace{1cm}\leq\Prb\left[\card(E_0)+\sum_{i\in I}\tau_p(nS_i,nh)\geq (1+s)\beta_p(v)\cH ^{d-1}(F)n^{d-1}\right]\nonumber\\
&\hspace{1cm}\leq \sum_{i\in I} \Prb[\tau_p(nS_i,nh)\geq (1+s/2)\beta_p(v)\cH ^{d-1}(S_i)n ^{d-1}]\,.
\end{align} 
Thanks to Theorem \ref{upperlargedeviationcyl}, there exist positive real numbers $C_1$, $C_2$ such that, for all $i\in I$,
\begin{align}\label{eqb22}
\Prb[\tau_p(nS_i,nh)\geq (1+s/2)\beta_p(v)\cH ^{d-1}(S_i)n ^{d-1}]\leq C_1\exp(-C_2hn^d)\,.
\end{align}
By combining inequalities \eqref{eqb1}  and \eqref{eqb22}, we obtain
\begin{align*}
\Prb[\tau_p(nF,nh)\geq \lambda n^{d-1}]\leq |I|C_1\exp(-C_2hn^d)\,,
\end{align*}
and the result follows.
\end{proof} 
We can now proceed to the proof of Theorem \ref{ULD}. 
\begin{proof}[Proof of Theorem \ref{ULD}]
Let $\ep>0$ and $\ep'>0$. By Lemma \ref{ApproP}, there exists a convex polytope $P$ such that $P\subset W_p$, $\cI_p(P)\leq (1+\ep')\cI_p(W_p)$ and $\cL^d(P)\geq (1-\ep')\cL^d (W_p)$. Up to multiplying $P$ by a constant $\alpha<1$ close to $1$, we can assume without loss of generality that $\cL^d(P)<\cL^d(W_p)$. We have, for small enough $\ep'$ (depending on $\ep$),
\begin{align}\label{bla}
&\Prb\left[n\varphi_n \geq (1+\ep) \dfrac{\cI_p(W_p)}{\theta_p(d)\cL^d(W_p)}\,\Big |0\in\sC_\infty\right]\nonumber\\
&\hspace{3cm}\leq\Prb\left[n\varphi_n \geq (1+\ep/2)\left(\dfrac{1+\ep'}{1-\ep'}\right) \dfrac{\cI_p(W_p)}{\theta_p(d)\cL^d(W_p)}\,\Big |\,0\in\sC_\infty\right]\nonumber\\
&\hspace{3cm}\leq \Prb\left[n\varphi_n \geq (1+\ep/2) \dfrac{\cI_p(P)}{\theta_p(d)\cL^d(P)}\,\Big |\,0\in\sC_\infty\right]\,.
\end{align}
Let us denote by $F_1,\dots, F_m$ the faces of $P$ and let $v_1,\dots,v_m$ be the associated exterior unit vectors. Let $\delta>0$. For $i\in\{1,\dots,m\}$, we define $$C_i=\cyl(F_i+\delta v_i,\delta)\,.$$ All the $C_i$ are of disjoint interiors because $P$ is convex. Indeed, assume there exists \smash{$z\in \smash{\mathring{C_i}\cap \mathring{C_j}}$} for some $i\neq j$. Then there exist unique $x\in F_i$, $y\in F_j$ and $h,h'<2\delta$ such that $z=x+hv_i=y+h'v_j$. The points $x$ and $y$ correspond to the orthogonal projection of $z$ on $P$. As $P$ is convex, the orthogonal projection on $P$ is unique and so $x=y=z$. This contradicts the fact that $z$ belongs to the interior of $C_i$.
We now aim to build a cutset that cuts $nP$ from infinity out of cutsets of minimal capacities for $\tau_p(n(F_i+\delta v_i), n\delta)$, $i\in\{1,\dots,m\}$. The union of these cutsets is not enough to form a cutset from $nP$ to infinity because there are holes between these cutsets. We shall add edges around the boundaries $\partial (n(F_i+\delta v_i))$ to close these holes  (see figure \ref{fig3}). As the distance between two adjacent boundaries $\partial(n (F_i+\delta v_i))$ decreases with $\delta$, by taking $\delta$ small enough, the size of the bridges and so their capacities are not too big. We recall that the capacity of a set, namely the number of open edges in the set, may be bounded from above by its size. Next, we control the maximal flow through the cylinders or equivalently the capacity of minimal cutsets in the cylinders with the help of Lemma \ref{lem2}. 
 
For $i\in \{1,\dots,m\}$, let $E'_i$ be a minimal cutset for $\tau_p(n(F_i+\delta v_i), n\delta)$, \textit{i.e.}, $E'_i$ cuts $C'_1(n(F_i+\delta v_i),\delta)$ from $C'_2(n(F_i+\delta v_i),\delta)$ and $|E'_i|_o=\tau_p(n(F_i+\delta v_i),\delta n)$. We shall add edges to control the space between $E'_i$ and the boundary $\partial (n(F_i+\delta v_i))$. Let $\zeta=4d$.
Let $i,j\in\{1,\dots,m\}$ such that $F_i$ and $F_j$ share a common side. We define
$$\cM(i,j)=\left(\cV(nF_i\cap n F_j, n\delta +\zeta)\setminus \cV(nF_i\cap n F_j, n\delta -\zeta)\right)\cap (nP)^c\,.$$
\noindent Let $M_{i,j}$ denote the set of the edges in $\E_n^d$ included in $\cM(i,j)$ (see figure \ref{fig3}).
There exists a constant $c'_d$ depending only on the dimension $d$ such that for all $i,j\in\{1,\dots,m\}$ such that $F_i$ and $F_j$ share a common side,
\begin{align}\label{contM}
\card(M_{i,j})\leq c_d \delta^{d-1} n^{d-1}\,.
\end{align} 
We set $$M=\bigcup_{i,j}M_{i,j}\,,$$ where the union is over $i,j\in\{1,\dots,m\}$ such that $i\neq j$ and $F_i$, $F_j$ share a common side. The set $\Gamma_n=M\cup \bigcup_{i=1}^m E'_i$ cuts $nP$ from infinity. We define $H_n$ to be the set of the vertices connected to $0$ by open paths which do not use an edge of $\Gamma_n$, \textit{i.e.},
$$H_n=\Big\{\,x\in \sZ^d,\, \text{$x$ is connected to $0$ with open edges in $\E^d \setminus \Gamma_n$}\,\Big\}\,.$$
\begin{figure}[H]
\def\svgwidth{0.9\textwidth}
\begin{center}
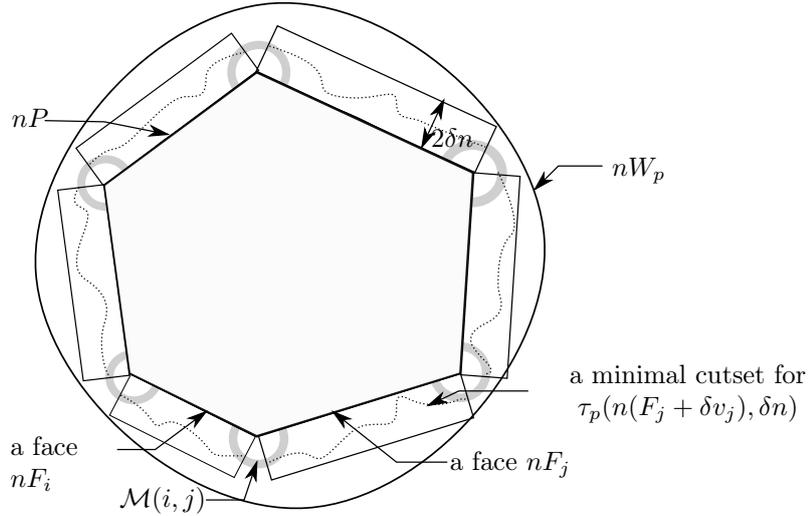
\caption[fig3]{\label{fig3}Construction of a cutset $\Gamma_n$ from $nP$ to infinity}
\end{center}
\end{figure}
\noindent By definition, the set $H_n$ is connected.
As we condition on the event $\{0\in \sC_\infty\}$, the set $H_n$ is a subgraph of $\sC_\infty$.
As $P$ is a polytope, $$\cI_p(P)=\sum_{i=1}^m \beta_p(v_i)\cH^{d-1}(F_i)\,.$$
Moreover, we have $$|\partial_{\sC_\infty} H_n|= |\partial ^o H_n|\leq  | \Gamma_n|_o\,,$$
where the last inequality comes from the fact that, by construction of $H_n$, if $e\in \partial H_n \setminus \Gamma_n$, then $e$ is necessarily closed. Using \eqref{contM}, we have
\begin{align}\label{eq44}
| \Gamma_n|_o&\leq \card(M)+\sum_{i=1}^m |E'_i|_o\nonumber\\
&\leq c_d m ^2\delta ^{d-1}n ^{d-1} +\sum_{i=1}^m\tau_p\big(n(F_i+\delta v_i),\delta n\big)\,.
\end{align}
We choose $\delta$ small enough so that 
\begin{align}\label{eq45}
m^2 c_d\delta^{d-1} <\delta\cI_p(P)/2\qquad \text{and}\qquad \cL^d(\cV(\partial P ,3\delta))\leq \delta \cL^d(P)\, .
\end{align}
Let us now estimate the probability that $|\Gamma_n|_o$ is abnormally big. Using inequalities \eqref{eq44} and \eqref{eq45}, we get 
\begin{align}\label{blab1}
&\Prb[|\Gamma_n|_o\geq (1+\delta)\cI_p(P) n^{d-1}\,|\, 0\in\sC_\infty]\nonumber\\
&\hspace{0.2cm}\leq\frac{1}{\theta_p}\Prb\Big[\card(M)+\sum_{i=1}^m\tau_p(n(F_i+\delta v_i),\delta n)\geq (1+\delta)\sum_{i=1}^m \beta_p(v_i)\cH^{d-1}(F_i)n^{d-1}\Big]\nonumber\\
&\hspace{0.2cm}\leq\frac{1}{\theta_p}\Prb\left[\sum_{i=1}^m\tau_p(n(F_i+\delta v_i),\delta n)\geq (1+\delta/2)\sum_{i=1}^m \beta_p(v_i)\cH^{d-1}(F_i)n^{d-1}\right]\nonumber\\
&\hspace{0.2cm}\leq \frac{1}{\theta_p}\sum_{i=1}^m \Prb[\tau_p(n(F_i+\delta v_i),\delta n)\geq (1+\delta/2)\beta_p(v_i)\cH ^{d-1}(F_i)n ^{d-1}]\,.
\end{align}
By Lemma \ref{lem2}, there exist positive constants $C_1$, $C_2$ depending on $d$, $p$, $P$ and $\delta$ such that, for all $1\leq i\leq m$,
\begin{align}\label{blab2}
\Prb[\tau_p(n(F_i+\delta v_i),\delta n)\geq (1+\delta/2)\beta_p(v_i)\cH ^{d-1}(F_i)n ^{d-1}]\leq C_1\exp(-C_2\delta n^{d})\,.
\end{align}
Finally, combining inequalities \eqref{blab1} and \eqref{blab2}, we obtain
\begin{align}\label{eq444}
\Prb[|\Gamma_n|_o\geq (1+\delta)\cI_p(P) n^{d-1}]\leq\frac{mC_1}{\theta_p} \exp(-C_2\delta n^{d})\,.
\end{align}

We shall now estimate the number of vertices in $H_n$ in order to check that $H_n$ is a valid subgraph. For that purpose, we use a renormalization argument. Let $k>0$. We partition $\sR^d$ into disjoint cubes of side length $1/k$. We define $B'_j$ as the union of $B_j$ and all its $3^d-1$ $*$-neighbors (the cubes $B$ having at least one vertex at $L^1$ distance less than $1$ from $B_j$). We consider $B_1,\dots,B_{l_1}$ the cubes such that $B'_1,\dots,B'_{l_1}$ are contained in $P\setminus \cV(\partial P, 2\delta)$ and $B_{l_1+1},\dots,B_{l_2}$ the cubes such that $B'_{l_1+1},\dots,B'_{l_2}$  intersect $\cV(\partial P, 2\delta)$.
We can choose $k$ large enough such that 
\begin{align}\label{eqbis}
\cL^d\left( \bigcup_{i=l_1+1}^{l_2}B_i\right)\leq  \cL^d(\cV(\partial P ,3\delta))\leq \delta \cL^d(P)\,.
\end{align}
We say that a cube $B_j$ is good if the following event $\cE_n^{(j)}$ occurs:
\begin{itemize}
\item There exists a unique open cluster of diameter larger than $n/k$ in $nB'_j$.
\item  We have $\dfrac{|\sC_\infty\cap nB_j|}{\cL^d(nB_j)}\in (\theta_p-\delta,\theta_p+\delta)\,.$
\end{itemize}
There exist positive constants $C_1$ and $C_2$ depending on $d$, $p$, $k$ and $\delta$ such that
\begin{align}\label{contE_j}
\Prb[\cE_n^{(j)c}]\leq C_1\exp(-C_2n) \,.
\end{align}
For a proof of the control of the probability of the first property see Theorem 7.68 in \cite{Grimmett99} or \cite{Pisztora}, for the second property see \cite{Pisztora}.
\noindent If the cube $B_j$ is good, we denote by $C_j$ its unique open cluster of diameter larger than $n/k$ in $nB'_j$, for $1\leq j\leq l_1$. On the event \smash{$\bigcap_{1\leq j\leq l_1} \cE_n^{(j)}\cap\big\{\,0\in\sC_\infty\,\big\}$}, the set \smash{$\bigcup_{j=1}^{l_1}C_j$} is connected without using edges of $\Gamma_n$ and contains $0$, therefore, it is a subgraph of $H_n$.  Furthermore, we claim that, on this event, we have $\sC_\infty\cap (\bigcup_ {1\leq j\leq l_1} nB_j)\subset H_n$. Indeed, let us assume that there exists $x\in \sC_\infty\cap (\bigcup_ {1\leq j\leq l_1} nB_j)$ that does not belong to $H_n$. Both $0$ and $x$ belong to $\sC_\infty$, therefore, $x$ is connected to $0$ by a path $\gamma=(x_0,e_1,\dots,e_l,x_l)$ with $x_0=0$ and $x_l=x$ that uses edges in $\Gamma_n$. We define $$r=\sup\big\{\,i\geq 1,\, e_i\in \Gamma_n\,\big\}\,.$$ 
\begin{figure}[H]
\def\svgwidth{0.7\textwidth}
\begin{center}
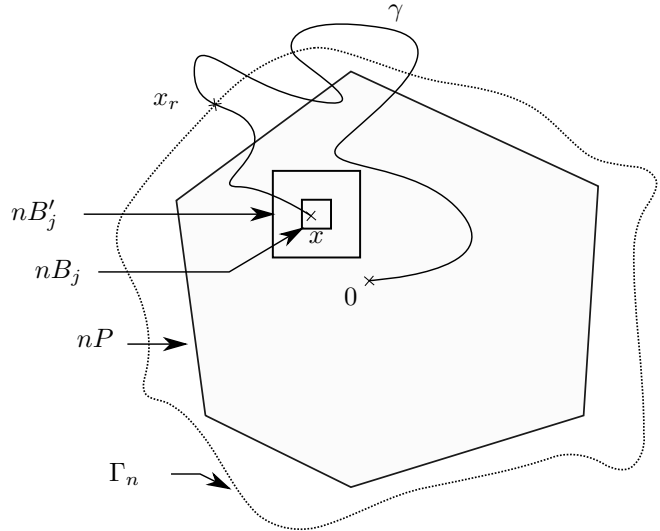
\caption[fig6]{\label{fig6}Vertices in $H_n$}
\end{center}
\end{figure}
\noindent By construction, as $e_l\notin \Gamma_n$, we have $r<l$. Let us denote $\gamma'=(x_r,e_{r+1},\dots,x_l)$. The path $\gamma'$ is not connected to $H_n$ without using edges in $\Gamma_n$ (see figure \ref{fig6}). Let $j$ such that $x\in nB_j$, by construction $x_r$ is outside $nB'_j$. Moreover, on the event $\smash{\cE_n^{(j)}}$, the cube $nB'_j$ contains a unique cluster of diameter larger than $n/k$. As the path $\gamma'$ starts outside $nB'_j$ and ends inside $nB_j$, its intersection with $nB'_j$ has a diameter larger than $n/k$. Besides, the path $\gamma'$ is not connected to $H_n$ in $nB'_j$ by an open path, so the cube $nB'_j$ contains two open clusters of diameter larger than $n/k$. This is a contradiction with the first property of a good cube.
Therefore, on the event $\bigcap_{1\leq j\leq l_1} \cE_n^{(j)}\cap \big\{\,0\in\sC_\infty\,\big\}$,
\begin{align}\label{eq45'}
|H_n|&\geq |\sC_\infty\cap (\cup_ {1\leq j\leq l_1} nB_j)|\nonumber\\
&\geq (\theta_p-\delta)\sum_{i=1}^{l_1}\cL^d(nB_i)\,. 
\end{align}
Thanks to inequalities \eqref{eqbis} and \eqref{eq45'}, we obtain
\begin{align}\label{eq46}
|H_n|\geq (\theta_p-\delta)(1-\delta)\cL^d(nP)\,.
\end{align}
To ensure that $H_n$ is a valid subgraph, it remains to check that $|H_n|\leq n^d$, yet we have
\begin{align*}
|H_n|&\leq (\theta_p+\delta)\sum_{i=1}^{l_1}\cL^d(nB_i) +\sum_{i=l_1+1}^{l_2}\cL^d(nB_i)\\
&\leq (\theta_p+\delta)n^d\cL^d(P)+n^d\delta \cL^d(P)\\
&\leq (\theta_p+2\delta)n^d\cL^d(P)\,.
\end{align*}
As $\cL^d(P)<\cL^d(W_p)$, we can choose $\delta$ small enough such that
$$|H_n|\leq \theta_p\cL^d(W_p)n^d\leq n^d\,.$$
Finally, on the event
$$\bigcap_{1\leq j\leq l_1} \cE_n^{(j)}\cap \big\{\,|\Gamma_n|_o\leq (1+\delta)\cI_p(P) n^{d-1}\,\big\}\cap \big\{\,0\in\sC_\infty\,\big\}\,,$$
combining \eqref{eq44} and \eqref{eq46}, we obtain, for small enough $\delta$,
$$n\varphi_n\leq \frac{|\Gamma_n|_o}{|H_n|}\leq (1+\delta)\dfrac{\cI_p(P)}{(\theta_p-2\delta)(1-\delta)\cL^d(P)}\leq (1+\ep/2)\dfrac{\cI_p(P)}{\theta_p\cL^d(P)}\,.$$
Combining the result of Lemma \ref{lem2} and inequalities \eqref{bla}, \eqref{eq444} and \eqref{contE_j}, we obtain
\begin{align*}
&\Prb\left[n\varphi_n \geq (1+\ep) \dfrac{\cI_p(W_p)}{\theta_p(d)\cL^d(W_p)}\Big |0\in\sC_\infty\right]\\&\hspace{5cm}\leq \frac{l_1 C_1}{\theta_p}\exp(-C_2n)+\frac{mC_1}{\theta_p}\exp(-C_2\delta n^{d})\,.
\end{align*}
This yields the result. 
\end{proof}

\section{Construction of a continuous object}\label{sectconstruction}
The aim of this section is to build a continuous object $P_n$ from a minimizer $G_n\in\cG_n$.
\subsection{Some useful results on the minimizers}
The following lemma ensures that the size of the minimizers $G_n\in\cG_n$ are of order $n^d$. 
\begin{lem}\label{nottoobig}Let $d\geq 2$ and $p>p_c(d)$. There exist positive constants $D_1,\,D_2$ and $\eta_1$ depending only on $d$ and $p$ such that, for all $n\geq 1$,
$$\Prb\Big[\,\exists G_n\in\cG_n,\, |G_n|\leq \eta_1n^d\,\big|\,0\in\sC_\infty\,\Big]\leq D_1\exp(-D_2n^{(d-1)/2d)})\,.$$
\end{lem}
To prove Lemma \ref{nottoobig}, we adapt the proof of Lemma A.8 in \cite{Gold2016}. We need the following proposition that ensures that the open edge boundary of a large subgraph is not too small.
\begin{prop}[Berger-Biskup-Hoffman-Kozma, Proposition 5.2. in \cite{berger2008anomalous}]\label{propv}Let $d\geq 2$ and $p>p_c(d)$. There exist positive constants $c_1,\,c_2$ and $c_3$ depending only on $d$ and $p$ such that, for all $t\geq 0$,
$$\Prb\left[\begin{array}{c}\text{There exists an open connected graph containing $0$}\\\text{ such that $|G|\geq t^{d/(d-1)}$, $|\partial^o G|\leq c_3|G|^{(d-1)/d}$}\end{array}\right]\leq c_1\exp(-c_2t)\,.$$
\end{prop}
\begin{proof}[Proof of Lemma \ref{nottoobig}]
Thanks to Theorem \ref{ULD}, there exist positive constants $c'_1$, $c'_2$ and $c'_3$ depending only on $p$ and $d$ such that for all $n\geq 1$,
 $$\Prb\Big[\,\varphi_n\geq c'_3n^{-1}\,\big|\,0\in\sC_\infty\,\Big]\leq c'_1\exp(-c'_2n)\,.$$ Let $G_n\in\cG_n$.
If $|G_n|\leq \sqrt{n}$, as $G_n\subset \sC_\infty$ the set $\partial^oG_n$ is non empty on the event $\{0\in\sC_\infty\}$ and so $\varphi_n \geq n^{-1/2}$. This is impossible for large $n$.
We now assume $|G_n|>\sqrt{n}$. Using Proposition \ref{propv} with $t=n^{(d-1)/2d}$, conditioning on $\{0\in\sC_\infty\}$, we obtain that $|\partial^o G_n|\geq c_3|G_n|^{(d-1)/d}$ with probability at least $1-c_1\exp(-c_2n^{(d-1)/2d})/\theta_p$. Moreover, on the event $\big\{\varphi_n\leq c'_3n^{-1}\big\}\cap\big\{0\in\sC_\infty\big\}$, we obtain 
$$ c_3 |G_n|^{-1/d}\leq \frac{|\partial^oG_n|}{|G_n|}=\varphi_n\leq c'_3n^{-1}.$$
So we set $\eta_1=(c_3/c'_3)^d$.
Finally,
\begin{align*}
\Prb&\Big[\exists G_n\in\cG_n,\, |G_n|\leq \eta_1n^d\,\big|\,0\in\sC_\infty\,\Big]\\
 &\hspace{2cm}\leq\Prb\Big[\,\varphi_n\geq c'_3n^{-1}\,\big|\,0\in\sC_\infty\,\Big]+\frac{c_1}{\theta_p}\exp(-c_2n^{(d-1)/2d})\\
&\hspace{2cm} \leq c'_1\exp(-c'_2n)+\frac{c_1}{\theta_p}\exp(-c_2n^{(d-1)/2d})\,.
\end{align*}
This yields the result.
\end{proof}

\subsection{Construction of a continuous set}
To study the upper large deviations, we needed to go from a continuous object to a discrete object. In this section, we do the opposite. From now on, we will always condition on the event $\{0\in\sC_\infty\}$. We start with $G_n\in\cG_n$ and we build a continuous object $P_n$. Our goal is to build a continuous object of finite perimeter which is close to $n^{-(d-1)}|\partial^o G_n|$. Although it seems natural to take the continuous object $P_n=n^{-1}(G_n+[-1/2,1/2]^d)$, this turns out to be a bad choice because the boundary $\partial G_n$ may be very tangled and its size may be of higher order than $n^{d-1}$. We will build from $G_n$ a graph $F_n$ with a smoother boundary $\Gamma_n\subset \E^d$ in order to build the continuous object $P_n$. At this point, there is some work left. If we consider the subgraph $F_n$ that contains all the vertices in $\sC_\infty$ enclosed in $\Gamma_n$, the symmetric difference $F_n\Delta G_n$ may be big due to the presence of holes in $G_n$, more precisely portions of $\sC_\infty$ enclosed in $\Gamma_n$ but not contained in $G_n$ (see Figure \ref{fig5}). Indeed, if these holes are too large, the symmetric difference $F_n\Delta G_n$ will be large too. However, we cannot keep all the holes in $G_n$ to build $F_n$ because when we will pass to a continuous object $P_n$, these holes will considerably increase the perimeter of $P_n$ so that $P_n$ may have a too large perimeter. The solution is to fill only the small holes to obtain $F_n$ so that the perimeter of $P_n$ remains of the correct order and the symmetric difference $F_n\Delta G_n$ remains small. In order to do so, we shall perform Zhang's construction in \cite{Zhang2017} to obtain a smooth boundary $\Gamma_n$ for $G_n$ but also to surgically remove these large holes from $G_n$ by cutting along a smooth boundary. This work was done in \cite{Gold2016}. We will only partially sketch Zhang's construction and we refer to  \cite{Zhang2017} for a rigorous proof and more details about the construction. Although we did the same construction as Gold in \cite{Gold2016}, we do not use the same argument to conclude. Gold used a procedure called webbing to link all the different contours together in order to obtain a single connected object, this simplifies the combinatorial estimates. Here, we do not perform the webbing procedure, instead we use adequate combinatorial estimates. Avoiding the webbing procedure enables us to extend the result to dimension $2$.

Let us define a renormalization process. For a large integer $k$, that will be chosen later, we set $B_k=[-k,k[^d\cap \sZ^d$ and define the following family of $k$-cubes, for $\textbf{i}\in \sZ^d$,
$$B_k(\textbf{i})=\tau_{\textbf{i}(2k+1)}(B_k)\,,$$
where $\tau_b$ denotes the shift in $\sZ^d$ with vector $b\in\sZ^d$. The lattice $\sZ^d$ is the disjoint union of this family: $\sZ^d=\sqcup_{\textbf{i}\in\sZ^d}B_k(\textbf{i})$. We introduce larger boxes $B'_k$, for $\textbf{i}\in \sZ^d$, we define
$$B'_k(\textbf{i})=\tau_{\textbf{i}(2k+1)}(B_{3k}).$$
  Underscore will be used to denote sets of cubes. For any set of $k$-cubes $\underline{A}$, the set $\underline{A}'$ denotes the set of the corresponding $3k$-cubes. Let $G_n\in\cG_n$. We first use Zhang's construction to build a smooth cutset $\Gamma_n$ that separates $G_n$ from infinity. We denote by $\underline{A}$ the set of $k$-cubes that intersect $\partial_e G_n$, the exterior edge boundary of $G_n$. We then modify the current configuration $\omega$ into a configuration $\omega'$ by closing all the open edges in $\partial G_n$. This procedure is only formal as we will eventually reopen these edges. Zhang's construction enables us  to extract a set of cubes $\underline{\Gamma} \subset \underline{A}$ such that $\underline{\Gamma}$ is $*$-connected and in the configuration $\omega'$, the union of the $3k$-cubes of $\underline{\Gamma}'$ contains a closed cutset $\Gamma_n$ that isolates $G_n$ from infinity and a rare event occurs in every cube of $\underline{\Gamma}$. These rare events are due to the existence of a closed cutset that creates a large interface of closed edges, this is a very unlikely event when $p>p_c(d)$.  Of course, when we will eventually switch back to the configuration $\omega$, these rare events will not occur anymore in some cubes.

Several connected components of $\sC_\infty \setminus G_n$ in $\sZ^d\setminus \Gamma_n$ are enclosed in $\Gamma_n$ (see Figure \ref{fig5}). We say that a connected component $C$ of $\sC_\infty$ is surrounded by $\Gamma_n$ if any path from $C$ to infinity has to use an edge of $\Gamma_n$. We will say that $C$ is large if $|C|\geq n^{1-1/2(d-1)}$.
We enumerate the large connected components $L_1,\dots, L_m$ and the small connected component $S_1,\dots,S_N$. We denote by $m(G_n)$ the number of large connected components of $\sC_\infty \setminus G_n$ enclosed in $\Gamma_n$. 
\begin{rk}
We insist here on the fact that these large components are not holes of the infinite cluster but holes of $G_n$ (see Figure \ref{fig5}). Intuitively, we do not expect that a minimizer contains such holes because the graph obtained by filling all these holes have a smaller isoperimetric ratio. Indeed, by filling these holes, we reduce the open edge boundary and increase the volume. However, by filling these holes, the volume may exceed $n^d$ and the graph we obtain by filling these holes may not be admissible. That is the reason why we cannot easily discard the presence of these large holes inside $G_n$. To obtain the proper order of large deviations, one would have to fix this issue.
\end{rk}
 We then build $F_n\subset \sC_\infty$ by filling the small connected components $S_1,\dots,S_N$ of $G_n$, \textit{i.e.}, 
\begin{align}\label{jsap5}
F_n=G_n\cup\bigcup_{i=1}^NS_i\,.
\end{align}
 At this point, the boundary $\partial F_n \setminus \partial_e F_n$ of $F_n$ may be still tangled around the large components. In the configuration $\omega'$, for each $1\leq j\leq m$, there exists a closed cutset that separates $L_j$ from infinity. We can apply Zhang's construction to each component $L_j$ in order to build a smooth closed cutset \smash{$\widehat{\Gamma}_n^{(j)}$} and its corresponding set of $k$-cubes \smash{$\widehat{\underline{\Gamma}}_n^{(j)}$}. Thanks to Zhang's construction, the set of cubes \smash{$\underline{\widehat{\Gamma}}_n^{(j)}$} is $*$-connected and in the configuration $\omega'$, a rare event occurs in each of its cubes. We denote the boundary of $F_n$ by $\widetilde{\Gamma}_n$ and its associated set of $k$-cubes $\underline{\widetilde{\Gamma}}_n$ as
$$\widetilde{\Gamma}_n =\Gamma_n\cup \bigcup_{i=1}^m\widehat{\Gamma}_n^{(i)},\,\hspace{1cm}\underline{\widetilde{\Gamma}}_n =\underline{\Gamma}_n\cup \bigcup_{i=1}^m\underline{\widehat{\Gamma}}_n^{(i)}\,.$$
\begin{figure}[H]
\def\svgwidth{1\textwidth}
\begin{center}
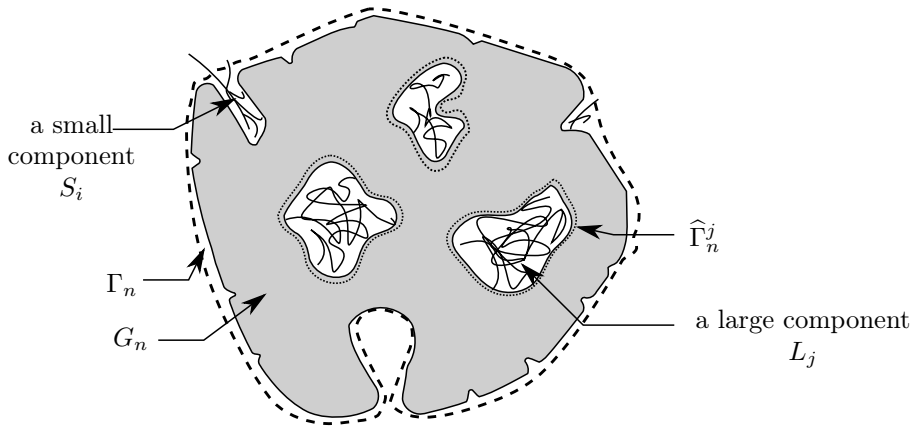
\caption[fig5]{\label{fig5}Construction of $\widetilde{\Gamma}_n$ for a $G_n\in\cG_n$}
\end{center}
\end{figure}
\noindent The set of $k$-cubes  $\smash{\underline{\widetilde{\Gamma}}_n}$ is not $*$-connected. It only contains cubes where a rare event occurs in the configuration $\omega'$. Although for some cubes these events do not occur anymore in the configuration $\omega$, we can bound from below the number of cubes that remain unchanged by $|\smash{\underline{\widetilde{\Gamma}}_n}|-|\partial ^o G_n|$. In these cubes, rare events still occur when we switch back to the original configuration $\omega$. Using a Peierls argument, we can deduce that, with high probability, $|\underline{\widetilde{\Gamma}}_n|$ and $|\partial ^o G_n|$ are of same order when $k$ is taken large enough. To perform the combinatorial estimates we will need the two following propositions.
\begin{prop}[Lemmas 6, 7 in \cite{Zhang2017}]\label{contprbmauvais}Let $d\geq2$ and let $p>p_c(d)$. There exist positive constants $C_1$ and $C_2$ depending only on $p$ and $d$ such that for each $k$-cube $B_k$,
$$\Prb(\text{a rare event occurs in }B_k)\leq C_1\e^{-C_2k}\,.$$
Moreover, this rare event depends only on the configuration of the $3k$-cube $B'_k$.
\end{prop}
\begin{rk} We do not define here what these rare events are, we refer to \cite{Zhang2017} for a precise definition of these rare events. For our purpose we only need to know that the decay is exponential in $k$. We say that a cube is abnormal if a rare event occurs in this cube.
\end{rk}
\begin{prop}\label{propcompcon}
Let $d\geq2$ and $p>p_c(d)$. There exist positive constants $c_1$, $c_2$ and $c_3$ such that 
$$\Prb\left[\exists G_n\in\cG_n,\,m(G_n)>c_3n^{d-2+3/2d}\,\big|\,0\in\sC_\infty\right]\leq c_1\exp(-c_2n^{1-3/2d})\,.$$
\end{prop}
\begin{proof}Thanks to Theorem \ref{ULD}, there exist positive constants $C'_1$, $C'_2$ and $C'_3$ depending only on $p$ and $d$ such that for all $n\geq 1$,
 $$\Prb\Big[\,\varphi_n\geq C'_3n^{-1}\,\big|\,0\in\sC_\infty\,\Big]\leq C'_1\exp(-C'_2n)\,.$$ Let $G_n\in\cG_n$. We have with probability at least $1- C'_1\exp(-C'_2n)$ that
 $$|\partial^o G_n|\leq C'_3n^{-1}|G_n|\leq C'_3n^{d-1}\,.$$
 Thanks to Proposition \ref{propv}, there exist positive constants $c'_1$, $c'_2$ and $c'_3$ depending only on $p$ and $d$ such that, for all $t\geq 0$, we have 
\begin{align}\label{jsap}
\Prb\left[\begin{array}{c}\text{There exists an open connected graph containing $0$}\\\text{ such that $|G|\geq t^{d/(d-1)}$, $|\partial^o G|\leq c'_3|G|^{(d-1)/d}$}\end{array}\right]\leq c'_1\exp(-c'_2t)\,.
\end{align}
 In the following, we set $t=n^{(1-1/2(d-1))(d-1)/d}=n^{1-3/2d}$. First notice that by construction, each $L_j$ is contained in $[-n^d,n^d]\cap\sZ^d$. We have
\begin{align*}
 &\Prb\left[\exists G_n\in\cG_n,\, \exists i\in\{1,\dots,m(G_n)\}, \,|\partial^oL_i|\leq c'_3n^{(1-1/2(d-1))d/(d-1)}\,\big|\,0\in\sC_\infty\,\right]\\
 &\leq \Prb\left[\exists G_n\in\cG_n, \,\exists i\in\{1,\dots,m(G_n)\}, \,|\partial^oL_i|\leq c'_3|L_i|^{d/(d-1)}\,\big|\,0\in\sC_\infty\right]\\
 &\leq \frac{1}{\theta_p}\Prb\left[\begin{array}{c}\text{There exists an open connected graph $G$ contained in }\\ \text{$[-n^d,n^d]\cap\sZ^d$ such that $|G|\geq t^{d/(d-1)}$, $|\partial^o G|\leq c'_3|G|^{(d-1)/d}$}\end{array}\right]\\
 &\leq \frac{1}{\theta_p}\sum_{x\in[-n^d,n^d]\cap\sZ^d}\Prb\left[\begin{array}{c}\text{There exists an open connected graph $G$ containing}\\\text{$x$ such that $|G|\geq t^{d/(d-1)}$, $|\partial^o G|\leq c'_3|G|^{(d-1)/d}$}\end{array}\right].
 \end{align*}
 Using the translation invariance together with inequality \eqref{jsap}, we obtain
 \begin{align*}
 \Prb&\left[\exists G_n\in\cG_n,\, \exists i\in\{1,\dots,m(G_n)\}, \,|\partial^oL_i|\leq c'_3n^{(1-\ep)d/(d-1)}\,\Big|\,0\in\sC_\infty\,\right]\\
 &\hspace{1cm}\leq \frac{(2n^d)^d}{\theta_p}\, \Prb\left[\begin{array}{c}\text{There exists an open connected graph $G$ containing $0$}\\\text{ such that $|G|\geq t^{d/(d-1)}$, $|\partial^o G|\leq c'_3|G|^{(d-1)/d}$}\end{array}\right]\\
 &\hspace{1cm}\leq\frac{(2n^d)^d}{\theta_p}c'_1\exp(-c'_2n^{1-3/2d})\,.
 \end{align*}
 By construction, for all $i\in\{1,\dots,m(G_n)\}$, we have $\partial^o L_i\subset\partial ^o G_n$ and for all $j\in\{1,\dots,m(G_n)\}$ such that $i\neq j$, we have $\partial ^o L_i\cap \partial ^o L_j=\emptyset$. Thus, with high probability, 
 $$m(G_n)\leq \frac{|\partial^oG_n|}{c'_3n^{(1-\ep)d/(d-1)}}\leq \frac{C'_3n^{d-1}}{c'_3n^{1-3/2d}}\leq \frac{C'_3}{c'_3}n^{d-2+3/2d}\,.$$
 Finally, by setting $c_3=C'_3/c'_3$, we obtain
 \begin{align*}
 \Prb&\left[\exists G_n\in\cG_n, \,m(G_n)>c_3n^{d-2+3/2d}\,\big|\,0\in\sC_\infty\,\right]\\
 &\leq \Prb\Big[\,\varphi_n\geq c'_3n^{-1}\,\big|\,0\in\sC_\infty\,\Big]+\Prb\left[\begin{array}{c}\exists G_n\in\cG_n,\, \exists i\in\{1,\dots,m(G_n)\},\\ \,|\partial^oL_i|\leq c'_3n^{(1-\ep)d/(d-1)}\end{array}\Big|\,0\in\sC_\infty\,\right]\\
 &\leq C'_1\exp(-C'_2n)+\frac{(2n^d)^d}{\theta_p}c'_1\exp(-c'_2n^{1-3/2d})\,.
 \end{align*}
 This yields the result.
\end{proof}
Using the control on the number of large components $m(G_n)$ of $\sC_\infty$ enclosed in $\Gamma_n$ and a Peierls argument, we obtain the following control of $|\widetilde{\Gamma}_n|$:
\begin{prop}\label{controlper}
Let $d\geq 2$ and $p>p_c(d)$. There exist positive constants $\beta_0$, $C_1$, $C_2$ depending only on $d$ and $p$ such that, for all $n\geq 1$, for all $\beta\geq \beta_0$,
$$\Prb\left[\max_{G_n\in\cG_n}|\widetilde{\Gamma}_n|\geq \beta n^{d-1}\,\big|\,0\in\sC_\infty\,\right]\leq C_1\exp(-C_2 n^{1-3/2d})\,.$$
\end{prop}
\begin{proof}Let $k$ be a large integer that we will choose later. We consider a renormalization process of parameter $k$. Let $G_n\in\cG_n$. First notice that as $\widetilde{\Gamma}_n\subset\bigcup_{B\in\underline{\widetilde{\Gamma}}_n}B'$, we have
$$|\widetilde{\Gamma}_n|\leq (6k)^d|\underline{\widetilde{\Gamma}}_n|\,.$$
Thus, it is enough to control the quantity $|\underline{\widetilde{\Gamma}}_n|$ to prove Proposition \ref{controlper}.
 We can rewrite $\underline{\widetilde{\Gamma}}_n$ as
$$\underline{\widetilde{\Gamma}}_n=\bigcup_{i=1}^{m'}\underline{A}_i\text{    with $m'\leq m(G_n)$}$$
where the $\underline{A}_i$ are pairwise disjoint $*$-connected sets of cubes. 
Thanks to Theorem \ref{ULD}, there exist positive constants $C'_1$, $C'_2$ and $C'_3$ depending only on $p$ and $d$ such that for all $n\geq 1$,
\begin{align}\label{jsap2}
\Prb\Big[\,\varphi_n\geq C'_3n^{-1}\,\big|\,0\in\sC_\infty\,\Big]\leq C'_1\exp(-C'_2n)\,.
\end{align}
 Let $G_n\in\cG_n$. We have with probability at least $1- C'_1\exp(-C'_2n)$ that
 $$|\partial^o G_n|\leq C'_3n^{d-1}\,.$$
 We choose $\beta$ large enough such that 
 $$C'_3\leq \frac{\beta}{2\cdot 4^d}\,,$$
 so that 
  $$|\partial^o G_n|\leq C'_3n^{d-1}\leq \frac{\beta}{2\cdot 4^d}n^{d-1}\,.$$
 We now want to sum over the possible realizations of $\underline{\widetilde{\Gamma}}_n$. 
Using Proposition \ref{propcompcon} together with inequality \eqref{jsap2}, we get
\begin{align}\label{jsap3}
\Prb&\left[\exists G_n\in\cG_n,\,|\underline{\widetilde{\Gamma}}_n|\geq \beta n^{d-1}\,\Big|\,0\in\sC_\infty\right]\nonumber\\
&\leq \Prb\left[\exists G_n\in \cG_n,\begin{array}{c} \sum_{i=1}^{m'}|\underline{A}_i|\geq \beta n^{d-1},\,m'\leq c'_3n^{d-2+3/2d},\\ |\partial^o G_n|\leq \frac{\beta}{2\cdot 4^d}n^{d-1}\end{array}\,\Big|\,0\in\sC_\infty\right]\nonumber\\
&\hspace{0.5cm}+c_1\exp(-c_2n^{1-3/2d})+C'_1\exp(-C'_2n)\nonumber\\
&\leq \sum _{j\geq\beta n^{d-1}}\sum_{m'=1}^{c'_3n^{d-2+3/2d}}\sum_{\substack{j_1+\dots+j_{m'}=j\\ j_1>0,\,\dots,j_{m'}> 0}}\sum_{x_1,\dots,x_{m'}\in[-n^d,n^d]^d\,\,}\sum_{\substack{A_1\in \Animals_{x_1}\\|A_1|=j_1}}\cdots\nonumber\\
&\hspace{0.2cm}\cdots\sum_{\substack{A_{m'}\in \Animals_{x_{m'}}\\|A_{m'}|=j_{m'}}}\Prb\Big[\exists G_n\in\cG_n,\,\underline{\widetilde{\Gamma}}_n=\bigcup_{i=1}^{m'}A_i,\,|\partial^o G_n|\leq \frac{\beta}{2\cdot 4^d}n^{d-1}\,\Big|\,0\in\sC_\infty\Big]\nonumber\\
& \hspace{0.5cm}+c_1\exp(-c_2n^{1-3/2d})+C'_1\exp(-C'_2n)\,.
\end{align}
Let us assume $\underline{\widetilde{\Gamma}}_n=\bigcup_{i=1}^{m'}A_i$. We can extract from $\underline{\widetilde{\Gamma}}_n$ a set of $k$-cubes $\underline{\widetilde{\Gamma'}}_n$ such that 
$|\underline{\widetilde{\Gamma'}}_n|\geq |\underline{\widetilde{\Gamma}}_n|/4^d$  and for any $\textbf{i}\neq\textbf{j}$ such that $B_k(\textbf{i}),B_k(\textbf{j}) \in \underline{\widetilde{\Gamma'}}_n$ we have $B'_k(\textbf{i})\cap B'_k(\textbf{j})=\emptyset$.
As the rare event depends only on the configuration in the $3k$-cube $B'_k(\textbf{j})$, the two following events $\big\{\text{a rare event occurs in $B_k(\textbf{i})$}\big\}$ and $\big\{\text{a rare event occurs in $B_k(\textbf{j})$}\big\}$ are independent. Using Proposition \ref{contprbmauvais}, we obtain
\begin{align*}
\Prb&\left[\exists G_n\in\cG_n,\,\underline{\widetilde{\Gamma}}_n=\bigcup_{i=1}^{m'}A_i,\,|\partial^o G_n|\leq \frac{\beta}{2\cdot 4^d}n^{d-1}\,\Big|\,0\in\sC_\infty\right]\\
&\hspace{0.3cm}\leq\Prb\left[\exists G_n\in\cG_n,\,\underline{\widetilde{\Gamma}}_n=\bigcup_{i=1}^{m'}A_i,\, |\underline{\widetilde{\Gamma'}}_n|\geq j/4^d,\,|\partial^o G_n|\leq \frac{\beta}{2\cdot 4^d}n^{d-1}\,\Big|\,0\in\sC_\infty\right]\\
&\hspace{0.3cm}\leq \Prb\left[\exists G_n\in\cG_n,\begin{array}{c}\underline{\widetilde{\Gamma}}_n=\bigcup_{i=1}^{m'}A_i,\,|\partial^o G_n|\leq \frac{\beta}{2\cdot 4^d}n^{d-1}, \\|\{B\subset\underline{\widetilde{\Gamma'}}_n, \text{ $B$ abnormal}\}|\geq j/4^d-|\partial ^o G_n|\end{array}\Big|\,0\in\sC_\infty\right]\\
&\hspace{0.3cm}\leq \Prb\left[\exists G_n\in\cG_n,\,\underline{\widetilde{\Gamma}}_n=\bigcup_{i=1}^{m'}A_i,\, |\{B\subset\underline{\widetilde{\Gamma'}}_n, \text{ $B$ abnormal}\}|\geq j/(2.4^d)\right]\cdot \frac{1}{\theta_p}\\
&\hspace{0.3cm}\leq\frac{4^d}{\theta_p} \sum_{l\geq j/(2.4^d)}\left(C_1\e^{-C_2k}\right)^{ l}\,\\
&\hspace{0.3cm}\leq\frac{2\cdot 4^d}{\theta_p} \left(C_1\e^{-C_2k}\right)^{ j/(2.4^d)}\,
\end{align*}
where $k$ will be chosen large enough such that $C_1\e^{-C_2k}\leq 1/2$.
So together with inequality \eqref{jsap3} and using Lemma \ref{tailleanimal}, we obtain
\begin{align*}
\Prb&\left[\exists G_n\in\cG_n,\,|\underline{\widetilde{\Gamma}}_n|\geq \beta n^{d-1}\,\big|\,0\in\sC_\infty\right]\\
&\leq \sum _{j\geq\beta n^{d-1}}\sum_{m'=1}^{c'_3n^{d-2+3/2d}}\sum_{\substack{j_1+\dots+j_{m'}=j\\ j_1>0,\,\dots,j_{m'}> 0}}\sum_{x_1,\dots,x_{m'}\in[-n^d,n^d]^d\,\,}\sum_{\substack{A_1\in \Animals_{x_1}\\|A_1|=j_1}}\cdots\\
&\cdots\sum_{\substack{A_{m'}\in \Animals_{x_{m'}}\\|A_{m'}|=j_{m'}}}\frac{2\cdot 4^d}{\theta_p} \left(C_1\e^{-C_2k}\right)^{ j/(2.4^d)}+c_1\e^{-c_2n^{1-3/2d}}+C'_1\e^{-C'_2n}\\
&\leq \frac{2\cdot 4^d}{\theta_p}  \sum _{j\geq\beta n^{d-1}}\left(C_1\e^{-C_2k}\right)^{ \frac{j}{2.4^d}}\sum_{m'=1}^{c'_3n^{d-2+3/2d}}\sum_{\substack{j_1+\dots+j_{m'}=j\\ j_1>0,\,\dots,j_{m'}> 0}}(2n)^{d^2m'}7^{dj_1}\cdots 7^{dj_{m'}}\\
&\hspace{0.5cm}+c_1\e^{-c_2n^{1-3/2d}}+C'_1\e^{-C'_2n}\\
&\leq  \frac{2\cdot 4^d}{\theta_p}  \sum _{j\geq\beta n^{d-1}}7^{dj}\left(C_1\e^{-C_2k}\right)^{ j/(2.4^d)} \sum_{m'=1}^{c'_3n^{d-2+3/2d}}(2n^d)^{dm'}\cdot\\
&\hspace{0.5cm}\times\left|\left\{\,(j_1,\dots,j_{m'}): \begin{array}{c} j_1+\dots+j_{m'}=j,\\\, j_1>0,\,\dots,j_{m'}> 0\end{array}\,\right\}\right|+c_1\e^{-c_2n^{1-3/2d}}+C'_1\e^{-C'_2n}\\
&\leq \frac{2\cdot 4^d}{\theta_p}    (2n^d)^{d(c'_3n^{d-2+3/2d}+2)} \sum _{j\geq\beta n^{d-1}} (2\cdot7^d)^j\left(C_1\e^{-C_2k}\right)^{ j/(2.4^d)} \\
&\hspace{0.5cm}+c_1\e^{-c_2n^{1-3/2d}}+C'_1\e^{-C'_2n}
\end{align*}
We now choose $k$ large enough such that 
$$C_1\e^{-C_2k}\leq \frac{1}{2}\qquad \text{and}\qquad \left((2\cdot 7^d)^{2\cdot 4^d}C_1\e^{-C_2k}\right)^{ 1/(2.4^d)}\leq \e^{-1}$$
Finally, we get
\begin{align*}
\Prb&\left[\exists G_n\in\cG_n,\,|\underline{\widetilde{\Gamma}}_n|\geq \beta n^{d-1}\,\big|\,0\in\sC_\infty\right]\\
&\leq  \frac{2(2n^d)^{d(c'_3n^{d-2+3/2d}+2)}4^{d}}{\theta_p}\sum_{j\geq \beta n ^ {d-1}} \left((2\cdot7^d)^{2\cdot4^d}C_1\e^{-C_2k}\right)^{j/(2.4^d)}\\
&\hspace{1cm}+c_1\e^{-c_2n^{1-3/2d}}+C'_1\e^{-C'_2n}\\
&\leq \frac{4 ^{d+1}\exp\left(2d^2c'_3n^{d-2+3/2d}\log n- \beta n^{d-1} \right)}{\theta_p} +c_1\e^{-c_2n^{1-3/2d}}+C'_1\e^{-C'_2n}\,.
\end{align*}
This yields the result for $\beta\geq \beta_0$ where $\beta_0$ is such that for all $n\geq0$, we have $\beta_0>(4d^2c'_3\log n) /n^{1-3/2d}$.
\end{proof}
We can now build the relevant continuous object $P_n$. Given a finite set of edges $S$, we define 
$$\hull(S)=\Big\{\,x\in \sZ^d:\, \text{any path from $x$ to infinity has to use an edge of $S$}\,\Big\}\,$$
and
$$H_n=\hull(\Gamma_n)\setminus \left(\bigcup_{i=1}^m \hull(\widehat{\Gamma}_n^{(i)})\right)\,.$$
\noindent We define $P_n$ and its associated measure $\nu_n$ as
$$P_n=\frac{1}{n} \left(H_n+\left[-\dfrac{1}{2},\dfrac{1}{2}\right]^d\right)\,,$$$$\forall E\in \cB\left(\sR^d\right),\,\nu_n(E)=\theta_p\cL^d(P_n\cap E)\,.$$
We obtain a control on the size of the perimeter of $P_n$ by  a straightforward application of Proposition \ref{controlper}:
\begin{cor}\label{cor}
Let $d\geq 2$ and $p>p_c(d)$. There exist positive constants $\beta_0$, $C_1$, $C_2$ depending only on $d$ and $p$ such that for all $n\geq 1$, for all $\beta>\beta_0$,
$$\Prb\left[\max_{G_n\in\cG_n}\cP(nP_n)\geq \beta n^{d-1}\,\big|\,0\in\sC_\infty\,\right]\leq C_1\e^{-c_2n^{1-3/2d}}\,.$$
\end{cor}
\noindent The following Lemma will be useful to compare the measure $\nu_n$ with the measure associated to $F_n$.
\begin{lem}
Let $G_n\in\cG_n$ and $F_n$ as defined in \eqref{jsap5}. We have  $F_n=H_n\cap \sC_\infty$. 
\end{lem}
\begin{proof}
Let $G_n\in\cG_n$. Let $x\in H_n\cap \sC_\infty$, then $x$ belongs to $\sC_\infty\cap \hull{\Gamma_n}$ but is not in any of the large connected components $L_1,\dots, L_m$. Therefore, $x$ belongs to $G_n$ or to one of the small components $S_1,\dots,S_N$ and so $x\in F_n$.

Conversely, let $x\in F_n$. It is clear that $x\in\hull(\Gamma_n)$. Let us assume $x\in G_n$ and that there exists $i$ such that \smash{$x\in\hull( \widehat{\Gamma}_n^{(i)})$}. As $G_n$ is connected there exists an open path $\gamma$ in $G_n$ that joins $x$ with \smash{$G_n\setminus \widehat{\Gamma}_n^{(i)}$}. As the edges of \smash{$\widehat{\Gamma}_n^{(i)}\setminus \partial ^oL_i$} are closed, $\gamma$ must use an edge of $\partial ^o L_i$ and so go through a vertex of $L_i$. That is a contradiction as the path $\gamma$ uses only vertices in $G_n$. Let us now assume that $x\in S_j$ and $x\in\hull( \widehat{\Gamma}_n^{(i)})$ for some $i$ and $j$. As $x\in\sC_\infty$, $x$ is connected to infinity by an open path $\gamma'$. However, by the same arguments, to exit $\hull(\widehat{\Gamma}_n^{(i)})$, the path $\gamma'$ has to go through a vertex of $L_i$. Thus, there exist an open path in $\sC_\infty\setminus G_n$ that joins $x$ to $L_i$. That is a contradiction as $x\notin L_i$. 

%Conversely, let us assume there exists $x\in F_n \setminus H_n$, \textit{i.e.}, for some $i\in\{1,\dots,m\}$, we have \smash{$x\in\hull(\widehat{\Gamma}_n^{(i)})\cap F_n$.} Therefore we have that $x\in\sC_\infty \cap \hull(\widehat{\Gamma}_n^{(i)})$ and $x\notin L_i$. By construction, in the configuration $\omega$, all the edges in \smash{$\widehat{\Gamma}_n^{(i)}$} are closed except the edges in $\partial ^o L_i$. As $x\in\sC_\infty$, $x$ is connected to $L_i$ by an open path $\gamma=(x_0,e_1,x_1,\dots,e_n,x_q)$ where $x_0=x$, $x_q\in L_i$ and $x_0,\dots, x_{q-1}\in\sC_\infty\setminus L_i$. This path $\gamma$ attains $L_i$ before hitting the set of edges \smash{$\hull(\widehat{\Gamma}_n^{(i)})$}, otherwise it will imply that $\gamma$ uses an edge of \smash{$\widehat{\Gamma}_n^{(i)}\setminus \partial ^oL_i$} that is necessarily closed.  Thus, $x$ is connected to $L_i$ in the hull. The vertices $x_0,\dots, x_{q-1}$ are not connected to $G_n$ by an open path in $G_n$. As $G_n$ is connected, the vertices $x_0,\dots, x_{q-1}$  are not in $G_n$ so they are connected to $L_i$ in $\sZ^d\setminus G_n$. Thus, the vertices $x_0,\dots, x_{n-1}$ belong to $L_i$. This is contradiction.
 Finally, $F_n\subset H_n\cap \sC_\infty$. 
\end{proof}

\subsection{Closeness of measures}
We shall show that for any ball of constant radius centered at a point $x\in\sZ^d$, the measures $\nu_n$ and $\mu_n$ restricted to this ball are close to each other in some weak sense.
\begin{prop}\label{propcontiguity}Let $p>p_c(d)$ and $r>0$.
 Let $u:]0,+\infty[\rightarrow ]0,+\infty[$ be a non-decreasing function such that $\lim_{t\rightarrow 0}u(t)=0$. For all $\delta>0$, there exist $C_1$ and $C_2$ depending on $d$, $p$, $u$ and $\delta$ such that for all $n\geq 1$, for any finite set $\fF_n$ of uniformly continuous functions that satisfies: 
$$\forall f\in\fF_n\quad \|f\|_\infty\leq 1\qquad\text{ and }\qquad\forall x,y\in\sR^d\quad |f(x)-f(y)|\leq u(\|x-y\|_2)\,,$$
we have
$$\Prb\left[\max _ {G_n\in\cG_n}\,\sup_{f\in\fF}|\mu_n(f\ind_{B(x,r)})-\nu_n(f\ind_{B(x,r)})|>\delta\,\Big |\,0\in\sC_\infty\right]\leq C_1\e^{-C_2n^{1-3/2d} }\,.$$
\end{prop}
\begin{rk}We state here the result in a general form. In the following, we will apply this Proposition for the particular case of sets of functions that are translates of the same function. The function $u$ is an upper bound on the modulus of continuity of the functions in $\fF_n$. If we think of $\fF_n$ as a set that grows with $n$, this condition may be interpreted as a sufficient condition to obtain compactness for the set $\fF_n$ in the limit.

\end{rk}
\noindent To prove this result, we will need the following proposition that is a corollary of the results in \cite{Pisztora}:
\begin{prop}\label{Pizt}Let $d\geq 2$ and $p>p_c(d)$. Let $r>0$, and let $Q\subset \sR^d$ be a cube of side length $2r$. Let $\delta>0$. There exist positive constants $c_1$ and $c_2$ depending on $d$, $p$ and $\delta$ such that 
$$\Prb\left[\frac{|\sC_\infty \cap Q|}{\cL^d (Q)}\notin (\theta_p-\delta,\theta_p+\delta)\right]\leq c_1\exp(-c_2r^{d-1})\,.$$
\end{prop} 
\begin{proof}[Proof of Proposition \ref{propcontiguity} ] Let $\delta>0$ and $\ep>0$ that we will choose later. Let $u:]0,+\infty[\rightarrow ]0,+\infty[$ be a non-decreasing function such that $\lim_{t\rightarrow 0}u(t)=0$. Let $n\geq 1$. Let $\fF_n$ be a finite set of uniformly continuous function that satisfies:
$$\forall f\in\fF_n\quad \|f\|_\infty\leq 1\qquad\text{ and }\qquad\forall x,y\in\sR^d\quad |f(x)-f(y)|\leq u(\|x-y\|_2)\,,$$
We define $$\widetilde{\mu}_n=\frac{1}{n^d}\sum_{x\in V(F_n)}\delta _{x/n}\,.$$
Thanks to Theorem \ref{ULD}, there exists a constant $\eta_3$ depending only on the dimension such that 
$$\Prb\left[n\varphi_n\geq \eta_3 \,\Big |\,0\in\sC_\infty\right] \leq C_1\exp(-C_2n)\,.$$
Let $G_n\in\cG_n$, with probability at least $1-C_1\exp(-C_2n)$, we have $$\frac{n|\partial ^o G_n|}{|G_n|}\leq \eta_3,$$
and so $|\partial ^o G_n|\leq \eta_3 n ^{d-1}$. As each small component $S_j$ is such that $\partial^o S_j \cap \partial^o G_n\neq \emptyset$, the number $N$ of small components is at most $\eta_3 n ^{d-1}$ and by definition of $F_n$,
$$|F_n\setminus G_n|\leq \sum_{j=1}^N |S_j|\leq \eta_3 n^{d-1/2(d-1)}\,.$$
Finally, with probability at least $1-C_1\exp(-C_2n)$, for all $f\in\fF_n$, $$|\mu_n(f)-\widetilde{\mu}_n(f)|\leq \frac{1}{n^d}\|f\|_\infty |F_n\setminus G_n| \leq \eta_3 n^{-1/2(d-1)}\,,$$
and
\begin{align}\label{delta}
\underline{\Prb}\left[\max_{G_n\in\cG_n}\,\sup_{f\in\fF_n}|\mu_n(f\ind_{B(x,r)})-\nu_n(f\ind_{B(x,r)})|>\eta_3 n^{-1/2(d-1)}\right]\leq C_1\e^{-C_2n}
\end{align}
where $\underline{\Prb}$ represents the probability measure conditioned on the event $\{0\in\sC_\infty\}$.
Let $x\in \sR^d$ and let $r>0$. Let $f\in\fF_n$. We now would like to estimate the quantity $$|\widetilde{\mu}_n(f\ind_{B(x,r)})-\nu_n(f\ind_{B(x,r)})|\,.$$
 We adapt the proof of 16.2 in \cite{Cerf-Pisztora}. We use again a renormalization argument but at a different scale $L=K \ln n$. We consider the lattice rescaled by this factor $L$. We say that a cluster $C$ is crossing in a box $B$ if for any two opposite faces of $B$, the cluster $C$ contains an open path in $B$ that joins these two faces. Let $\ep>0$. For $\underline{y}\in \sZ^d$,  we define $B_n(\underline{y})=(2L\underline{y}/n)+[-L/n,L/n]^d$
and $B'_n(\underline{y})=(2L\underline{y}/n)+[-3L/n,3L/n]^d$. Let $X(\underline{y})$ be the indicator function of the event $\cE_n(\underline{y})$.
This event occurs if 

\noindent $\bullet$ Inside $nB'_n(\underline{y})$, there is a unique crossing cluster $C'$ that crosses the $3^d$ sub-boxes of $nB'_n(\underline{y})$. Moreover, $C'$ is the only cluster in $nB'_n(\underline{y})$ of diameter larger than $L$.

\noindent $\bullet$ Inside $nB_n(\underline{y})$, there is a crossing cluster $C^*$ such that $$|C^*|\geq (\theta_p-\ep) \cL ^d(nB_n(\underline{y}))\,.$$

\noindent $\bullet$ We have $\left|\{x\in nB_n(\underline{y}): x\longleftrightarrow\partial nB_n(\underline{y}\}\right|\leq (\theta_p+\ep) \cL ^d( nB_n(\underline{y}))$.

On the event $\cE_n(\underline{y})$, any cluster $C\subset nB_n(\underline{y})$ that is connected by an open path to $\partial (nB'_n(\underline{y}))$ is the unique crossing cluster, \textit{i.e.}, $C=C'=C^*$ and so it also satisfies 
$$\dfrac{|C|}{\cL^d(nB_n(\underline{y}))}\in[\theta_p-\ep,\theta_p+\ep]\,.$$
The family $(X(\underline{y}))_{\underline{y}\in\sZ^d}$ is a site percolation process on the macroscopic lattice. The states of the sites are not independent from each other but there is only a short range dependency. Indeed, for any $\underline{y}$ and $\underline{z}$ such that $|\underline{y}-\underline{z}|_\infty \geq 3$, we have that $X(\underline{y})$ and $X(\underline{z})$ are independent. We define the connected component  $\sC(\underline{y})$ of $\underline{y}$ as
$$\sC(\underline{y})=\Big\{\,\underline{z}\in\sZ^d\,:\,\text{$\underline{z}$ is connected to $\underline{y}$ by a macroscopic open path}\,\Big\}\,.$$
Let $$\underline{D}=\{\underline{y}\in\sZ^d:\, B_n(\underline{y})\subset B(x,r)\}\,.$$
We have 
\begin{align}\label{eq'3}
|\underline{D}|L^d\leq n^d\cL ^d(B(x,r)).
\end{align}
There exists an integer $n_0=n_0(u(\ep))$ such that, for $n\geq n_0(u(\ep))$, we have $L/n\leq u(\ep)$ so that
$$\cL^d\left(B(x,r)\setminus \bigcup_{\underline{y}\in \underline{D}}B_n(\underline{y})\right)\leq \ep \cL^d (B(x,r)),$$
$$\forall w,z\in\sR^d,\, \|w-z\|_2\leq \dfrac{L}{n}\Rightarrow\,|f(x)-f(y)|\leq \ep\,.$$
The last statement comes from the fact that $f$ belongs to $\fF_n$.
By decomposing $|\widetilde{\mu}_n(f\ind_{B(x,r)})-\nu_n(f\ind_{B(x,r)})|$ on cubes of size $L/n$, we obtain:
\begin{align}\label{eq'2}
&|\widetilde{\mu}_n(f\ind_{B(x,r)})-\nu_n(f\ind_{B(x,r)})|\nonumber\\
&\hspace{1cm}\leq 2\cL^d\left( B(x,r)\setminus \bigcup_{\underline{y}\in \underline{D}}B_n(\underline{y})\right)+\sum_{\underline{y}\in \underline{D}}\left|\int_{B_n(\underline{y})}fd\widetilde{\mu}_n-\int_{B_n(\underline{y})}fd\nu_n\right|\nonumber\\
&\hspace{1cm}\leq 4\ep \cL^d(B(x,r))+\sum_{\underline{y}\in \underline{D}}\left|\widetilde{\mu}_n(B_n(\underline{y}))-\nu_n(B_n(\underline{y}))\right|\,.
\end{align}
Let $\underline{y}\in \underline{D}$. We need to distinguish several cases:

\noindent $\bullet$ If $B_n(\underline{y})\cap P_n=\emptyset$, then $\nu_n(B_n(\underline{y}))=\widetilde{\mu}_n(B_n(\underline{y}))=0$. From now on we will only consider cubes such that $B_n(\underline{y})\cap P_n\neq\emptyset$.
%\item If $X(\underline{y})=0$, then we bound $$|\widetilde{\mu}_n(B_n(\underline{y}))-\nu_n(B_n(\underline{y}))|\leq \dfrac{1}{n^d}|B_n(\underline{y})|\ind_{X(\underline{y})=0}$$

\noindent $\bullet$ If $B_n(\underline{y})\not\subset P_n$, then we bound $$|\widetilde{\mu}_n(B_n(\underline{y}))-\nu_n(B_n(\underline{y}))|\leq \dfrac{1}{n^d}|B_n(\underline{y})|$$
and as $B_n(\underline{y})\cap P_n\neq\emptyset$, the cube intersects the boundary of $P_n$. Thus,
$$B_n(\underline{y})\subset \left\{z\in\sR^d:\, d_\infty (z,\partial P_n\cap B(x,r))\leq \dfrac{L}{n}\right\}\,.$$
Moreover,
\begin{align*}
&\cL^d\left(\left\{z\in\sR^d:\, d_\infty (z,\partial P_n\cap B(x,r))\leq \dfrac{L}{n}\right\} \right)\\
&\hspace{0.5cm}\leq \Big|\big\{\,x\in H_n,\,\exists y\in \sZ^d\setminus H_n, \,\|x-y\|_1=1\,\big\} \cap B(nx,nr+d)\Big|\left(\frac{2L+2}{n} \right)^d\\
&\hspace{0.5cm}\leq \cP(nP_n,B(nx,nr+d))\left(\frac{3L}{n} \right)^d\\
&\hspace{0.5cm}\leq \cP(P_n,B(x,r+d))\frac{(3L)^d}{n}\,.
\end{align*}

\noindent $\bullet$ If $B_n(\underline{y})\subset P_n$ and $|\sC(\underline{y})|=\infty$, then the crossing cluster $C^*$ of $B_n(\underline{y})$ is a portion of $\sC_\infty$ and $$\nu_n(B_n(\underline{y}))=\theta_p \frac{\cL^d(nB_n(\underline{y}))}{n^d}\text{ and }\widetilde{\mu}_n(B_n(\underline{y}))=\frac{|(nB_n(\underline{y}))\cap C^* |}{n^d}\,.$$
Thus, we have $$\widetilde{\mu}_n(B_n(\underline{y}))\in \left[(\theta_p-\ep)\cL^d(B_n(\underline{y})),(\theta_p+\ep)\cL^d(B_n(\underline{y}))\right]$$
and 
$$|\widetilde{\mu}_n(B_n(\underline{y}))-\nu_n(B_n(\underline{y}))|\leq \ep\cL^d(B_n(\underline{y}))\,.$$

\noindent $\bullet$ If $B_n(\underline{y})\subset P_n$ and $|\sC(\underline{y})|<\infty$, then we bound $$|\widetilde{\mu}_n(B_n(\underline{y}))-\nu_n(B_n(\underline{y}))|\leq \cL^d(B_n(\underline{y}))\ind_{|\sC(\underline{y})|<\infty}$$ 
By summing the previous inequalities over $\underline{y}\in \underline{D}$, thanks to inequality \eqref{eq'3} and \eqref{eq'2}, we obtain
\begin{align*}
&|\widetilde{\mu}_n(f\ind_{B(x,r)})-\nu_n(f\ind_{B(x,r)})|\\
&\hspace{1cm}\leq\cL^d(B(x,r))\Big( 5\ep+ \frac{1}{|\underline{D}|}\sum_{\underline{y}\in\underline{D}}\ind_ {|\sC(\underline{y})|<\infty}\Big)\nonumber+\cP(P_n,B(x,r+d))\frac{(3L)^d}{n}\,.
\end{align*}
Let $c(r)=6\cL^d(B(0,r))+3^d$, we get
\begin{align}\label{eqbis1}
&\underline{\Prb}\left[\max_{ G_n\in\cG_n}\,\sup_{f\in\fF_n}|\widetilde{\mu}_n(f\ind_{B(x,r)})-\nu_n(f\ind_{B(x,r)})|>c(r)\ep\right]\nonumber\\
&\hspace{0.5cm}\leq \frac{1}{\theta_p}\Prb\left[\frac{1}{|\underline{D}|}\sum_{\underline{y}\in\underline{D}}\ind_ {|\sC(\underline{y})|<\infty}\geq \ep\right]+\underline{\Prb}\left[\max_{G_n\in\cG_n}\cP(P_n,B(x,r+d))\geq\ep\frac{n}{L^d}\right]\,.
\end{align}
Besides, using Corollary \ref{cor}, for $n$ large enough, we obtain
\begin{align}\label{ineq1}
\underline{\Prb}\left[\max _{ G_n\in\cG_n}\,\cP(P_n,B(x,r+d))\geq\ep\frac{n}{L^d}\right]\leq \underline{\Prb}\Big[\max_{G_n\in\cG_n}\cP(P_n)\geq\beta\Big]\leq c_1\e^{-c_2n^{1-3/2d}}\,.
\end{align}
Let $\Lambda$ be the cube centered at $x$ of side length $2r$. We define $$\underline{\Lambda}=\Big\{\,\underline{y}\in\sZ^d:\, B_n(\underline{y})\subset\Lambda\,\Big\}\,.$$ As $B(x,r)\subset \Lambda$, we have $\underline{D}\subset \underline{\Lambda}$ and
\begin{align}\label{eq'4}
\frac{1}{|\underline{D}|}\sum_{\underline{y}\in\underline{D}}\ind_ {|\sC(\underline{y})|<\infty}\leq \frac{(2d)^d}{|\underline{\Lambda}|}\sum_{\underline{y}\in\underline{\Lambda}}\ind_ {|\sC(\underline{y})|<\infty}\,.
\end{align}
Let $q\in[0,1]$ be such that $\theta_q>1-\ep/(2(2d)^d)$. As  the family $(X(\underline{y}))_{\underline{y}\in \sZ^d}$ is identically distributed, has a short range dependency and is such that $\Prb(X(\underline{0})=1)$ goes to $1$ when $n$ goes to infinity (see for instance Chapter 9 in \cite{Cerf:StFlour}), then we can apply Liggett Schonmann and Stacey's result \cite{Liggett}: for $n$ large enough, the family $(X(\underline{y}),\underline{y}\in \sZ^d)$  stochastically dominates $(\widetilde{X}(\underline{y}),\underline{y}\in \sZ^d)$ a family of independent Bernoulli variable of parameter $q$. We denote by $\smash{\underline{\widetilde{\sC}}_\infty}$ the unique infinite cluster of the Bernoulli field $(\widetilde{X}(\underline{y}))_{\underline{y}\in \sZ^d}$.  Using inequality \eqref{eq'4} and the stochastic domination, we get
\begin{align*}
\Prb\left[\frac{1}{|\underline{D}|}\sum_{\underline{y}\in\underline{D}}\ind_ {|\sC(\underline{y})|<\infty}\geq \ep\right]&\leq \Prb\left[\frac{(2d)^d}{|\underline{\Lambda}|}\sum_{\underline{y}\in\underline{\Lambda}}\ind_ {|\sC(\underline{y})|<\infty}\geq \ep\right]\nonumber\\
&\leq \Prb\left[\frac{1}{|\underline{\Lambda}|}\sum_{\underline{y}\in\underline{\Lambda}}\ind_ {\underline{y}\notin  \widetilde{\underline{\sC}}_\infty}\geq \frac{\ep}{(2d)^d}\right]\nonumber\\
&\leq \Prb\left[\frac{\left|\underline{\Lambda}\cap \widetilde{\underline{\sC}}_\infty\right|} {|\underline{\Lambda}|}\notin\left(\theta_q-\frac{\ep}{2(2d)^d},\theta_q+\frac{\ep}{2(2d)^d}\right)\right]\,.
\end{align*}
Using Proposition \ref{Pizt}, we obtain
\begin{align}\label{ineq2}
\Prb\left[\frac{1}{|\underline{D}|}\sum_{\underline{y}\in\underline{D}}\ind_ {|\sC(\underline{y})|<\infty}\geq \ep\right]
\leq c'_1\exp\left(-c'_2\left(\frac{rn}{L}\right)^{d-1}\right)\,.
\end{align}

%The cube process $(X(\underline{y}))_{\underline{y}\in \sZ^d}$ satisfies the following condition for $n$ large enough:
%$$\forall \underline{x}\in\sZ^d,\, \Prb[X(\underline{x})=0|X(\underline{z}),|\underline{x}-\underline{z}|_\infty\geq 3]\leq \delta$$
%where $\delta =\ep/(2d)^{d+1}$. Using the estimates of Lemma 9.11 in \cite{Cerf:StFlour},
%\begin{align}\label{ineq2}
%\Prb\left[\frac{1}{|\underline{D}|}\sum_{\underline{y}\in\underline{D}}\ind_ {X(\underline{y})=0}\geq \ep\right]\leq 3^d\exp\left(-\Lambda^*\left(\frac{\ep}{(2d)^d},\frac{\ep}{(2d)^{d+1}}\right)\left\lfloor \frac{n^d\cL^d(\Lambda)}{(6L)^d}\right\rfloor\right)
%\end{align}
%where $$\Lambda^* (\ep,\delta)=\ep\ln\frac{\ep}{\delta}+(1-\ep)\ln\frac{1-\ep}{1-\delta}$$
%is the Cram\'er transform of a Bernoulli variable with parameter $\delta$.

\noindent We set $\ep =\delta/(2c(r))$. Finally, thanks to inequalities \eqref{delta}, \eqref{eqbis1}, \eqref{ineq1} and \eqref{ineq2}, we have for $n\geq n_0(u(\ep))$
\begin{align*}
&\underline{\Prb}\left[\max_{G_n\in\cG_n}\,\sup_{f\in\fF_n}|\widetilde{\mu}_n(f\ind_{B(x,r)})-\nu_n(f\ind_{B(x,r)})|>\delta\right]\\
&\hspace{1.8cm}\leq \underline{\Prb}\left[\max_{G_n\in\cG_n}\,\sup_{f\in\fF_n}|\mu_n(f)-\widetilde{\mu}_n(f)|>\delta/2\right]\\
&\hspace{2.5cm}+\underline{\Prb}\left[\max_{G_n\in\cG_n}\,\sup_{f\in\fF_n}|\widetilde{\mu}_n(f\ind_{B(x,r)})-\nu_n(f\ind_{B(x,r)})|>c(r)\ep\right]\\
&\hspace{1.8cm}\leq C_1\exp(-C_2n)+  \frac{c'_1}{\theta_p}\exp\left(-c'_2\left(\frac{rn}{L}\right)^{d-1}\right)+c_1\e^{-c_2n^{1-3/2d}}\,.
\end{align*}
The result follows.
%Thanks to inequalities \eqref{ineq1} and \eqref{ineq2}, we have
%$$\limsup_{n\rightarrow \infty}\frac{1}{n^{\kappa(d)}}\ln \Prb[|\mu_n(f)-\nu_n(f)|>\ep]\leq -c$$
%where $c$ does not depend on $f$ nor $\ep$. 
%Let $g$ be a uniformly continuous, compactly supported function such that 
%$$|\fd(\mu_n,\nu_n)-|\mu_n(g)-\nu_n(g)||\leq \ep/2$$
%Thus,
%$$\Prb\left[\max _{G_n\in \cG_n}\fd(\mu_n,\nu_n)>\ep\right]\leq\Prb[|\mu_n(g)-\nu_n(g)|>\ep/2] $$
%and the result follows.
%The set of uniformly continuous functions $f$ on $B(x,r)$ such that $\|f\|_\infty \leq 1$ is compact. There exists a finite covering $(f_i)_{1\leq i \leq M}$ such that $(B(f_i,\delta/2))_{1\leq i \leq M}$  covers the set. Thus,
%\begin{align*}
%&\Prb\left[\sup_{f}|\mu_n(f\ind_{B(x,r)})-\nu_n(f\ind_{B(x,r)})|>\delta\right]\\
%&\hspace{3cm}\leq \sum _{i=1}^M \Prb \left[|\mu_n(f_i\ind_{B(x,r)})-\nu_n(f_i\ind_{B(x,r)})|>\delta/2\right] 
%\end{align*}
%and using \eqref{limsup},
%\begin{align}
%\limsup_{n\rightarrow\infty}\frac{1}{n}\ln(\Prb[\sup_{f}|\mu_n(f\ind_{B(x,r)})-\nu_n(f\ind_{B(x,r)})|>\delta])\leq -c_2
%\end{align}
%where the supremum is over all uniformly continuous function on $\sR^d$ such that $\|f\|_\infty \leq 1$.
%Finally, there exists a constant $C$ depending only on $r$, $p$ and $d$ such that
%$$\Prb\Big(\,\sup_{f}|\mu_n(f\ind_{B(x,r)})-\nu_n(f\ind_{B(x,r)})|>\delta\,\Big)\leq C\exp(-c_2n)\,.$$
\end{proof}

\section{Lower large deviations and shape Theorem}\label{sectionLLD}
\subsection{Closeness to the set of Wulff shapes}
The aim of this section is to prove Theorem \ref{prel}. 

\begin{proof}[Proof of Theorem \ref{prel}]
 Let $\ep>0$. Let $\xi>0$ that we will choose later depending on $\ep$.  We define $\lambda$ such that
$$1-\lambda=\frac{1}{1+\xi}\,.$$
 We denote by $W_\xi$:
$$W_\xi=\left\{ \nu_{W+x}: \begin{array}{c} x\in\sR^d,\,W\text{ is a dilate of $W_p$ such that}\\\cL^d((1-\lambda)W_p)\leq \cL^d(W)\leq \cL^d((1+2\xi)W_p)
\end{array}\right\}\,.$$  Let $u:]0,+\infty[\rightarrow ]0,+\infty[$ be a non-decreasing function such that $\lim_{t\rightarrow 0}u(t)=0$. Let $n\geq 1$. Let $\fF_n$ be a finite set of uniformly continuous function that satisfies for all $f\in\fF_n,$
$$ \|f\|_\infty\leq 1\qquad\text{ and }\qquad\forall x,y\in\sR^d,\, |f(x)-f(y)|\leq u(\|x-y\|_2)\,.$$ We define the weak neighborhood $\cV(\cW_\xi,\fF_n,\ep)$ of $\cW_\xi$ given $\fF_n$ and $\ep$ as
 $$\cV(\cW_\xi,\fF_n,\ep)=\Big\{\nu \in\cM(\sR^d)\,\,:\exists \mu\in\cW_\xi,\, \sup_{f\in\fF_n}|\nu(f)-\mu(f)|\leq \ep\,\Big\}\,.$$
Our goal is to show that $\mu_n$ is in the set $\cV(\cW_\xi,\fF_n,\ep)$ with high probability.

\noindent \textbf{Step \textit{(i)}:} Let $G_n\in\cG_n$. Thanks to Proposition \ref{propcontiguity}, the measures $\mu_n$ and $\nu_n$ associated with $P_n$ and $G_n$ are locally close to each other. In the following, it will be more convenient to work with the continuous object $P_n$ instead of $G_n$.  We can localize almost all the volume of $P_n$ in a random region that is a union of balls of constant radius. We follow the method in Chapter 17 in \cite{Cerf:StFlour}. We can cover $P_n$ in $\sR^d$, up to a small fractional volume, by a finite number of random disjoint balls of constant size. Thanks to the isoperimetric inequalities, we can then control the volume of $P_n$ outside of these balls.
Let $\delta>0$ be a real number that we will choose later. We denote by $X$:
$$X=\left\{x\in \sZ^d: \cL^d(B(x,1)\cap P_n)\geq \delta\right\}.$$
On the event $\big\{\,|\widetilde{\Gamma}_n|\leq \beta n^{d-1}\,\big\}$, the set $X$ is included in $B(0, \beta n^{d-2})$ and is therefore finite. 
As each point in $\sR^d$ belongs to at most $2^d$ balls among the $B(x,1)$, $x\in \sZ^d$, then using Proposition \ref{isop}
$$\delta |X| \leq \sum_{x\in X}\cL^d (B(x,1)\cap P_n)\leq 2^d \cL^d(P_n) \leq  2^d c_{iso} \cP(P_n) ^{\frac{d}{d-1}} \leq  2^d c_{iso} \beta  ^{\frac{d}{d-1}}$$
%à détailler
and finally $|X|\leq M$ where $M= 2^d c_{iso} \beta ^{\frac{d}{d-1}}/\delta$.
We now would like to control the volume of $P_n$ outside the balls $B(x,1)$ in $X$, \textit{i.e.}, to bound the measure of $P_n\setminus \bigcup_{x\in X}B(x,1)$. For $x\in \sZ^d \setminus X$, by the isoperimetric inequality in Proposition \ref{isop}, we obtain as in section 17 in \cite{Cerf:StFlour} 
\begin{align}\label{contpn}
\cL^d \left(P_n\setminus \bigcup_{x\in X}B(x,1)\right)&\leq \sum_{x\in \sZ^d\setminus X}\cL^d(P_n\cap B(x,1))\nonumber\\
&\leq \delta^{1/d} b_{iso}^{\frac{d}{d-1}} \sum_{x\in \sZ^d\setminus X} \cP(P_n,\mathring{B}(x,1))\nonumber\\
&=\delta^{1/d} b_{iso}^{\frac{d}{d-1}} \sum_{x\in \sZ^d\setminus X} \cH^{d-1}(\partial^* P_n\cap \mathring{B}(x,1))\nonumber\\
&\leq 2^d \delta^{1/d} b_{iso}^{\frac{d}{d-1}} \cH^{d-1}(\partial ^* (P_n))= 2^d \delta^{1/d} b_{iso}^{\frac{d}{d-1}}\cP(P_n)\nonumber\\
&\leq  2^d \delta^{1/d} b_{iso}^{\frac{d}{d-1}}\beta.
\end{align}
We note $\eta = 2^d \delta^{1/d} b_{iso}^{\frac{d}{d-1}} \beta$.
Therefore, if $\cP(P_n)\leq \beta $, then $X\subset B(0, \beta n^{d-2})$, $|X|\leq M $ and $\cL^d (P_n\setminus \cup_{x\in X}B(x,1))\leq \eta $.
 \begin{figure}[H]
\def\svgwidth{0.7\textwidth}
\begin{center}
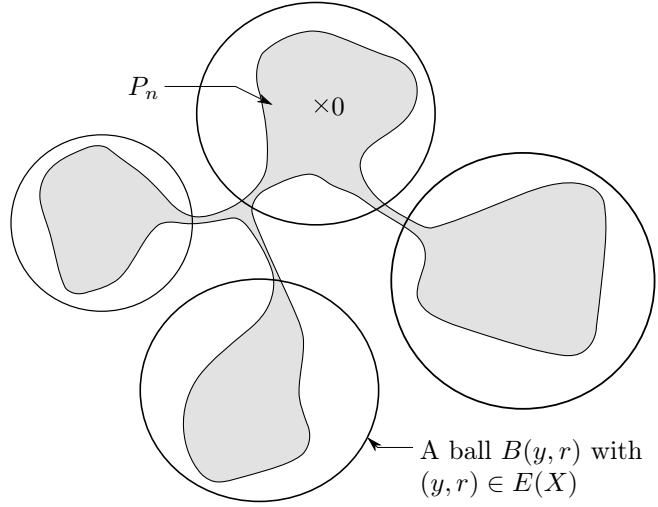
\caption[fig2]{\label{fig4}Covering almost all the volume of $P_n$ by balls of constant radius}
\end{center}
\end{figure}
We next would like to perform a kind of surgery between the balls. To do so, we first shall build from the balls $(B(x,1))_{x\in X}$ a family of balls that covers $\cup_{x\in X}B(x,1)$ and such that the balls are far apart (see Figure \ref{fig4}). This is the purpose of Lemma 17.1. in \cite{Cerf:StFlour}. We obtain a subset $$E(X)=\{(y_1,r_1),\dots,(y_m,r_m)\}\subset X\times \big\{\,1,\dots,3^{|X|}\,\big\}$$ such that $|E(X)|\leq |X|$ and 
\begin{description}[font=$\bullet$]
\item $\forall(a,r)\in E(X),\, B(a,r)\cap X\neq \emptyset$
\item $\cup_{x\in X}B(x,1)\subset \cup _{(a,r)\in E(X)} B(a,r)$
\item $\forall (a,r),\,(b,s)\in E(X),\, (a,r)\neq (b,s)\Rightarrow B(a,r+1)\cap B(b,s+1)=\emptyset$
\end{description}

 We set $$\varphi_{W_p}=\frac{\cI_p(W_p)}{\theta_p\cL^d(W_p)}\,.$$ 
 Let $\delta'>0$ be a real number that we will choose later. By applying Corollary \ref{cor} and Theorem \ref{ULD}, we obtain by conditioning on $E(X)$,
\begin{align}\label{eqE'}
&\Prb\left[\exists G_n\in\cG_n,\, \mu_n \notin \cV(\cW_\xi,\fF_n,\ep)\,\Big |\,0\in\sC_\infty\right]\nonumber\\
&\hspace{.8cm}\leq \underline{\Prb}[\max_{G_n\in\cG_n} \cP(nP_n)\geq \beta n^{d-1}]+\underline{\Prb}[n\varphi_n> (1+\delta')\varphi_{W_p}]\nonumber\\
&\hspace{1cm}+\underline{\Prb}[\exists G_n\in\cG_n,\, \mu_n \notin \cV(\cW_\xi,\fF_n,\ep), \, \cP(P_n)\leq \beta , n\varphi_n\leq (1+\delta')\varphi_{W_p}]\nonumber\\
&\hspace{.8cm}\leq b_1\exp(-b_2 n^{1-3/2d})+ b'_1\exp(-b'_2n)\nonumber \\
&\hspace{1cm}+ \sum_{1\leq m\leq M}\sum_{y_1,\dots,y_m}\sum_{r_1,\dots,r_m} \underline{\Prb}\left[\begin{array}{c}\exists G_n\in\cG_n,\, \mu_n \notin \cV(\cW_\xi,\fF_n,\ep),\\ E(X)=\{(y_1,r_1),\dots,(y_m,r_m)\}\\\cP(P_n)\leq \beta,\, n\varphi_n\leq (1+\delta')\varphi_{W_p} \end{array}\right]\,,
\end{align}
where the second summation is over $y_1,\dots,y_m$ in $\sZ^d \cap B(0, \beta n^{d-2})$ and the third summation is over $r_1,\dots,r_m$ in $\{1,\dots,3^M\}$. The number of ways to choose $m$ and $r_1,\dots,r_m$ is bounded from above by a constant depending only on $M$, while the number of ways of choosing the centers $y_1,\dots, y_m$ is polynomial in $n$. We next control the probability inside the sums. We will only focus on what happens inside the balls.

\noindent\textbf{Step \textit{(ii)}:} Let $\{(y_1,r_1),\dots,(y_m,r_m)\}$ be a value for the random set $E(X)$ which occurs with positive probability. We define $\Omega=\Omega(E(X))$ as 
$$\Omega=\bigcup_{i=1}^m \mathring{B}(y_i,r_i+1)\,,$$
and the restriction  $\overline{P}_n$ of $P_n$ to the balls determined by $E(X)$: $$\overline{P}_n=P_n\cap \left(\bigcup_{i=1}^m \mathring{B}(y_i,r_i+1)\right)\,.$$ 
Thus, using inequality \eqref{contpn}, we have
\begin{align}\label{contpbarn}
\cL^d(P_n\setminus \overline{P}_n)\leq \eta\,. 
\end{align}
We show now that $\nu_{\overline{P}_n}(f)$ is close to $\mu_n(f)$ with high probability on the event $$\big\{\,E(X)=\{(y_1,r_1),\dots,(y_m,r_m)\}\,\big\}\,.$$
It is easy to check that $\fF_n\cup\{1\}$ associated with the function $u$ satisfies the conditions required in Proposition \ref{propcontiguity}. So that applying Proposition \ref{propcontiguity} for every $r\in\{1,\dots,3^M\}$, there exist positive constants $c_1$, $c_2$ depending on $M$, $u$, and $\delta$ such that for all $x\in\sZ^d$
\begin{align*}
&\max_{r\in\{1,\dots,3^M\}}\underline{\Prb}\left[\max_{G_n\in\cG_n}\,\sup_{f\in\fF_n\cup\{1\}}|\nu_{n}(f\ind_{B(x,r)})-\mu_n(f\ind_{B(x,r)})|>\frac{\eta}{M}\right]\\
&\hspace{9cm}\leq c_1\e^{-c_2n^{1-3/2d}}\,.
\end{align*}
Thus, using inequality \eqref{contpbarn}, we obtain
\begin{align}\label{eq262}
&\underline{\Prb}\left[\max_{G_n\in\cG_n}\,\sup_{f\in\fF_n\cup\{1\}}|\nu_{\overline{P}_n}(f)-\mu_n(f)|>2\eta,\,E(X)=\{(y_1,r_1),\dots,(y_m,r_m)\}\right]\nonumber\\
&\hspace{0.5cm}\leq \sum_{i=1}^m\underline{\Prb}\left[\max_{G_n\in\cG_n}\,\sup_{f\in\fF_n\cup\{1\}}|\nu_{\overline{P}_n}(f\ind_{B(y_i,r_i)})-\mu_n(f\ind_{B(y_i,r_i)})|>\eta/M\right]\nonumber\\
&\hspace{0.5cm}\leq M\max_{r\in\{1,\dots,3^M\}}\underline{\Prb}\left[\max_{G_n\in\cG_n}\,\sup_{f\in\fF_n\cup\{1\}}|\nu_{n}(f\ind_{B(y_1,r)})-\mu_n(f\ind_{B(y_1,r)})|>\eta/M\right]\nonumber\\
&\hspace{0.5cm}\leq Mc_1\e^{-c_2n^{1-3/2d}}\,.
\end{align}
In particular, on the event $\{E(X)=\{(y_1,r_1),\dots,(y_m,r_m)\}\}$, with probability at least $1- Mc_1\exp(-c_2n^{1-3/2d} )$, we have
\begin{align}\label{eq261}
\left|\theta_p\cL ^d(\overline{P}_n)-\frac{|G_n|}{n^d}\right|\leq 2\eta\,.
\end{align}
Moreover, by Lemma \ref{nottoobig}, there exist positive constants $\eta_1$, $D_1$ and $D_2$ such that
$$\underline{\Prb}\left[\min_{G_n\in\cG_n}|G_n|\leq \eta_1 n ^d\right]\leq D_1\exp(-D_2n^{(d-1)/2d)})\,.$$
%Let us set $\rho=\min(\xi^d,\ep/4,\delta'(\eta_1-2\eta\theta_p)/4,\eta_1/2\theta_p)$. 
We recall that $\eta$ is a function of $\delta$. We will choose $\delta$ small enough such that 
\begin{align}\label{conddelta1}
\eta\leq \min\left(\frac{\eta_1}{4},\frac{\xi^d}{2},\frac{\ep}{8},\frac{\eta_1}{3\theta_p}\right)\,.
\end{align}
Other conditions will be imposed later on $\delta$.

On the event $\left\{\min_{G_n\in\cG_n}|G_n|> \eta_1 n ^d\right\}$, using inequalities \eqref{eq261} and \eqref{conddelta1}, we obtain
\begin{align}\label{infP}
\cL^d(\overline{P}_n)\geq \frac{1}{\theta_p}\left(\dfrac{|G_n|}{n^d}-2\eta\right)\geq \frac{1}{\theta_p}\left(\eta_1-2\eta\right)\geq\frac{\eta_1}{2\theta_p}
\end{align}
and as $\cL^d(W_p)=1/\theta_p$, using inequality \eqref{eq261}, we have
\begin{align}\label{supP}
\cL^d(\overline{P}_n)\leq \frac{1}{\theta_p}\left(\dfrac{|G_n|}{n^d}+2\eta\right)\leq \frac{1}{\theta_p}(1+\xi^d)= \cL^d(W_p)(1+\xi^d)\leq \cL^d((1+\xi)W_p)\,.
\end{align}
For $\nu \in \cW_\xi$, we have
\begin{align*}
\sup_{f\in\fF_n}|\nu_{\overline{P}_n}(f)-\nu(f)|\geq \sup_{f\in\fF_n}|\mu_n(f)-\nu(f)|-\sup_{f\in\fF_n}|\mu_n(f)-\nu_{\overline{P}_n}(f)|\,,
\end{align*}
so that, together with inequalities \eqref{eq262} and \eqref{conddelta1}, with high probability,
\begin{align*}
\mu_n \notin \cV(\cW_\xi,\fF_n,\ep)\implies \nu_{\overline{P}_n} \notin \cV(\cW_\xi,\fF_n,3\ep/4)\,.
\end{align*}
Thus, combining with inequalities \eqref{infP} and \eqref{supP}, we have
\begin{align}\label{finn1}
&\underline{\Prb}\left[\begin{array}{c}\exists G_n\in\cG_n, \,\mu_n \notin \cV(\cW_\xi,\fF_n,\ep),\,\cP(P_n)\leq \beta\,,\\ E(X)=\{(y_1,r_1),\dots,(y_m,r_m)\},\,
 n\varphi_n\leq (1+\delta')\varphi_{W_p} \end{array}\right]\nonumber\\
 &\hspace{0.2cm}\leq \underline{\Prb}\left[\begin{array}{c}\exists G_n\in\cG_n, \,\nu_{\overline{P}_n}\notin \cV(\cW_\xi,\fF_n,3\ep/4),\,n\varphi_n\leq (1+\delta')\varphi_{W_p} ,\\ \frac{\eta_1}{2\theta_p}\leq\cL^d(\overline{P}_n)\leq  \cL^d((1+\xi)W_p),\,  E(X)=\{(y_1,r_1),\dots,(y_m,r_m)\}\end{array}\right]\nonumber\\
 &\hspace{0.3cm}+Mc_1\e^{-c_2n^{1-3/2d}}+D_1\exp(-D_2n^{(d-1)/2d})\,.
\end{align}
%We introduce the space $$\sC_\beta=\left\{\,F\text{ Borel subset of } \Omega,\,\cP\left(F,\Omega\right)\leq \beta\,\right\}$$endowed with the topology $L^1$ associated to the distance $d(F,F')=\cL^d(F\Delta F')$, where $\Delta $ is the symmetric difference between sets. For this topology, the space $\sC_\beta$ is compact. We now build a finite covering of $\sC_{\beta}$ which is uniform over the points $y_1,\dots,y_m$.
We do not cover $\overline{P}_n$ directly but we cover separately each $\overline{P}_n\cap B(y_k,r_k+1)$ for $k\in\{1,\dots,m\}$.
For any $r\in\{\,1,\dots,3^M\,\}$, we define the space $$\sC_\beta ^{(r)}=\Big\{\,F\subset \mathring{B}(0,r+1),\,\cP(F, \mathring{B}(0,r+1))\leq \beta\,\Big\}\,$$
endowed with the topology $L^1$ associated to the distance $d(F,F')=\cL^d(F\Delta F')$, where $\Delta $ is the symmetric difference between sets. For this topology, the space $\sC_\beta^{(r)}$ is compact.
Suppose that we associate to each $F\in\sC_\beta ^ {(r)}$ a positive number $\ep_F\leq \min(\eta,\cL^d(\xi W_p))/M$. The collection of open sets 
$$\Big\{\,H\text{ Borel subset of }\mathring{B}(0,r+1)\,:\,\cL^d(H\Delta F)<\ep_F\,\Big\},\,F\in\sC_\beta^{(r)},$$ is then an open covering of $\sC_\beta^{(r)}$. By compactness, we can extract a finite covering $(F_i^{(r)},\ep_{F_i^{(r)}})_{1\leq i\leq N^{(r)}}$ of $\sC_\beta^{(r)}$. 
%One can check that the following family is a finite covering for $\sC_\beta$:
%$$\Big(\,\bigcup _{k=1}^m \Big( F^{(r_k)}_{i_k}+y_k,\sum_{k=1}^m\ep_{F^{(r_k)}_{i_k}}\,\Big), 1\leq i_1\leq N^{(r_1)},\,\dots,\, 1\leq i_m\leq N^{(r_m)}\,\Big)$$
%that is uniform over the points $y_1,\dots,y_m$. To ease the notation, we will denote this finite covering of $\sC_\beta$ as $(F_i,\ep_{F_i})_{1\leq i \leq N}$. Notice that by construction $\ep_{F_i}\leq \eta$.
By union bound, we obtain
\begin{align}\label{finn2}
&\underline{\Prb}\left[\begin{array}{c}\exists G_n\in\cG_n, \,\nu_{\overline{P}_n}\notin \cV(\cW_\xi,\fF_n,3\ep/4),\,n\varphi_n\leq (1+\delta')\varphi_{W_p} ,\\ \frac{\eta_1}{2\theta_p}\leq\cL^d(\overline{P}_n)\leq  \cL^d((1+\xi)W_p),\,  E(X)=\{(y_1,r_1),\dots,(y_m,r_m)\}\end{array}\right]\nonumber\\
&\hspace{6cm}\leq\sum_{i_1=1}^ {N^{(r_1)}}\cdots \sum_{i_m=1}^ {N^{(r_m)}} \underline{\Prb}[\cF_{i_1,\dots,i_m}]
\end{align}
%We can extract a finite covering $((\nu_{F_i},\ep_{F_i}))_{1\leq i\leq N}$ of $\cP_{E(X),\xi}\setminus \cV(\cW_\xi,\ep/4)$ that is compact for the distance $\fd$.
%For every $i\in\{1,\dots,N\}$, we can choose $\ep_{F_i}$ small enough such that 
%\begin{align}\label{cond1}
%\ep_{F_i}\leq \delta'\frac{\eta_1-2\eta\theta_p}{4}
%\end{align}
%and
%$$\ep_{F_i}\leq \min\left(\frac{\ep}{2},\eta\right)\,.$$
 where
 $$\cF_{i_1,\dots,i_m}=\left\{\begin{array}{c}\exists G_n\in\cG_n\,:\,\forall \,1\leq k\leq m,\\\cL^d((F_{i_k}^{(r_k)}+y_k)\Delta(\overline{P}_n\cap B(y_k,r_k+1)))\leq \ep_{F_{i_k}^{(r_k)}},\\ \nu_{\overline{P}_n}\notin \cV(\cW_\xi,\fF_n,3\ep/4),\,n\varphi_n\leq (1+\delta')\varphi_{W_p} ,\\ \frac{\eta_1}{2\theta_p}\leq\cL^d(\overline{P}_n)\leq  \cL^d((1+\xi)W_p),\\  E(X)=\{(y_1,r_1),\dots,(y_m,r_m)\}\end{array}\right\}.$$
 So we need to study the quantity $\underline{\Prb}[\cF]$ for a generic $m$-uplet $(F_1,\dots,F_m)\in \sC_\beta ^{(r_1)}\times \dots \times  \sC_\beta ^{(r_m)}$ and their associated $\ep_{F_1},\dots, \ep_{F_m} $.
 By definition of the Cheeger constant $\varphi_n$, we obtain
\begin{align*}
\underline{\Prb}[\cF]&=\underline{\Prb}\left[\exists G_n\in\cG_n\,:\,\begin{array}{c}\forall 1\leq i \leq m,\\\cL^d((F_i+y_i)\Delta(\overline{P}_n\cap B(y_i,r_i+1)))\leq \ep_{F_i},\\ \nu_{\overline{P}_n}\notin \cV(\cW_\xi,\fF_n,3\ep/4),\\ |\partial^oG_n|\leq (1+\delta')n^{-1}|G_n|\varphi_{W_p}  ,\\ \frac{\eta_1}{2\theta_p}\leq\cL^d(\overline{P}_n)\leq  \cL^d((1+\xi)W_p),\\  E(X)=\{(y_1,r_1),\dots,(y_m,r_m)\}\end{array}\right]\,.
\end{align*}
To lighten the notations, we set 
$$F=\bigcup_{i=1}^m (F_i+y_i)\,.$$ We have
\begin{align}\label{ep_F}
\cL^d(F\Delta \overline{P}_n)&=\sum_{i=1}^m \cL ^d((\overline{P}_n \cap B(y_i,r_i+1))\Delta(F_i+y_i))\nonumber\\
&\leq \sum_{i=1}^ m \ep_{F_i}\leq \min (\eta, \cL^d(\xi W_p))\,.
\end{align}
Whereas the surface tension of $F$ in the interior of these balls corresponds to the surface tension of our minimizer $G_n$, the surface tension of $F$ along the boundary of the balls $B(y_j,r_j+1)$ does not correspond to the surface tension of $G_n$ because we have artificially created it. Roughly speaking, $F$ is the continuous object corresponding to the graph $G_n$ intersected with the $nB(y_j,r_j+1)$. This new graph has extra surface tension compared to $G_n$ due to the fact that we have built it by cutting $G_n$ along the boundary of these balls. However, our hope is to cut along the boundary of these balls in such a way that the surface tension we create is negligible. We do not work on $G_n$ but on the continuous object $F$, but we have to keep in mind that these two objects are close. The idea is to cut $F$ in the regions $B(y_i,r_i+1)\setminus B(y_i,r_i)$, $i\in \{1,\dots,m\}$. These regions contain a negligible volume of $G_n$ and so of $F$, we want to cut $F$ in these regions along a surface of negligible perimeter and so of negligible surface tension.
By Lemma 14.4 in \cite{Cerf:StFlour}, for $i\in \{1,\dots,m\}$, for $\cH^1$ almost all $t$ in $]0,1[$,
\begin{align}\label{eq1}
\cI (F\cap B(y_i,r_i+t))\leq \cI (F\cap \mathring{B}(y_i,r_i+t))+ \beta_{max}\cH^{d-1}(F\cap \partial B(y_i,r_i+t))\,.
\end{align}
Let $T$ be a subset of $]0,1[$ where all the above inequalities hold simultaneously. We recall that for any $i\in\{1,\dots,m\}$, $\ep_{F_i}\leq\eta/M$. We have $\cH^1 (T)=1$ and when we integrate in polar coordinates, using inequality \eqref{ep_F},
\begin{align*}
\int_T \sum_{i=1}^m \cH^{d-1}(F\cap \partial B(y_i,r_i+t)) dt& = \sum_{i=1}^m\cL ^d(F\cap B(y_i,r_i+1)\setminus B(y_i,r_i)) \\ 
& \leq \sum_{i=1}^m \cL^d((F_i+y_i)\setminus B(y_i,r_i))\\
& \leq \cL^d\Big( \overline{P}_n \setminus \bigcup_{i=1}^m B(y_i,r_i)\Big)+\cL ^d(\overline{P}_n\Delta F)\\
&\leq 2\eta\,.
\end{align*}
Thus, there exists $t\in T$ such that 
\begin{align}\label{eq3b}
\sum_{i=1}^m \cH^{d-1}(F\cap \partial B(y_i,r_i+t))  \leq 3\eta\,. 
\end{align}
We next set 
$$\overline{F}=F\cap \left(\bigcup_{i=1}^m B(y_i,r_i+t)\right)\,.$$%%\text{ and } \overline{\nu}=\nu_{\overline{F}}.$$
Using inequality \eqref{eq3b}, we get
\begin{align}\label{pbarf}
\cP(\overline{F})&\leq \cP\left(\overline{F},\bigcup_{i=1}^m \mathring{B}(y_i,r_i+t)\right)+ \sum_{i=1}^m \cH^{d-1}(F\cap \partial B(y_i,r_i+t)) \nonumber\\
&\leq \cP\left(\overline{F},\bigcup_{i=1}^m \mathring{B}(y_i,r_i+t)\right)+3\eta\,,
\end{align}
and using Proposition \ref{propI},
\begin{align}\label{eqIp}
\cI_p(\overline{F})\leq \cI_p\left(\overline{F}, \bigcup_{i=1}^m \mathring{B}(y_i,r_i+t)\right)+3\beta_{max}\eta\,.
\end{align}
On the event $\cF$, using inequality \eqref{ep_F}, we obtain 
\begin{align}\label{ldf}
\cL^d(F)&\leq \cL^d(\overline{F})+\cL^d\left(F\setminus \bigcup _{i=1}^m B(y_i,r_i)\right)\nonumber\\
&\leq \cL^d(\overline{F})+\cL^d(F\Delta \overline{P}_n)+\cL^d\left(\overline{P}_n\setminus \bigcup _{i=1}^m B(y_i,r_i)\right)\leq  \cL^d(\overline{F})+2\eta\,.
\end{align}
Finally, using inequalities \eqref{pbarf} and \eqref{ldf}, we obtain
\begin{align}\label{eqIp2}
\cI_p\left(\overline{F},\bigcup_{i=1}^m \mathring{B}(y_i,r_i+t)\right)&\geq \beta_{min}\,\cP\left(\overline{F},\bigcup_{i=1}^m \mathring{B}(y_i,r_i+t)\right)\nonumber\\
&\geq \beta_{min} (\cP(\overline{F})-3\eta)\,.
\end{align}
and using again inequality \eqref{ep_F},
\begin{align}\label{eqIpp}
\cL^d(F)\geq \cL^d(\overline{P}_n)-\cL^d(\overline{P}_n\Delta F)\geq \frac{\eta_1}{2\theta_p}-\eta\,.
\end{align}
Using the isoperimetric inequalities of Proposition \ref{isop} and inequalities \eqref{ldf} and \eqref{eqIpp}, we get
\begin{align}\label{eqIp3}
\cP(\overline{F})&\geq \left(\frac{\cL^d(\overline{F})}{c_{iso}}\right)^{1-1/d}\geq \left(\frac{\cL^d(F)-2\eta}{c_{iso}}\right)^{1-1/d}\geq\left(\frac{\eta_1-6\eta\theta_p}{2\theta_p c_{iso}}\right)^{1-1/d}\,.
\end{align} 
Next, we choose $\delta$ small enough to obtain a $\eta$ that satisfies the following inequalities:
\begin{align}\label{conddelta11}
3\beta_{max}\eta\leq \frac{\lambda\beta_{min}}{2}\left(\left(\frac{\eta_1-6\eta\theta_p}{2\theta_p c_{iso}}\right)^{1-1/d}-3\eta\right)\,, 
\end{align}
and also $$\eta_1\geq 6\eta\theta_p\,.$$ With this choice of $\delta$, we obtain with high probability, using inequalities \eqref{eqIp}, \eqref{eqIp2} and \eqref{eqIp3},
\begin{align}\label{eql}
\cI_p(\overline{F})\leq (1+\lambda/2)\cI_p\left(\overline{F}, \bigcup_{i=1}^m \mathring{B}(y_i,r_i+t)\right)\leq (1+\lambda/2)\cI_p(F,\Omega) \,.
\end{align}
% &\leq \Prb\left[\begin{array}{c}\exists G_n\in\cG_n, \fd(\nu_n,\nu_{F})\leq \ep_{F},\, \\  E(X)=\{(y_1,r_1),\dots,(y_m,r_m)\},\,
% n|\partial^oG_n|\leq (1+\delta')n^d(\theta \cL^d(F_j)+\ep_{F})\varphi_{\cW_p} \end{array}\right]\\%a verifier
%  &\leq \Prb\left[\begin{array}{c}\exists G_n\in\cG_n, \fd(\nu_n,\nu_{F})\leq \ep_{F},\, \\  E(X)=\{(y_1,r_1),\dots,(y_m,r_m)\},\,
 %|\partial^oG_n|\leq (1+\delta')^2n^{d-1}\theta \cL^d(F)\varphi_{\cW_p} \end{array}\right]
Let $G_n\in\cG_n$, on the event $\cF$, we have $$\left|\theta_p\cL^d (\overline{P}_n)- \dfrac{|G_n|}{n^d}\right|\leq 2\eta\,.$$ So that, together with inequality \eqref{ldf},
\begin{align*}
|G_n|&\leq n^{d}(\theta_p \cL^d(F)+\theta_p\cL^d(\overline{P}_n\Delta F)+2\eta)\\
&\leq n^{d}(\theta_p \cL^d(\overline{F})+\ep_F+4\eta)\\
&\leq n^{d}\theta_p \cL^d(\overline{F})\left(1+\frac{5\eta}{\theta_p\cL^d(\overline{F})}\right)\,.
\end{align*}
Let us now choose $\delta$ small enough so that
\begin{align}\label{conddelta2}
\frac{5\eta}{\eta_1/2-3\eta\theta_p}\leq \delta'\,.
\end{align}
Using inequalities \eqref{ldf} and \eqref{eqIpp}, we obtain
\begin{align*}
|G_n|&\leq n^{d}\theta_p \cL^d(\overline{F})\left(1+\frac{5\eta}{\eta_1/2-3\eta\theta_p}\right)\\
&\leq n^{d}\theta_p \cL^d(\overline{F})(1+\delta').
\end{align*}
Finally, let $r$ be such that $\cL^d(\overline{F})=\cL^d(rW_p)$, we get
\begin{align*}
(1+\delta')n^{-1}|G_n|\varphi_{W_p}
&\leq (1+\delta')^2n^{d-1} \frac{\varphi_{W_p} }{\varphi_{\overline{F}}}\cI_p(\overline{F}) \\
&\leq (1+\delta')^2\frac{ \cI_p(rW_p)}{\cI_p(\overline{F})}rn^{d-1} \cI_p(\overline{F}) \,.
\end{align*}
We now choose $\delta'$ small enough such that
\begin{align}\label{conddelta'}
(1+\delta')^2 (1-\lambda)\leq 1-\frac{\lambda}{2}\,.
\end{align}
Using inequality \eqref{ep_F}, we obtain  $$\cL^d(\overline{F})\leq \cL^d(\overline{P}_n)+\cL^d(\overline{P_n}\Delta F)\leq \cL ^d((1+\xi )W_p)+\cL ^d(\xi W_p)\leq \cL ^d((1+2\xi )W_p)\,$$
and so $r\leq 1+2\xi$.
We distinguish now two cases:\\
$\bullet$ If $r\leq 1-\lambda$, using inequality \eqref{conddelta'}
\begin{align*}
(1+\delta')^2\frac{ \cI_p(rW_p)}{\cI_p(\overline{F})}rn^{d-1} \cI_p(\overline{F}) \leq (1-\lambda/2)n^{d-1} \cI_p(\overline{F})
\end{align*}
where we used the fact that the Wulff crystal is a minimizer for $\cI_p$, \textit{i.e.}, that $\cI_p(rW_p)\leq \cI_p(\overline{F})$.\\
$\bullet$ Let us assume that $r\in(1-\lambda,1+2\xi]$. We recall that on the event $\cF$, for all $\nu\in\cW_\xi$, $$\sup_{f\in\fF_n}|\nu_{\overline{P}_n}(f)-\nu(f)|\geq 3\ep/4\,.$$ 
Thus, for all $x\in\sR^d$, for $f\in\fF_n$ we have
\begin{align*}
|\nu_{\overline{P}_n}(f)-\nu_{rW_p+x}(f)|&\leq\left| \int_{\overline{P}_n\setminus (rW_p+x)}f(x)d\cL ^d(x)-\int_{(rW_p+x)\setminus \overline{P}_n}f(x)d\cL ^d(x)\right|\\
&\leq \int_{\overline{P}_n\setminus (rW_p+x)}|f(x)|d\cL ^d(x)+\int_{(rW_p+x)\setminus \overline{P}_n}|f(x)|d\cL ^d(x)\\
&\leq \int_{\overline{P}_n\setminus (rW_p+x)}1d\cL ^d(x)+\int_{(rW_p+x)\setminus \overline{P}_n}1 d\cL ^d(x)\\
&\leq \cL^d\big(\overline{P}_n\Delta (rW_p+x)\big)\,,
\end{align*}
and so,
\begin{align*}
\cL^d\big(\overline{P}_n\Delta (rW_p+x)\big)\geq \sup_{f\in\fF_n}|\nu_{\overline{P}_n}(f)-\nu_{rW_p+x}(f)|\geq 3\ep/4
\end{align*}
and as $\eta$ satisfies inequality \eqref{conddelta1}, we obtain
\begin{align*}
\cL^d\big(\overline{F}\Delta (rW_p+x)\big)\geq \cL^d\big(\overline{P}_n\Delta (rW_p+x)\big)-\cL^d\big(\overline{F}\Delta\overline{P}_n\big)\geq 3\ep/4-\eta\geq \ep/2.
\end{align*}
Moreover, as $rW_p$ is a minimizer for the isoperimetric problem, there exists a constant $c(\ep)>0$, that is a non-decreasing function of $\ep$ depending also on $p$ and $r$, that goes to $0$ when $\ep$ goes to $0$, such that
\begin{align*}
\inf\Big\{\,\cI(E)\,: \,\forall x\in\sR^d,\,\cL^d(E\Delta(x+rW_p))\geq \ep/2,\,\cL^d(E)=\cL^d(rW_p)\,\Big\}\\\geq\cI_p(rW_p)(1+c(\ep))\,.
\end{align*}
Finally, 
$$\frac{ \cI_p(rW_p)}{\cI_p(\overline{F})}\leq \frac{1}{1+c(\ep)}\,$$
and so,
$$(1+\delta')^2\frac{ \cI_p(rW_p)}{\cI_p(\overline{F})}rn^{d-1} \cI_p(\overline{F}) \leq \frac{(1+\delta')^2}{1+c(\ep)} (1+2\xi)n^{d-1} \cI_p(\overline{F})\,.$$
We choose $\xi$ small enough depending on $\ep$ such that 
$$\frac{1+2\xi}{1+c(\ep)}\leq 1-\lambda =\frac{1}{1+\xi}\,.$$
This is equivalent to choose $\xi$ such that 
\begin{align}\label{condxi}
3\xi +2\xi^2 \leq c(\ep)\,.
\end{align}
 %To ensure the existence of such $\ep$, we may restrict ourself to the study of $\xi$ strictly smaller than a deterministic constant depending on the supremum of the function $c$. This is not an issue as we are only interested in small values of $\xi$. Note that we can choose $\ep$ such that $\ep$ goes to $0$ when $\xi$ goes to $0$.
We obtain using inequality \eqref{conddelta'}
$$(1+\delta')^2\frac{ \cI_p(rW_p)}{\cI_p(\overline{F})}rn^{d-1} \cI_p(\overline{F}) \leq (1-\lambda/2)n^{d-1} \cI_p(\overline{F})\,.$$
Finally, combining the two cases, with $\ep$ and $\delta'$ properly chosen and inequality \eqref{eql}, we obtain
\begin{align}\label{eq3}
\underline{\Prb}[\cF]\leq \underline{\Prb}\left[\exists G_n\in\cG_n : \begin{array}{c}\forall 1\leq i\leq m,\\\cL^d((\overline{P}_n\cap B(y_i,r_i+1))\Delta (F_i+y_i))\leq \ep_{F_i}, \\ 
 |\partial^oG_n|\leq \left(1-\frac{\lambda^2}{4}\right)n^{d-1}\cI_p(F,\Omega), \\E(X)=\{(y_1,r_1),\dots,(y_m,r_m)\} \end{array}\right]\,.
 \end{align}

%We recall that $\cL^d (P_n\setminus \cup_{x\in X}B(x,1))\leq \eta$. Thus, by construction we get $$\cL^d \left(P_n\setminus \bigcup_{1\leq i\leq m}B(y_i,r_i)\right)\leq \eta$$ and so  $\fd(\nu_n, \overline{\nu}_n)\leq \theta_p \eta$. We choose $\delta$ small enough such that $\theta_p\eta \leq \ep /2$. Moreover, if $\cP(P_n)\leq \beta$ then $\cP(\overline{P}_n)\leq \beta +\alpha_{d-1}m$.
\noindent\textbf{Step \textit{(iii)}:} The remaining of the proof follows the same ideas as in \cite{CerfTheret09infc}. We link the probability defined in the right hand side of \eqref{eq3} with the probability that the flow is abnormally small in some local region of $\partial F\cap \Omega\,.$ We now want to cover $\partial F$ by balls of small radius such that $\partial {F}$ is "almost flat" in each ball, this is the purpose of the following Lemma:
\begin{lem}\label{lemball} [Lemma 1 in \cite{CerfTheret09infc}] Let $R>0$. Let $F$ be a subset of $\mathring{B}(0,R)$ of finite perimeter. For every positive constants $\delta'$ and $\eta'$, there exists a finite family of closed disjoint balls $(B(x_i,\rho_i))_{i\in I \cup K}$ and vectors $(v_i)_{i\in I \cup K}$, such that, letting $B_i=B(x_i,\rho_i)$ and $B_i^-=B^-(x_i,\rho_i,v_i)$, we have for all $i\in I$
$$ x_i\in \partial^* F \cap \mathring{B}(0,R), \, \rho_i\in ]0,1[, \, B_i\subset \mathring{B}(0,R),\, \cL^d((F\cap B_i)\Delta B_i^-)\leq \delta' \alpha_d \rho_i^d,$$
and
$$\left|\cI_p(F,\mathring{B}(0,R))- \sum_{i\in I} \alpha_{d-1}\rho_i^{d-1}(\nu(n_F(x_i))\right |\leq \eta'. $$
\end{lem} 
We apply Lemma \ref{lemball} to each $F_k\subset \mathring{B}(0,r_k+1)$, with $\delta_2>0$ that will be chosen later and $\eta' =\lambda^4 \cI_p(F,\Omega)/16M$. We obtain for each $k$, a family $\big(B^{(k)}_i\big(x^{(k)}_i,\rho^{(k)}_i,
v^{(k)}_i\big)\big)_{i\in I^{(k)}}$ that does not depend on $y_1,\dots,y_m$, so that 
\begin{align}\label{pasdidee}
\left|\cI_p(F_k,\mathring{B}(0,r_k+1))- \sum_{i\in I} \alpha_{d-1}(\rho_i^{(k)})^{d-1}(\nu(n_{F_k}(x^{(k)}_i))\right |\leq \eta'. 
\end{align}
 We now choose 
\begin{align}\label{condep} 
 \ep_{F_k}\leq\min\left(\,\min_{i\in I^{(k)}}\alpha_d (\rho^{(k)}_i)^d\delta_2,\,\frac{\eta}{M},\,\frac{\cL^d(\xi W_p)}{M}\,\right)\,,
\end{align} 
  for a fixed $\delta_2$ that we will choose later. Besides, as the balls $B(y_k,r_k+1)$ are disjoint, for $k\in\{1,\dots,m\}$, we have
$$\cI_p(F,\Omega)=\sum_{k=1}^m \cI_p 
(F\cap B(y_k,r_k+1),\Omega)=\sum_{k=1}^m \cI_p 
(F_k,\mathring{B}(0,r_k+1))\,.$$ Using inequality \eqref{pasdidee}, we obtain 
 $$\left|\cI_p(F,\Omega)- \sum_{k=1}^m\sum_{i\in I^{(k)}} \alpha_{d-1}(\rho_i^{(k)})^{d-1}\nu(n_{F_k}(x_i^{(k)}))\right |\leq m\eta'\leq \lambda^4 \cI_p(F,\Omega)/16\,.$$ 
 So, we get
 $$\cI_p(F,\Omega)\leq \frac{1}{1-\lambda^4 /16}\left(\sum_{k=1}^m\sum_{i\in I^{(k)}} \alpha_{d-1}(\rho_i^{(k)})^{d-1}\nu(n_{F_k}(x_i^{(k)}))\right)$$
 and 
$$\left(1-\dfrac{\lambda^2}{4}\right)\cI_p(F,\Omega)\leq \frac{1-\lambda ^2/4}{1- \lambda^4/16}\left(\sum_{k=1}^m\sum_{i\in I^{(k)}} \alpha_{d-1}(\rho_i^{(k)})^{d-1}\nu(n_{F_k}(x_i^{(k)}))\right)\,.$$ 
Whence setting  $w=\lambda^2/(4+\lambda^2)<1$,
\begin{align}\label{fin1}
\left(1-\dfrac{\lambda^2}{4}\right)\cI_p(F,\Omega)\leq (1- w)\left(\sum_{k=1}^m\sum_{i\in I^{(k)}} \alpha_{d-1}(\rho_i^{(k)})^{d-1}\nu(n_{F_k}(x_i^{(k)}))\right)\,.
\end{align}
Since the balls $(B^{(k)}_i+y_k)_{1\leq k\leq m,\, i\in I^{(k)}}$ are pairwise disjoint, we have
\begin{align}\label{fin2}
|\partial^o G_n|\geq \sum_{k=1}^m\sum_{i\in I^{(k)}}|(\partial ^o G_n)\cap (n(B_i^{(k)}+y_k))|\,.
\end{align}
Using inequalities \eqref{fin1} and \eqref{fin2}, we get
\begin{align}\label{eqGint}
&\underline{\Prb}\left[\exists G_n\in\cG_n,\begin{array}{c} \cL^d((\overline{P}_n\cap B(y_i,r_i+1))\Delta (F_i+y_i))\leq \ep_{F_i},\,1\leq i\leq m, \\  |\partial^oG_n|\leq (1-\lambda^2/4)n^{d-1} \cI_p(F,\Omega),\\E(X)=\{(y_1,r_1),\dots,(y_m,r_m)\}\end{array} \right]\nonumber\\
&\hspace{0.1cm}\leq \underline{\Prb}\left[\begin{array}{c} \exists G_n\in\cG_n,\,\cL^d((\overline{P}_n\cap B(y_i,r_i+1))\Delta (F_i+y_i))\leq \ep_{F_i},\,1\leq i\leq m,\\
\sum_{k=1}^m\sum_{i\in I^{(k)}} |(\partial ^oG_n)\cap (n(B_i^{(k)}+y_k))|\hfill\\\hfill\leq (1- w)n^{d-1}\left(\sum_{k=1}^m\sum_{i\in I^{(k)}} \alpha_{d-1}(\rho_i^{(k)})^{d-1}\nu(n_{F_k}(x_i^{(k)}))\right)\end{array}\right]\,.
\end{align}
Let $k\in\{1,\dots,m\}$. We aim to control $\card ((G_n\cap n (B_i^{(k)}+y_k))\Delta (n (B_i^ {(k)}+y_k)^-\cap\sZ^d) )$.  To do so, it is more convenient to work with the graph $F_n$. In the following, we drop the superscript $(k)$ for clarity. With high probability, we have
\begin{align*}
&\card ((G_n\cap n (B_i^{(k)}+y_k))\Delta (n (B_i^ {(k)}+y_k)^-\cap\sZ^d) )\\
&\hspace{2cm}\leq \card ((F_n\cap n(B_i+y_k))\Delta (n(B_i+y_k)^-\cap \sZ^d))+\card(F_n\setminus G_n)\\
&\hspace{2cm}\leq  \card ((F_n\cap n(B_i+y_k))\Delta (n(B_i+y_k)^-\cap \sZ^d))+\eta_3n^{d-1/2(d-1)}\,.
\end{align*}
As $B_i+y_k\subset B(y_k,r_k+1)$, we have
\begin{align*}
\cL^d((nP_n\cap n (B_i+y_k))\Delta (n (B_i+y_k)^-))&\leq \cL^d((nF_k\cap n B_i)\Delta (n B_i^-))\\
&\hspace{0.4cm}+n^d\cL^d(P_n\Delta (F_k+y_k))\\
&\leq  n^d\alpha_d \rho_i^d\delta_2+ \ep_{F_k}\leq  2n^d\alpha_d \rho_i^d\delta_2\,.
\end{align*}
By the same arguments as in section 5.2 in  \cite{CerfTheret09infc},
\begin{align*}
\card& ((F_n\cap n(B_i+y_k))\Delta n(B_i+y_k)^-))\\
&\hspace{1.5cm}\leq \cL^d(((nP_n\cap n (B_i+y_k))\Delta n (B_i+y_k)^-)\cap\sZ^d+[-1/2,1/2]^d)\\
&\hspace{1.5cm}\leq 2 n^d\alpha_d \rho_i^d\delta_2+ n^{d-1} 4d(\cH^{d-1}(\partial B_i)+\cH^{d-1}(\partial B_i^-))\,.
\end{align*}
Finally, for $n$ large enough,
$$\card ((G_n\cap n( B_i+y_k))\Delta (n (B_i+y_k)^-\cap\sZ^d))\leq 4n^d\alpha_d \rho_i^d\delta_2\,.$$
Thus, using inequality \eqref{eqGint}, for large enough $n$, 
\begin{align}
&\underline{\Prb}\left[\begin{array}{c}\exists G_n\in\cG_n,\cL^d((\overline{P}_n\cap B(y_i,r_i+1))\Delta (F_i+y_i))\leq \ep_{F_i},\,1\leq i\leq m,\, \\ |\partial^oG_n|\leq (1-\lambda^2/4)n^{d-1} \cI_p(F,\Omega),\,E(X)=\{(y_1,r_1),\dots,(y_m,r_m)\}\end{array} \right]\nonumber\\
&\hspace{0.5cm} \leq\sum_{k=1} ^m\sum_{i\in I^{(k)}}\underline{\Prb}\left[\begin{array}{c}\exists G_n\in\cG_n, \,\\\big|(G_n\cap n (B_i+y_k))\Delta (n (B_i^-+y_k)\cap\sZ^d)\big|\leq 4\delta_2 \alpha_d \rho_i^d n^d,\\  |(\partial ^oG_n)\cap n(B_i+y_k)|\hfill\\ \hfill \leq (1- w)n^{d-1}\left(\alpha_{d-1}\rho_i^{d-1}\nu(n_{F_k}(x_i^{(k)}))\right)\end{array}\right]\nonumber\\
&\hspace{0.5cm} \leq\frac{1}{\theta_p}\sum_{k=1} ^m\sum_{i\in I^{(k)}}\Prb[G(x^{(k)}_i+y_k,\rho^{(k)}_i,n_{F_k}(x_i^{(k)}),w,\delta_2)]\,
\end{align}
 where $G(x,r,v,w,\delta_2)$ is the event that there exists a set $U\subset B\cap \sZ^d$ such that:
 $$\card (U\Delta (n B^-(x,r,v)\cap\sZ ^d))\leq  4\delta_2 \alpha_d r^d n^d$$
 and
 $$|(\partial ^oG_n)\cap n B|\leq (1-w) \alpha_{d-1}r^{d-1}(\nu(v)n ^{d-1}\,.$$

This event depends only on the edges inside $B(x,r,v)$ and is invariant under integer translation. So that,
\begin{align}\label{eqG}
&\underline{\Prb}\left[\begin{array}{c}\exists G_n\in\cG_n,\cL^d((\overline{P}_n\cap B(y_i,r_i+1))\Delta (F_i+y_i))\leq \ep_{F_i},\,1\leq i\leq m,\, \\ |\partial^oG_n|\leq (1-\lambda^2/4)n^{d-1} \cI_p(F,\Omega),\,E(X)=\{(y_1,r_1),\dots,(y_m,r_m)\}\end{array} \right]\nonumber\\
&\hspace{0.5cm} \leq\frac{1}{\theta_p}\sum_{k=1} ^m\sum_{i\in I^{(k)}}\Prb[G(x^{(k)}_i,\rho^{(k)}_i,n_{F_k}(x_i^{(k)}),w,\delta_2)]\,.
\end{align}
 This event is a rare event. Indeed, if this event occurs, we can show that the capacity of the minimal cutset that separates the upper half part of $B(x,r,v)$ (upper half part according to the direction $v$) from the lower half part is abnormally small. To do so, we build from the set $U$ an almost flat cutset in the ball. The fact that $\card (U\Delta B^-(x,r,v))$ is small implies that $\partial_ e U$ is almost flat and is close to $\disc(x,r,v)$. However, this does not prevent the existence of long thin strands that might escape the ball and prevent $U$ from being a cutset in the ball. The idea is to cut these strands by adding edges at a fixed height. We have to choose the appropriate height to ensure that the extra edges we needed to add to cut these strands are not too many, so that we can control their capacity. The new set of edges we create by adding to $U$ these edges will be in a sense a cutset. The last thing to do is then to cover the $\disc(x,r,v)$ by hyperrectangles in order to use the estimate that the flow is abnormally small in a cylinder. This work was done in section 6 in \cite{CerfTheret09infc}. It is possible to choose $\delta_2$ depending on $F_1,\dots, F_m $, $G$ and $w$ such that for all $k\in\{1,\dots,m\}$, there exist positive constants $C^{F_k}_{1,i}$ and $C^{F_k}_{2,i}$ depending on $G$, $d$, $F_k$, $i$ and $w$ so that for all $i\in I^{(k)}$, 
$$\Prb[G(x_i,\rho_i,n_{F_k}(x_i),w,\delta_2)]\leq C^{F_k}_{1,i}\exp(-C^{F_k}_{2,i}n^{d-1})\,.$$
Note that this upper bound is uniform on $y_1,\dots,y_m$ but still depends on $r_1,\dots,r_m$.
Together with inequalities \eqref{eq3} and \eqref{eqG}, we obtain 
\begin{align*}
\underline{\Prb}[\cF]&\leq\underline{\Prb}\left[\begin{array}{c}\exists G_n\in\cG_n, \,\cL^d((\overline{P}_n\cap B(y_i,r_i+1))\Delta (F_i+y_i))\leq \ep_{F_i},\,1\leq i\leq m,\\\,  |\partial^oG_n|\leq (1-\lambda^2/4)n^{d-1} \cI_p(F,\Omega),\\ E(X)=\{(y_1,r_1),\dots,(y_m,r_m)\} \end{array}\right]\\
& \leq\frac{1}{\theta_p}\sum_{k=1}^m\sum_{i\in I^{(k)}} C^{F_k}_{1,i}\exp(-C^{F_k}_{2,i}n^{d-1})\,.
\end{align*}
So there exist positive constants $C_1^{F_1},\dots,C_1^{F_m}$ and $C_2^{F_1},\dots,C_2^{F_m}$ such that
\begin{align}\label{eqG'}
\underline{\Prb}[\cF]\leq \sum_{k=1}^mC_1^{F_k}\exp(-C_2^{F_k}n ^{d-1})\,.
\end{align}
Combining inequalities \eqref{eqE'}, \eqref{finn1}, \eqref{finn2} and \eqref{eqG'}, we obtain for small enough $\delta_2$,
\begin{align}\label{eqE2}
&\Prb\left[\exists G_n\in\cG_n, \,\mu_n \notin \cV(\cW_\xi,\fF_n,\ep)\,|\,0\in\sC_\infty\,\right]\nonumber\\
&\hspace{0.4cm}\leq  b_1\e^{-b_2n^{1-3/2d}}+ b'_1\e^{-b'_2n}+ \sum_{m=1}^M\sum_{y_1,\dots,y_m}\sum_{r_1,\dots,r_m}\sum_{i_1=1}^ {N^{(r_1)}}\cdots \sum_{i_m=1}^ {N^{(r_m)}} \underline{\Prb}[\cF_{i_1,\dots,i_m}]\nonumber\\
&\hspace{0.6cm}+ M3^{M^2}C_dn^{M(d-2)}\Big(Mc_1\e^{-c_2n^{1-3/2d}}+D_1\e^{-D_2n^{(d-1)/2d}}\Big)\nonumber\\
&\hspace{0.4cm}\leq  b_1\e^{-b_2n^{1-3/2d}}+ b'_1\e^{-b'_2n}\nonumber \\
&\hspace{0.6cm}+ \sum_{m=1}^M\sum_{y_1,\dots,y_m}\sum_{r_1,\dots,r_m} \sum_{i_1=1}^ {N^{(r_1)}}\cdots \sum_{i_m=1}^ {N^{(r_m)}}\sum_{k=1} ^m\frac{ C^{F_{i_k}^{(r_k)}}_{1}}{\theta_p}\e^{-C^{F_{i_k}^{(r_k)}}_{2}n^{d-1}}\,\nonumber\\
&\hspace{1.2cm}+ M3^{M^2}C_dn^{M(d-2)}\Big(Mc_1\e^{-c_2n^{1-3/2d}}+D_1\e^{-D_2n^{(d-1)/2d}}\Big)\nonumber\\
&\hspace{0.4cm}\leq  b_1\e^{-b_2n^{1-3/2d}}+ b'_1\e^{-b'_2n}\nonumber \\
&\hspace{0.6cm}+ \sum_{m=1}^M\sum_{y_1,\dots,y_m} 3^{M^2}\max_{r_1,\dots,r_m}\Bigg\{ \sum_{i_1=1}^ {N^{(r_1)}}\cdots \sum_{i_m=1}^ {N^{(r_m)}} \sum_{k=1} ^m \frac{C^{F_{i_k}^{(r_k)}}_{1}}{\theta_p}\e^{-C^{F_{i_k}^{(r_k)}}_{2}n^{d-1}}\Bigg\}\,\nonumber\\
&\hspace{1.2cm}+ M3^{M^2}C_dn^{M(d-2)}\Big(Mc_1\e^{-c_2n^{1-3/2d}}+D_1\e^{-D_2n^{(d-1)/2d}}\Big)\nonumber\\
&\hspace{0.4cm}\leq  b_1\e^{-b_2n^{1-3/2d}}+ b'_1\e^{-b'_2n}\nonumber \\&\hspace{0.6cm}+ C_dn^{M(d-2)}\sum_{m=1}^M3^{M^2}\max_{r_1,\dots,r_m}\Bigg\{ \sum_{i_1=1}^ {N^{(r_1)}}\cdots \sum_{i_m=1}^ {N^{(r_m)}} \sum_{k=1} ^m \frac{C^{F_{i_k}^{(r_k)}}_{1}}{\theta_p}\e^{-C^{F_{i_k}^{(r_k)}}_{2}n^{d-1}}\Bigg\}\,\nonumber\\
&\hspace{0.6cm}+ M3^{M^2}C_dn^{M(d-2)}\Big(Mc_1\e^{-c_2n^{1-3/2d}}+D_1\e^{-D_2n^{(d-1)/2d}}\Big)
\end{align}
where $C_d$ is a constant depending only on the dimension and the maximum is over $r_1,\dots,r_m\in\{1,\dots,3^M\}$. We recall that $M$, $N$ and the number of ways of choosing $r_1,\dots,r_m$ are finite and independent of $n$.
\begin{rk}To obtain inequality \eqref{eqE2}, it is crucial to use a covering of $\sC_\beta$ that is uniform in $y_1,\dots,y_m$.
\end{rk}
Let us assume $\mu_n\notin \cV(\cW,\fF_n,2\ep)$. Let $\nu\in\cW_\xi$, we can write $\nu=\nu_{x+rW_p}$ with $x\in\sR^d$ and $r\in[1-\lambda,1+2\xi]$. We have for all $f\in\fF_n$
\begin{align}\label{distwulff}
|\nu_{x+W_p}(f)-\nu_{x+rW_p}(f)|&\leq \max\left(\cL^d(W_p\setminus(1-\lambda)W_p),\,\cL^d((1+2\xi)W_p\setminus W_p)\right)\nonumber\\
&\leq c(p,d,\xi)
\end{align}
where $c(p,d,\xi)$ is a constant that goes to $0$ when $\xi$ goes to $0$. So that 
$$\sup_{f\in\fF_n}|\nu_{x+W_p}(f)-\nu_{x+rW_p}(f)|\leq c(p,d,\xi)\,.$$
As $\mu_n\notin \cV(\cW,\fF_n,2\ep)$, we have
$$\sup_{f\in\fF_n}|\mu_n(f)-\nu_{x+W_p}(f)|>2\ep\,.$$
So that up to choosing a smaller $\xi$, we have 
\begin{align}\label{cond1} 
c(p,d,\xi)\leq \ep
\end{align}
and so 
\begin{align*}
\underline{\Prb}&[\exists G_n\in\cG_n,\,\forall\nu\in\cW,\,\sup_{f\in\fF_n}|\mu_n(f)-\nu_{x+W_p}(f)|>2\ep]\\
&\hspace{2cm}\leq \underline{\Prb}[\exists G_n\in\cG_n,\,\forall\nu\in\cW_\xi,\,\sup_{f\in\fF_n}|\mu_n(f)-\nu_{x+W_p}(f)|>\ep]\,.
\end{align*}
Finally, using \eqref{eqE2}, there exist positive constants $C_1$ and $C_2$ depending on $\ep$, $u$, $p$ and $d$ such that for all $n\geq 1$,
$$\Prb\left[\exists G_n\in\cG_n,\,\mu_n \notin \cV(\cW,\fF_n,2\ep)\,\big|\,0\in\sC_\infty\right]\leq C_1\e^{-c_2n^{1-3/2d}}\,$$
and the result follows.

To conclude, let us sum up the order in which the constants are chosen. We first choose $\ep>0$. Next, we choose $\xi$ small enough such that it satisfies both inequalities \eqref{condxi} and \eqref{cond1}, and $\delta'$ such that it satisfies inequality \eqref{conddelta'}. Next, we choose $\delta$ such that $\eta(\delta)$ satisfies inequalities \eqref{conddelta1}, \eqref{conddelta11} and \eqref{conddelta2}. We choose $\delta_2$ depending on $w$ (and so on $\ep$) and $G$. The parameter $\delta_2$ has to satisfy some inequalities that we do not detail here, we refer to section 7 in \cite{CerfTheret09infc}. Finally, to each $r$ in $\{1,\dots, 3 ^M\}$, to each $F\in\sC_\beta ^{(r)}$, we choose $\ep_F$ in such a way it satifies  inequality \eqref{condep}.
 \end{proof}
\subsection{Proof of Theorem \ref{LLD}} 

In this section we prove Theorem \ref{LLD}. Thanks to Theorem \ref{prel}, we know that with high probability $\mu_n$ is close to the set $\cW$ and so it is close to the measure of a translate of the Wulff shape. In fact, as $\mu_n$ has its support included in $B(0,n^{d-1})$, the measure $\mu_n$ is close to $\cW_n$, the set of measures defined as:$$\cW_n=\Big\{\nu_{x+W_p},\,x\in B(0,n^{d-1})\,\Big\}\,.$$ The continuous set $\cW_n$ can be approximated by a finite set $\widetilde{\cW}$ containing a polynomial number of measures such that $\mu_n$ is close to $\widetilde{\cW}$ and so is close to at least one measure in $\widetilde{\cW}$.
  Let $\ep>0$ and let $w>0$ be a real number depending on $\ep$ that we will choose later. We first use Lemma \ref{lemball}, to cover $W_p$ by a finite number of balls of small radius such that $W_p$ is almost flat in each ball. Let $\delta_2$ that will be chosen later and let $(B(x_k,\rho_k,v_k))_{k\in J}$ be a family associated to $W_p, \delta_2, \ep$ that satisfies the conditions stated in Lemma \ref{lemball}. We will use this covering for all the translates of the Wulff shape. We set $\ep_W=\min_{k\in J} \alpha_d \rho_k^d\delta_2$. We now cover $\cW_n$ by a polynomial in $n$ number of balls of radius less than $\ep_W$. Let $\xi>0$ small enough such that  $$\forall x,y\in\sR^d,\,\,\|x-y\|_2\leq \xi\,\implies \,\cL^d\left((x+W_p)\Delta(y+W_p)\right)\leq \frac{\ep_W}{4}\,.$$ By construction, $\mu_n$ has its support included in $B(0,n^{d-1})$. We can cover $B(0,n^{d-1})$ by a polynomial in $n$ number of balls of radius $\xi$.  More precisely, there exist $z_1,\dots,z_{M'}\in B(0,n^{d-1})$, such that $M'$ is polynomial in $n$ and 
$$B(0,n^{d-1})\subset \bigcup_{i=1}^{M'}B( z_i,\xi)\,.$$ 
\smash{We set} $$\widetilde{ \cW}=\big\{\,\nu_{z_i+W_p},i=1,\dots,M'\,\big\}\,.$$  Let $\delta>0$ we will choose later. We define $W_p^\delta$ and $W_p^{-\delta}$ as
$$W_p^\delta=\{x\in\sR^d\,:\,d_2(x,W_p)\leq \delta\} \text{ and } W_p^{-\delta}=\{x \in W_p\,:\,d_2(x,\partial W_p)\geq \delta\}\,.$$
Let us define $g$ as
$$g(x) = \left\{
    \begin{array}{ll}
         \min(d_2(x,W_p)/\delta, 1)& \mbox{if} \,\,x\in\sR ^d \setminus W_p\\
        -\min(d_2(x,\partial W_p)/\delta, 1) & \mbox{if}\,\, x\in W_p
    \end{array}
\right.\, .$$
The function $g$ is uniformly continuous and satisfies $\|g\|_\infty\leq 1$. For each $i\in\{1,\dots,M'\}$, we define $g_i$ by $g_i(x)=g(x-z_i)$ for $x\in\sR^d$, and $\fF=\{g_i,\,1\leq i\leq M'\}\cup\{1\}$. The set $\fF$ is a set made of translates of $g$ and the constant function equal to $1$. If the measure $\mu_n$ is in the local weak neighborhood $\cV(\cW,\fF,\frac{\ep_W}{4})$, then there exists $\nu_{x+W_p}$ in $\cV(\cW_n,\fF,\frac{\ep_W}{4})$ such that
$$\sup_{f\in\fF}|\nu_{x+W_p}(f)-\mu_n(f)|\leq \frac{\ep_W}{4}\,.$$
Moreover there exists an $i\in\{1,\dots,M'\}$ such that $x\in B(z_i,\xi)$ and so $$\sup_{f\in\fF}|\nu_{x+W_p}(f)-\nu_{z_i+W_p}(f)|\leq \cL^d\left((x+W_p)\Delta(z_i+W_p)\right)\leq \frac{\ep_W}{4}$$ and also
$$\mu_n\in\cV\big(\widetilde{\cW},\fF,\ep_W/2\big)\,.$$
Let us choose $r>0$ large enough so that the ball $B(0,r-2d)$ contains $W_p$. For $x\in\sR^d$, we define $\lfloor x\rfloor$ to be the closest point to $x$ in $\sZ^d$ for the Euclidean distance. For any $i\in\{1,\dots,M'\}$, we have $$W+z_i\subset B(\lfloor z_i\rfloor,r)\,.$$ Let us define the function $u$ such that for all $\iota >0$,
$$u(\iota)=\min \left(\sup\big\{\,\delta>0, \,\forall x,y\in\sR^d, \,\|x-y\|_2\leq \delta \implies |g(x)-g(y)|\leq \iota\,\big\},1\right)\,.$$ As the function $g$ is uniformly continuous, the function $u$ is positive. Moreover, as $\fF$ is made of translated of $g$ and the constant function equal to $1$, it is clear that this set satisfies the condition stated in Proposition \ref{propcontiguity} associated with the function $u$. Using Proposition \ref{propcontiguity} with the function $u$, there exist positive constants $C_1$, $C_2$ depending only on $r$, $u$, $p$ and $\ep_W$ such that for all $i\in\{1,\dots,M'\}$
\begin{align}\label{eqcontig}
\underline{\Prb}\left[\max_{G_n\in\cG_n}\,\sup_{f\in\fF}|\mu_n(f\ind_{B(\lfloor z_i\rfloor,r)})-\nu_n(f\ind_{B(\lfloor z_i\rfloor,r)})|>\ep_W/4\right]\leq C_1\e^{-c_2n^{1-3/2d}}\,.
\end{align}
The point of choosing such a set $\fF$ is that we can deduce from the fact that the quantity $\sup_{f\in\fF}|\mu_n(f)-\nu_{W+z_i}(f)|$ is small that the associated symmetric difference $\cL^d((P_n\cap B(\lfloor z_i\rfloor,r))\Delta(z_i+W_p))$ is small.
Indeed, we have
\begin{align}\label{diffsym}
&\cL^d((P_n\cap B(\lfloor z_i\rfloor,r))\Delta(z_i+W_p))\nonumber\\
&\hspace{0.5cm}=\int _{(P_n\cap B(\lfloor z_i\rfloor,r))\setminus (z_i+W_p)}1d\cL^d(x)+ \int _{(z_i+W_p)\setminus P_n}1d\cL^d(x)\nonumber\\
&\hspace{0.5cm}\leq \int _{(P_n\cap B(\lfloor z_i\rfloor,r))\setminus (z_i+W_p)}g_i(x)d\cL^d(x)- \int _{(z_i+W_p)\setminus P_n}g_i(x)d\cL^d(x)\nonumber\\
&\hspace{1cm}+\cL^d(W^\delta_p\setminus W^{-\delta}_p)\nonumber\\
&\hspace{0.5cm}=|\nu_n(g_i\ind_{B(\lfloor z_i\rfloor,r))})-\nu_{W+z_i}(g_i\ind_{B(\lfloor z_i\rfloor,r))})|+\cL^d(W^\delta_p\setminus W^{-\delta}_p)\nonumber\\
&\hspace{0.5cm} \leq \sup_{f\in\fF}|\mu_n(f\ind_{B(\lfloor z_i\rfloor,r)})-\nu_n(f\ind_{B(\lfloor z_i\rfloor,r)})|\nonumber\\
&\hspace{1cm}+ \sup_{f\in\fF}|\mu_n(f\ind_{B(\lfloor z_i\rfloor,r)})-\nu_{W+z_i}(f\ind_{B(\lfloor z_i\rfloor,r)})| +\cL^d(W^\delta_p\setminus W^{-\delta}_p)\,.
\end{align}
So we choose $\delta$ small enough so that 
\begin{align}\label{choixdelta}
\cL^d(W^\delta_p\setminus W^{-\delta}_p) \leq \frac{\ep_W}{4}\,.
\end{align}
Moreover, we have
\begin{align}\label{f1}
\underline{\Prb}&\left[\exists G_n\in\cG_n,\, |\partial ^o G_n|\leq (1-w)\cI_p(W_p)n ^{d-1},\,\mu_n\in\cV\big(\widetilde{\cW},\fF,\ep_W/2)\right]\nonumber\\
&\hspace{1cm}\leq\sum_{i=1}^ {M'}\Prb\left[\begin{array}{c} \exists G_n\in\cG_n,\,|\partial ^o G_n|\leq (1-w)\cI_p(W_p)n ^{d-1},\\\sup_{f\in\fF} |\mu_n(f)-\nu_{W+z_i}(f)|\leq \ep_W/2\end{array}\,\Big |\,0\in\sC_\infty\right]\,.
\end{align}
Using inequalities \eqref{eqcontig}, \eqref{diffsym} and \eqref{choixdelta}, we obtain 
\begin{align}\label{f12}
&\Prb\left[\exists G_n\in\cG_n,\begin{array}{c} |\partial ^o G_n|\leq (1-w)\cI_p(W_p)n ^{d-1},\\ \sup_{f\in\fF}|\mu_n(f)-\nu_{W+z_i}(f)|\leq \ep_W/2\end{array}\,\Big |\,0\in\sC_\infty\right]\nonumber\\
&\hspace{0.1cm}\leq \underline{\Prb}\left[\exists G_n\in\cG_n,\begin{array}{c} |\partial ^o G_n|\leq (1-w)\cI_p(W_p)n ^{d-1},\\ \sup_{f\in\fF}|\mu_n(f\ind_{B(\lfloor z_i\rfloor,r)})-\nu_{W+z_i}(f\ind_{B(\lfloor z_i\rfloor,r)})|\leq \ep_W/2\end{array}\right]\nonumber\\
&\hspace{0.1cm}\leq \underline{\Prb}\left[\exists G_n\in\cG_n,\begin{array}{c} |\partial ^o G_n|\leq (1-w)\cI_p(W_p)n ^{d-1},\\ \cL^d((P_n\cap B(\lfloor z_i\rfloor,r))\Delta(z_i+W_p))\leq \ep_W\end{array}\right] + C_1\e^{-c_2n^{1-3/2d}}\,.
\end{align}
Finally, we proceed as in inequality \eqref{eqG} in the proof of Theorem \ref{prel}:
\begin{align}\label{f13}
 &\Prb\left[\exists G_n\in\cG_n,\begin{array}{c} |\partial ^o G_n|\leq (1-w)\cI_p(W_p)n ^{d-1},\\ \cL^d((P_n\cap B(\lfloor z_i\rfloor,r))\Delta(z_i+W_p))\leq \ep_W\end{array}\,\Big |\,0\in\sC_\infty\right]\nonumber\\
 &\hspace{4cm}\leq \frac{1}{\theta_p} \sum_{k\in J} \Prb\left[G(z_i+x_k,\rho_k,n_{W_p}(x_k),w,\delta_2)\right]\,.
\end{align}
It is possible to choose $\delta_2$ depending on $W$, $G$ and $w$ (see again section 6 in \cite{CerfTheret09infc}) such that there exist positive constants $C_{1,k}$ and $C_{2,k}$ depending on $G$, $d$, $W$, $k$ and $w$ so that for all $k\in J$,
$$\Prb[G(x_k,\rho_k,n_{W_p}(x_k),w,\delta_2)]\leq C_{1,k}\exp(-C_{2,k}n^{d-1})\,.$$
So combining inequalities \eqref{f1}, \eqref{f12} and \eqref{f13}, we obtain
\begin{align}\label{fin1'}
\underline{\Prb}&\left[\exists G_n\in\cG_n,\,\, |\partial ^o G_n|\leq (1-w)\cI_p(W_p)n ^{d-1},\,\,\mu_n\in\cV\big(\widetilde{\cW},\fF,\ep_W/2)\right]\nonumber\\
&\hspace{2.5cm}\leq M'\Big(C_1\e^{-c_2n^{1-3/2d}}+ \frac{1}{\theta_p}\sum_{k\in J}C_{1,k}\exp(-C_{2,k}n^{d-1})\Big)\,.
\end{align}
Moreover, we have
\begin{align}\label{f2}
\underline{\Prb}&\left[\exists G_n\in\cG_n, \frac{|G_n|}{n^d}\geq  (1+w)\theta_p \cL^d(W_p),\,\mu_n\in\cV\big(\widetilde{\cW},\fF,\ep_W/2)\right]\nonumber\\
&\hspace{2cm}\leq \sum_{i=1}^ {M'}\Prb\left[\begin{array}{c} \exists G_n\in\cG_n,\frac{|G_n|}{n^d}\geq  (1+w)\theta_p \cL^d(W_p),\\ \, |\mu_n(1)-\nu_{W+z_i}(1)|\leq \ep_W/2\end{array}\,\Big |\,0\in\sC_\infty\right]\nonumber\\
& \hspace{2cm}\leq \sum_{i=1}^ {M'}\Prb\left[\begin{array}{c} \exists G_n\in\cG_n,\frac{|G_n|}{n^d}\geq  (1+w)\theta_p \cL^d(W_p),\\ \, \left|\frac{|G_n|}{n^d}-\theta_p\cL^d(W_p)\right|\leq \ep_W/2\end{array}\,\Big |\,0\in\sC_\infty\right]
\end{align}
where we recall that $\theta_p\cL ^d(W_p)=1$, so up to choosing a smaller $\ep_W$, we assume that $\ep_W\leq 2w$ so that the probability in the sum is equal to $0$. 
Finally, combining inequalities \eqref{fin1'} and \eqref{f2}, we obtain
\begin{align}\label{bla2}
\Prb&\left[n\varphi_n \geq \frac{1-w}{1+w}\frac{\cI_p(W_p)}{\theta_p\cL ^d(W_p)}\,\Big |\,0\in\sC_\infty\right]\leq\underline{\Prb}\left[\exists G_n\in\cG_n,\,\mu_n\notin\cV\big(\widetilde{\cW},\fF,\ep_W/2\big)\right]\nonumber\\
&\hspace{2.5cm}+M'\Big(C_1\exp(-C_2n)+ \frac{1}{\theta_p}\sum_{k\in J}C_{1,k}\exp(-C_{2,k}n^{d-1})\Big)\,.
\end{align}
Thanks to Theorem \ref{prel}, there exist positive constants $C'_1$, $C'_2$, depending on $p$, $u$, $\ep_W$ and $d$ such that 
$$\Prb\left[\exists G_n\in\cG_n,\,\mu_n\notin\cV\big(\widetilde{\cW},\fF,\ep_W/2\big)\,\Big |\,0\in\sC_\infty\right]\leq C'_1\exp(-C'_2n^{1-3/2d})\,.$$
By choosing $w$ small enough, we obtain
\begin{align*}
\Prb&\left[n\varphi_n \geq (1-\ep)\frac{\cI_p(W_p)}{\theta_p\cL ^d(W_p)}\,\Big |\,0\in\sC_\infty\right]\\
&\hspace{0.3cm}\leq C'_1\exp(-C'_2 n^{1-3/2d})+M'\Big(C_1\exp(-C_2n)+\sum_{k\in J}C_{1,k}\exp(-C_{2,k}n^{d-1})\Big)\,.
\end{align*}
As $M'$ is polynomial in $n$, the result follows.

\subsection{Proof of Theorem \ref{formeGn}}
Let $\ep>0$. As in the proof of Theorem \ref{LLD}, there exists an integer $M'$ that is polynomial in $n$ and $z_1,\dots,z_{M'}$ points of $B(0,n^{d-1})$ such that for any finite set $\fF$ of continuous functions of infinite norm at most $1$, if $\mu_n\in\cV(\cW,\fF,\ep)$ then  \smash{$\mu_n\in\cV(\widetilde{\cW},\fF,2\ep)$ where $\widetilde{\cW}=\big\{\,\nu_{z_i+W_p},i=1,\dots,M'\,\big\}\,.$ } %Let $\xi>0$ small enough such that for all $x,y$ such that  $$\|x-y\|_2\leq \xi\,\implies \,\cL^d\left((x+W_p)\Delta(y+W_p)\right)\leq \frac{\ep}{2}\,.$$ By construction, $\mu_n$ has its support included in $B(0,n^{d-1})$. We can cover $B(0,n^{d-1})$ by a polynomial in $n$ number of balls of radius $\xi$.  More precisely, there exist $z_1,\dots,z_{M'}\in B(0,n^{d-1})$, such that $M'$ is polynomial in $n$ and 
%$$B(0,n^{d-1})\subset \bigcup_{i=1}^{M'}B( z_i,\xi)\,.$$ \smash{We set} $$\widetilde{ \cW}=\big\{\,\nu_{z_i+W_p},\,i=1,\dots,M'\,\big\}\,.$$ 
 Let $\delta>0$ we will choose later. Let us define $f$ and $g$ as
$$f(x) =  \min(d_2(x,\sR^d\setminus  W_p^\delta)/\delta, 1),\,\mbox{for $x\in\sR^d$}$$
and
$$g(x) = \min(d_2(x,W_p)/\delta, 1),\,\, \mbox{for $x\in\sR^d$}\,.$$
The functions $f$ and $g$ are uniformly continuous and satisfy $\|f\|_\infty\leq 1$ and $\|g\|_\infty\leq 1$. For each $i\in\{1,\dots,M'\}$, we define $f_i$ by $f_i(x)=f(x-z_i)$  and $g_i$ by $g_i(x)=f(x-z_i)$ for $x\in\sR^d$. We define $$\fF=\{f_i,\,1\leq i\leq M'\}\cup\{g_i,\,1\leq i\leq M'\}\,.$$ 
Let $G_n\in \cG_n$. Let $i\in\{1,\dots,M'\}$. We have
\begin{align}\label{eq25.1}
|G_n\Delta ((n(W_p+z_i))\cap \sC_\infty)|= |G_n\setminus  n(W_p+z_i)|+|(n(W_p+z_i)\cap \sC_\infty)\setminus G_n|\,.
\end{align}
Using a renormalization argument as in the proof of Theorem \ref{ULD}, there exist positive constants $C_1$ and $C_2$ depending on $p$, $\ep$ and $d$ such that for all $i\in\{1,\dots,M'\}$,
$$\Prb\left[\left|\frac{|(n(W+z_i))\cap\sC_\infty|}{n^d}-\theta_p\cL^d(W_p)\right|  \geq \ep \,\Big|\,0\in\sC_\infty\right]\leq C_1\exp(-C_2n)\,.$$
As $G_n\cap (n(W_p+z_i))\subset (n(W_p+z_i))\cap \sC_\infty$, we have with probability at least $1-C_1\exp(-C_2n)$,
\begin{align*}
|((n&(W_p+z_i))\cap \sC_\infty)\setminus G_n|\\
&= |(n(W_p+z_i))\cap \sC_\infty|- |G_n\cap (n(W_p+z_i))|\\
&\leq \theta_p\cL^d(W_p)n^d+n^d\ep-n^d\mu_n(f_i)+|n((W^\delta_p +z_i)\setminus (W_p+z_i))\cap\sZ^d|\,.
\end{align*}
We can find a constant $c(\delta)$ depending only on $\delta$, $p$ and $d$, such that $c(\delta)$ goes to $0$ when $\delta$ goes to $0$ and for all $z\in\sR^d$
$$|n((W^\delta_p +z)\setminus (W_p+z))\cap\sZ^d|\leq c(\delta)n^d\,,$$
so that,
\begin{align}\label{eq25.2}
|((n(W_p+z_i))\cap \sC_\infty)\setminus G_n|&\leq n^d|\nu_{W_p+z_i}(f_i)-\mu_n(f_i)|+(\ep +c(\delta))n ^d\nonumber\\
&\leq n^d\sup_{h\in\fF}|\nu_{W_p+z_i}(h)-\mu_n(h)|+(\ep +c(\delta))n ^d\,.
\end{align}
Moreover, noticing that $\nu_{W_p+z_i}(g_i)=0$, we obtain
\begin{align}\label{eq25.3}
|G_n\setminus n(W+z_i)|&\leq n ^d \mu_n(g_i)+ |n((W^\delta_p +z_i)\setminus (W_p+z_i))\cap\sZ^d|\nonumber\\
&\leq n^d|\mu_n(g_i)-\nu_{W_p+z_i}(g_i)|+n^d c(\delta)\nonumber\\
&\leq n^d\sup_{h\in\fF}|\nu_{W_p+z_i}(h)-\mu_n(h)|+ n^dc(\delta)\,.
\end{align}
Combining inequalities \eqref{eq25.1}, \eqref{eq25.2} and \eqref{eq25.3}, with high probability, we have
\begin{align*}
&\inf_{z\in\sR^d}\frac{1}{n^d}|G_n\Delta ((n(W_p+z))\cap \sC_\infty)|\\
&\hspace{1cm}\leq \min_{1\leq i\leq M'}\frac{1}{n^d}|G_n\Delta ((n(W_p+z_i))\cap \sC_\infty)|\\
&\hspace{1cm}\leq \min_{\nu \in \widetilde{\cW}}\left\{\sup_{h\in\fF}|\nu(h)-\mu_n(h)| + \sup_{h\in\fF}|\nu(h)-\mu_n(h)|\right\}+\ep +2c(\delta)\\
&\hspace{1cm}\leq 2\min_{\nu \in \widetilde{\cW}}\sup_{h\in\fF}|\nu(h)-\mu_n(h)| +\ep +2c(\delta)\,.
\end{align*}
Let us define for any $\iota>0$,
$$u_g(\iota)=\min \left(\sup\big\{\,\delta>0, \,\forall x,y\in\sR^d, \,\|x-y\|_2\leq \delta \implies |g(x)-g(y)|\leq \iota\,\big\},1\right)\,,$$
$$u_f(\iota)=\min \left(\sup\big\{\,\delta>0, \,\forall x,y\in\sR^d, \,\|x-y\|_2\leq \delta \implies |f(x)-f(y)|\leq \iota\,\big\},1\right)\,$$
and $u=\min(u_f,u_g)$. This function is positive because the function $f$ and $g$ are uniformly continuous.
It is easy to check that $\fF$ satisfies the condition required in Theorem \ref{prel} associated with the function $u$. Thus, there exist positive constants $c_1$ and $c_2$ depending on $p$, $u$, $\ep$ and $d$ such that
$$\Prb \left[\exists G_n\in\cG_n, \,\inf_{\nu\in\cW}\sup_{h\in\fF}|\nu(h)-\mu_n(h)|\geq \ep \,\Big|\,0\in\sC_\infty \right]\leq c_1\e^{-c_2n^{1-3/2d}} \,$$
and so
$$\Prb \left[\exists G_n\in\cG_n, \,\min_{\nu\in\widetilde{\cW}}\sup_{h\in\fF}|\nu(h)-\mu_n(h)|\geq 2\ep \,\Big|\,0\in\sC_\infty \right]\leq c_1\e^{-c_2n^{1-3/2d}} \,.$$
We now choose $\delta$ small enough such that $c(\delta)\leq \ep$ so that 
\begin{align*}
&\Prb\left[\exists G_n \in\cG_n,\, \inf_{z\in\sR^d}\frac{1}{n^d}|G_n\Delta ((n(W_p+z))\cap \sC_\infty)|\geq 7\ep\,\Big|\,0\in\sC_\infty\right]\\
&\hspace{4cm}\leq c_1\e^{-c_2n^{1-3/2d}}+M'C_1\exp(-C_2n)\,.
\end{align*}
As $M'$ is polynomial in $n$, this yields the result.

\paragraph{Acknowledgments} I wish to express my gratitude to Rapha{\"e}l Cerf for showing me this problem and for giving me the opportunity to work with him for an internship. I thank him for our fruitful discussions and his patience. This research was partially supported by the ANR project PPPP (ANR-16-CE40-0016).
\bibliographystyle{plain}
\bibliography{biblio}

\end{document}